\documentclass[12pt,reqno]{amsart}

\usepackage{color}

\newcommand{\QQ}{\mathcal{Q}}

\newcommand{\T}{{\mathbf T}^m}

\newcommand{\szego}{Szeg\"o }

\newcommand{\kahler}{K\"ahler }

\newcommand{\PP}{{\mathbb P}}
\newcommand{\R}{{\mathbb R}}
\newcommand{\C}{{\mathbb C}}

\newcommand{\Z}{{\mathbb Z}}

\newcommand{\N}{{\mathbb N}}
\newcommand{\CP}{\C\PP}

\newcommand{\dbar}{\bar\partial}
\newcommand{\ddbar}{\partial\dbar}

\newcommand{\half}{{\frac{1}{2}}}

\renewcommand{\phi}{\varphi}

\newcommand{\bcal}{\mathcal{B}}
\newcommand{\ccal}{\mathcal{C}}
\newcommand{\dcal}{\mathcal{D}}

\newcommand{\fcal}{\mathcal{F}}
\newcommand{\gcal}{\mathcal{G}}
\newcommand{\hcal}{\mathcal{H}}
\newcommand{\ical}{\mathcal{I}}

\newcommand{\lcal}{\mathcal{L}}

\newcommand{\pcal}{\mathcal{P}}
\newcommand{\rcal}{\mathcal{R}}
\newcommand{\ocal}{\mathcal{O}}

\newcommand{\tcal}{\mathcal{T}}

\def\us{{\underline s}}

\def    \half   {{\frac{1}{2}}}

\def    \Z  {{\mathbb Z}}
\def    \R  {{\mathbb R}}
\def    \C  {{\mathbb C}}

 \def   \half   {{\frac{1}{2}}}

\newtheorem{theo}{{\sc Theorem}}[section]
\newtheorem{cor}[theo]{{\sc Corollary}}

\newtheorem{lem}[theo]{{\sc Lemma}}
\newtheorem{prop}[theo]{{\sc Proposition}}

\newenvironment{rem}{\medskip\noindent{\it Remark:\/} }{\medskip}
\newenvironment{defin}{\medskip\noindent{\it Definition:\/} }{\medskip}

\title[Bergman metrics and geodesics ]
{Bergman metrics and geodesics in the space of K\"ahler metrics on
toric varieties}

\author{Jian Song}
\author{Steve Zelditch }

\address{Department of Mathematics, Rutgers University, New Brunswick, NJ 08854, USA} \email{jiansong@math.rutgers.edu}

\address{Department of Mathematics, Johns Hopkins University, Baltimore,
MD 21218, USA} \email{zelditch@math.jhu.edu}

\thanks{Research partially supported by
 National Science Foundation grants DMS-06-04805 and
 DMS-06-03850.}

\date{\today}

\begin{document}

\maketitle

\begin{abstract} A guiding principle in \kahler geometry is that the
infinite dimensional symmetric space  $\hcal$ of K\"ahler metrics
in a fixed K\"ahler class on a polarized projective K\"ahler
manifold $M$ should be well approximated by finite dimensional
submanifolds $\bcal_k \subset \hcal$ of Bergman metrics of height
$k$ (Yau, Tian, Donaldson). The Bergman metric spaces are
symmetric spaces of type $G_{\C}/G$ where $G = U(d_k + 1)$ for
certain $d_k$. This article establishes the basic estimates for
Bergman approximations for geometric families of toric \kahler
manifolds.

The approximation results are applied to the endpoint problem for
geodesics of $\hcal$, which are solutions of a homogeneous complex
Monge-Amp\`ere equation in $A \times X$, where $A \subset \C$ is
an annulus. Donaldson, Arezzo-Tian and Phong-Sturm raised the
question whether $\hcal$- geodesics with fixed endpoints  can be
approximated by geodesics of  $\bcal_k$. Phong-Sturm proved weak
$C^0$-convergence of Bergman to Monge-Amp\`ere geodesics on a
general \kahler manifold. Our approximation results show that one
has $C^2(A \times X)$ convergence  in the case of toric K\"ahler
metrics, extending our earlier result on $\CP^1$.

In subsequent papers, the techniques of this article are applied
to approximations for  harmonic maps into $\hcal$,  to test
configuration geodesic rays and to the smooth initial value
problem.

\end{abstract}

\tableofcontents

\section{Introduction}

This is the first in a series of articles  on the Riemannian
geometry of the space \begin{equation} \label{HCALDEF} \hcal\ = \
\{\phi\in C^{\infty} (M) : \omega_\phi\ = \ \omega_0+ dd^c \phi>0\
\}
\end{equation}of  \kahler metrics in the class $[\omega_0]$  of a polarized
projective \kahler manifold $(M, \omega_0, L)$,   equipped with
the Riemannian metric $g_{\hcal}$  of Mabuchi-Semmes-Donaldson
\cite{M,S2,D2},
\begin{equation} \label{metric} ||\psi||^2_{g_{\hcal}, \phi}\ = \ \int_M |\psi|^2\
\frac{\omega_{\phi}^m}{m!},\ \, \;\; {\rm ~ where~} \phi \in \hcal
{\rm~ and ~} \psi \in T_{\phi} \hcal \simeq C^{\infty}(M).
\end{equation}
Here, $L \to M$ is an ample line bundle with $c_1(L) = [
\omega_0]$. Formally, $(\hcal, g_{\hcal})$ is an infinite
dimensional non-positively curved  symmetric space
 of the type
$G_{\C}/G$, where $G = SDiff_{\omega_0}(M)$ is the group of
Hamiltonian symplectic diffeomorphisms of $(M, \omega_0)$. This
statement is only formal since  $G$ does not possess a
complexification and $\hcal$ is an  incomplete, infinite
dimensional space. An attractive approach to the infinite
dimensional geometry is to approximate it by  a sequence of finite
dimensional submanifolds $\bcal_k \subset \hcal$ of so-called
Bergman (or Fubini-Study) metrics.  The space $\bcal_k$ of Bergman
metrics may be identified with the finite dimensional symmetric
space  $GL(d_k + 1, \C)/ U(d_k + 1)$ where $d_k$ is a certain
dimension. Thus, $\bcal_k$ is equipped with a finite dimensional
symmetric space metric $g_{\bcal_k}$, which is not the same as the
submanifold Riemannian metric induced on it by $g_{\hcal}$. The
purpose of the series is to show that much of the symmetric space
geometry of $(\bcal_k, g_{\bcal_k})$ tends to the infinite
dimensional symmetric space geometry of $(\hcal, g_{\hcal})$ as $k
\to \infty$.

To put the problem and results in perspective, we recall that  at
the level of individual metrics $\omega \in \hcal$, there exists a
well-developed approximation theory: Given $\omega$, one can
define  a canonical sequence of Bergman metrics $\omega_k \in
\bcal_k$ which approximates $\omega$ in the $C^{\infty}$ topology
(see (\ref{TK})), in  much the same way that smooth functions can
be approximated by Bernstein polynomials (Yau \cite{Y} and Tian
\cite{T}; see also \cite{C,Z,Z2}).  The approximation theory is
based on microlocal analysis in the complex domain, specifically
Bergman kernel asymptotics on and off the diagonal
\cite{BSj,C,Z,D,PS3}.
 As will be shown in \cite{RZ3}, one  may use the same
methods to prove that the geometry of $(\bcal_k, g_{\bcal_k})$
tends to the geometry of $(\hcal, g_{\hcal})$ at the infinitesimal
level: e.g. that   the Riemann metric, connection and curvature
tensor of $\bcal_k$ tend to the Riemann metric, connection and
curvature of $\hcal$. But our principal aim in this series
 is to extend the approximation from
pointwise or infinitesimal objects  to more global aspects of the
geometry, such as   such as $\bcal_k$ -geodesics  or harmonic maps
to $(\bcal_k, g_{\bcal_k})$. These more global approximation
problems are much more difficult than the infinitesimal ones. The
obstacles are analogous to those involved in complexifying
$SDiff_{\omega_0}(M)$. We will explain this comparison in more
detail in \S \ref{BAC} at the end of the introduction.

This article is concerned with the   approximation of
$g_{\hcal}$-geodesic segments $\omega_t$ in $\hcal$  with fixed
endpoints by $g_{\bcal_k}$-geodesic segments  in $\bcal_k$.
 As recalled in \S \ref{BACKGROUND}, the geodesic equation for the
  \kahler potentials $\phi_t$ of $\omega_t$  is a  complex homogeneous Monge-Amp\`ere equation. Little
  is known about the solutions of the Dirichlet problem  at present beyond the regularity result that $\phi_t \in C^{1,\alpha}([0, T] \times
  M)$ for all $\alpha < 1$
  if the endpoint metrics are smooth (see X. Chen \cite{Ch} and Chen-Tian \cite{CT} for results and background).
  It is therefore
natural to study the  approximation of Monge-Amp\`ere
$g_{\hcal}$-geodesics $\phi_t$ by the much simpler
$g_{\bcal_k}$-geodesics
  $\phi_k(t,z)$, which are defined by one parameter
subgroups of $GL(d_k + 1, \C)$ (see (\ref{toricphik})). The
 problem of approximating $\hcal$-geodesic segments
between two smooth endpoints by $\bcal_k$-geodesic segments    was
raised by Donaldson \cite{D}, Arezzo-Tian \cite{AT} and
Phong-Sturm \cite{PS} and was studied in depth by Phong-Sturm in
\cite{PS,PS1}. They proved in \cite{PS} that $\phi_k (t, z) \to
\phi_t$ in a weak $C^0$ sense on $[0, 1] \times M$ (see
(\ref{PSRESULT})); a $C^0$ result with a remainder estimate was
later proved by Berndtsson \cite{B} for a somewhat different
approximation.

 In this article, we study the $g_{\bcal_k}$-approximation of
 $g_{\hcal}$-geodesics in the case of a polarized projective toric
\kahler manifold. Our main  result is that a $g_{\hcal}$ geodesic
segment  of toric \kahler metrics  with fixed endpoints is
approximated in $C^2$ by a sequence $\phi_k(t, z)$ of toric
$g_{\bcal_k}$- geodesic segments. More precisely, for any $T \in
\R_+$,  $\phi_k(t, z) \to \phi_t(z)$ in $C^2([0, T] \times M)$,
generalizing the results of \cite{SoZ} in the case of $\CP^1$. It
is natural to study convergence of two (space-time) derivatives
since the \kahler metric $\omega_{\phi} = \omega_0 + dd^c \phi$
involves two derivatives. In the course of the  proof, we
introduce methods which have  many other applications to  global
approximation problems on toric \kahler manifolds, and which
should also  have applications to non-toric \kahler manifolds.

Here,   as  in \cite{SoZ2,RZ,RZ2}, we restrict to the toric
setting  because, at this stage, it is possible to obtain much
stronger results than for general \kahler manifolds and because it
is  one of the few settings where we can see clearly what is
involved in the classical limit as $k \to \infty$. The simplifying
feature of toric \kahler manifolds is that they are completely
integrable on both the classical and quantum level. In Riemannian
terms,  the submanifolds of toric metrics of $\hcal$ and $\bcal_k$
form totally geodesic flats. Hence in the toric case, the geodesic
equation along the flat  is linearized by the Legendre transform,
with the consequence that there  exists an explicit formula for
the Monge-Amp\`ere geodesic $\phi_t$ between two smooth endpoint
metrics. In particular, the explicit formula shows that geodesics
between smooth endpoints are smooth. We use this explicit solution
throughout the article, starting from (\ref{REWRITE}). Thus, in
the toric case we only need to prove $C^2$-convergence of the
Bergman approximation. An analogous result   on a general \kahler
manifold would require an improvement on  the known regularity
results on Monge-Amp\`ere geodesics in addition to a convergence
result. We refer to \cite{CT} for the state of the art on the
regularity theory.

\subsection{\label{BACKGROUND} Background}

To state our results, we need some notation and background.   Let
$L \to M^m$ be an ample holomorphic line bundle over a compact
complex manifold of dimension $m$.  Let $\omega_0 \in H^{(1,1)}(M,
\Z)$ denote an integral \kahler form. Fixing a reference hermitian
metric $h_0$ on $L$, we may write other hermitian metrics on $L$
as $h_{\phi} = e^{-\phi} h_0$, and then the  space of hermitian
metrics $h$ on $L$ with curvature $(1,1)$-forms $\omega_h$ in the
class of $\omega_0$ may (by the $\ddbar$ lemma) be identified with
the space $\hcal$  of relative \kahler potentials (\ref{HCALDEF}).
We may then identify the  tangent space $T_\phi\hcal$ at
$\phi\in\hcal$  with $C^{\infty}(M)$. Following \cite{M, S2, D},
we  define the Riemannian metric (\ref{metric})  on $\hcal$. With
this Riemannian metric, $\hcal$ is  formally
 an infinite
dimensional non-positively  curved symmetric space.

The space $\bcal_k$ of Bergman (or Fubini-Study) metrics of height
$k$ is defined as follows: Let
  $H^0(M, L^k)$  denote the space of  holomorphic
sections of the $k$th power $L^k \to M$ of $L$ and let $d_k + 1 =
\dim H^0(M, L^k)$. We let $\bcal H^0(M, L^k)$ denote the manifold
of all bases $\underline{s} = \{s_0, \dots, s_{d_k}\}$ of $H^0(M,
L^k)$. Given a basis, we define the  Kodaira embedding
\begin{equation} \label{KODEMB} \iota_{\underline{s}}: M \to \CP^{d_k},\;\;z \to [s_0(z),
\dots, s_{d_k}(z)]. \end{equation}
 We then define a Bergman metric (or equivalently, Fubini-Study)
 metric of height $k$ to be a metric of the form
 \begin{equation} \label{FSDEFa}  h_{\underline{s}} := (\iota_{\underline{s}}^*
h_{FS})^{1/k} = \frac{h_0}{\left( \sum_{j = 0}^{d_k}
|s_j(z)|^2_{h_0^k} \right)^{1/k}}, \end{equation} where $h_{FS}$
is the Fubini-Study Hermitian metric on $\ocal(1) \to \CP^{d_k}$.
We then define
\begin{equation}
 \label{BERGMETDEF} \; {\mathcal
B}_k = \{h_{\underline{s}}, \;\; \underline{s} \in \bcal H^0(M,
L^k) \}.
\end{equation}
We use the same notation for the associated space of  potentials
$\phi$ such that $h_{\underline{s}} = e^{- \phi} h_0$ and for the
associated \kahler metrics $\omega_{\phi}$.
 We observe that with a choice of basis of $H^0(M, L^k)$ we may
identify  ${\mathcal B}_k$ with the symmetric space $ GL(d_k + 1,
\C) / U(d_k + 1)$ since $GL(d_k + 1, \C)$ acts transitively on the
set of bases, while $\iota_{\underline{s}}^* h_{FS}$ is unchanged
if we replace the basis $\underline{s}$ by a unitary change of
basis.

Several further identifications are important. The first is that
${\mathcal B}_k$ may be identified with the space $\ical_k$ of
Hermitian inner products  on $H^0(M, L^k)$, the correspondence
being that a basis is identified with an inner product for which
the basis is Hermitian orthonormal. As in \cite{D,D4}, we define
maps
$$Hilb_k: \hcal \to \ical_k, $$
by the rule that a Hermitian metric  $h \in \hcal$ induces the
inner products on $H^0(M, L^k)$,
\begin{equation} \label{HILBDEF} ||s||^2_{Hilb_k(h)} = R \int_M |s(z)|_{h^k}^2 dV_h,
\end{equation}
where $dV_h = \frac{\omega_h^m}{m!}$, and where $R = \frac{d_k +
1}{Vol(M, dV_h)}.$ Also, $h^k$ denotes the induced metric on
$L^k$. Further, we define the identifications
$$FS_k: \ical_k \simeq {\mathcal B}_k $$
as follows: an inner product $G = \langle ~,~ \rangle$ on $H^0(M,
L^k)$ determines a $G$-orthonormal  basis  $\underline{s} =
\underline{s}_G$ of $H^0(M, L^k)$ and an associated Kodaira
embedding (\ref{KODEMB}) and Bergman metric (\ref{FSDEFa}). Thus,
\begin{equation} \label{FSDEF} FS_k(G) = h_{\underline{s}_G}. \end{equation} The right side
is independent of the choice of $h_0$ and the choice of
orthonormal basis.  As observed in \cite{D,PS}, $FS_k(G)$ is
characterized by the fact that for any $G$-orthonormal basis
$\{s_j\}$ of $H^0(M, L^k)$, we have
\begin{equation} \label{CONSTANT} \sum_{j = 0}^{d_k}
|s_j(z)|_{FS_k(G)}^2 \equiv 1, \;\; (\forall z \in M).
\end{equation}

Metrics in  $\bcal_k$ are defined by an algebro-geometric
construction. By analogy with the approximation of real numbers by
rational numbers, we say that $h \in \hcal$ (or its curvature form
$\omega_h$) has {\it height} $k$ if $h \in \bcal_k$.  A basic fact
is that the union
$${\mathcal B} = \bigcup_{k=1}^{\infty} \bcal_k$$
of Bergman metrics  is dense in the $C^{\infty}$-topology  in the
space $\hcal$  (see \cite{T, Z}).  Indeed, \begin{equation}
\label{TK} \frac{FS_k \circ Hilb_k(h)}{h} = 1 + O(k^{-2}),
\end{equation} where the remainder is estimated in $C^r(M)$ for any $r > 0$;
left side moreover has  a complete asymptotic expansion  (see
\cite{D3, PS} for precise statements).

Now that we have defined the spaces $\hcal$ and $\bcal_k$, we can
compare Monge-Amp\`ere geodesics and Bergman geodesics. Geodesics
of $\hcal$ satisfy the Euler-Lagrange equations for the energy
functional determined by (\ref{metric}); see (\ref{EF}). By
\cite{M,S2,D2}, the geodesics of $\hcal$ in this metric are the
paths $h_t = e^{-\phi_t}h_0$ which satisfy the equation
\begin{equation} \ddot\phi- \frac{1}{2} |\nabla \dot{\phi}|_{\omega_{\phi}}^2=0,
\end{equation}
which  may be interpreted as a homogeneous complex Monge-Amp\`ere
equation on $A \times M$ where $A$ is an annulus \cite{S2,D2}.

Geodesics in $\bcal_k$ with respect to the symmetric space metric
are given by orbits of certain  one-parameter subgroups
$\sigma_k^t = e^{t A_k}$ of $GL(d_k + 1, \C)$. In the
identification of $\bcal_k$ with the symmetric space $\ical_k
\simeq GL(d_k + 1, \C)/U(d_k+ 1)$ of inner products, the 1 PS (one
parameter subgroup) $e^{t A_k} \in GL(d_k+1)$  changes an
orthonormal basis $\hat\us^{(0)}$  for the initial inner product
$G_0$  to an orthonormal basis $e^{t A_k} \cdot\hat\us^{(0)}$ for
$G_t$ where $G_t$ is a geodesic of $\ical_k$. Geometrically, a
Bergman geodesic may be visualized as the path of metrics on $M$
obtained by holomorphically embedding $M$ using a basis of $H^0(M,
L^k)$ and then moving the embedding under the 1 PS subgroup $e^{t
A_k}$ of motions of $\CP^{d_k}$. The difficulty is to interpret
this simple extrinsic motion in intrinsic terms on $M$.

In this article, we only study the  endpoint problem for the
geodesic equation. We assume given $h_0, h_1 \in \hcal$ and let
$h(t) $ denote the Monge-Amp\`ere geodesic between them. We then
consider the geodesic $G_k(t)$ of $\ical_k$ between $G_k(0) =
Hilb_k(h_0)$ and $G_k(1) = Hilb_k(h_1)$ or equivalently between
$FS_k \circ Hilb_k(h_0)$ and $FS_k \circ  Hilb_k(h_1)$. Without
loss of generality, we may assume that the change of orthonormal
basis (or change of inner product) matrix $\sigma_k = e^{A_k}$
between $Hilb_k(h_0), Hilb_k(h_1)$ is diagonal with entries
$e^{\lambda_0},...,e^{\lambda_{d_k}}$ for some $\lambda_j\in\R$.
Let $\hat\us^{(t)}= e^{t A_k} \; \cdot\hat\us^{(0)}$ where $e^{t
A_k}$ is diagonal with entries $e^{\lambda_jt}$. Define
\begin{equation}\label{HKT}   h_k(t) : = FS_k \circ G_k(t)  = h_{\hat\us^{(t)}}=: h_0e^{-\phi_k(t)}. \end{equation}
 It
follows immediately from (\ref{CONSTANT}) that
\begin{equation} \label{intphi}
   \phi_k(t;z)\ = \ {1\over k}
\log\left(\sum_{j=0}^N e^{2\lambda_jt}|\hat s_j^{(0)}|^2_{h_0^k}
\right).
\end{equation}
We emphasize that $\phi_k(t;z) $ is the intrinsic $\bcal_k$
geodesic between the endpoints $FS_k \circ Hilb_k(h_0)$ and $FS_k
\circ Hilb_k(h_1)$. It is of course quite distinct from the
$Hilb_k$-image of the Monge-Amp\`ere geodesic;  the latter is not
intrinsic to $\bcal_k$ and one cannot gain any information on the
$\hcal$-geodesic by studying it.

Let us summarize the notation for hermitian metrics and geodesics
of metrics:

\begin{itemize}

\item For any metric $h$ on $L$, $h^k$ denotes the induced metric
on $L^k$, and for any metric $H$ on $L^k$, $H^{\frac{1}{k}}$ is
the induced metric on $L$;

\item Given $h_0 \in \hcal$, $h_t = e^{- \phi_t} h_0$ is the
Monge-Amp\`ere geodesic;

\item $h_k = FS \circ Hilb_k(h) \in \bcal_k$ is the natural
approximating Bergman metric to $h$, and $h_k(t) = e^{- \phi_k(t)}
h_0$ is the Bergman geodesic (\ref{HKT}).

\end{itemize}

The main result of Phong-Sturm \cite{PS} is that the
Monge-Amp\`ere geodesic $\phi_t$ is approximated by the 1PS
Bergman geodesic $\phi_k(t,z)$ in the following weak $C^0$ sense:
\begin{equation} \label{PSRESULT} \phi_t(z) = \lim_{\ell \to
\infty}\left[ \sup_{k \geq \ell} \phi_k(t, z)\right]^*, \;\;
\mbox{uniformly as} \; \ell \to \infty,\end{equation} where $u^*$
is the upper envelope of $u$, i.e., $u^*(\zeta_0) = \lim_{\epsilon
\to 0} \sup_{|\zeta - \zeta_0| < \epsilon} u(\zeta). $ In
particular, without taking the upper envelope, $ \sup_{k \geq
\ell} \phi_k(t, z) \to \phi(t,z)$ almost everywhere as $\ell \to
\infty$. See also \cite{B} for the subsequent proof of an
analogous result for the adjoint bundle $L^k \otimes K$ (where $K$
is the canonical bundle) with an error estimate $||\phi_k(t) -
\phi(t)||_{C^0} = O(\frac{\log k}{k})$.

\subsection{Statement of results}

Our purpose is to show that  the degree of convergence of $h_k(t)
\to h_t$ or equivalently of $\phi_k(t, z) \to \phi_t(z)$ is much
stronger for toric hermitian metrics on the invariant line bundle
$L \to M$ over a smooth toric \kahler manifold. We recall that a
toric variety $M$ of dimension $m$ carries the holomorphic action
of a complex torus $(\C^*)^m$ with an open dense orbit. The
associated real torus $\T = (S^1)^m$ acts on $M$ in a Hamiltonian
fashion with respect to any invariant \kahler metric $\omega$,
i.e., it possesses a moment map $\mu: M \to P$ with image a convex
lattice polytope. Here, and henceforth, $P$ denotes the closed
polytope; its interior is denoted $P^o$  (see \S \ref{TV} for
background). Objects associated to $M$ are called toric if they
are invariant or equivariant with respect to the torus action
(real or complex, depending on the context).  We define the space
of toric Hermitian metrics by
 \begin{equation} \hcal_{\T} = \{\phi \in \hcal: (e^{i \theta})^* \phi = \phi, \;\; {\rm ~for~all~} e^{i \theta}
  \in \T\}. \end{equation}
 Here, we assume the reference metric $h_0$ is
 $\T$-invariant. We note that since $\T$ has a moment map, it automatically
 lifts to $L$ and hence it makes sense to say that $h_0: L \to \C$ is invariant under it.
  With a slight abuse of notation carried over from  \cite{D},  we also let
 $\phi$ denote the full \kahler potential on the open orbit, i.e.,
 $\omega_{\phi} = dd^c \phi$ on the open orbit. It is clearly
 $\T$-invariant.

Our main result is

\begin{theo} \label{SUM} Let $L \to M$ be a very  ample  toric line bundle over
a smooth compact toric variety $M$. Let $\hcal_T$ denote the space
of toric Hermitian metrics on $L$. Let $h_0, h_1 \in \hcal_T$ and
let $h_t$ be the Monge-Amp\`ere geodesic between them. Let
$h_k(t)$ be the Bergman geodesic between $Hilb_k(h_0)$ and
$Hilb_k(h_1)$ in $\bcal_k$. Let $h_k(t) = e^{- \phi_k(t, z)} h_0$
and let $h_t = e^{- \phi_t(z)} h_0$. Then
$$\lim_{k \to \infty} \phi_k(t, z) = \phi_t(z) $$
in  $C^{2}([0, 1] \times M) $. In fact, there exists $C$ independent of $k$ such that
$$||\phi_k - \phi||_{C^2([0, 1] \times M)} \; \leq C \; k^{-1/3 + \epsilon}, \;\; \forall \epsilon > 0. $$
\end{theo}

Our methods show moreover that away from the divisor at infinity
$\dcal$ (cf. \S \ref{TV}), the function $\phi_k(t,z)$ has an
asymptotic expansion in powers of  $k^{-1}$, and  converges in
$C^{\infty}$ to $\phi_t$. But the asymptotics become complicated
near $\dcal$, and  require  a `multi-scale' analysis involving
distance to boundary facets. It is therefore not clear  whether
$\phi_k$ has an asymptotic expansion in $k^{-1}$ globally on $M$.
At  least, no such asymptotics follow from the known Bergman
kernel asymptotics, on or off the diagonal.
 The analysis of
these regimes for general toric varieties  seems to be fundamental
in `quantum mechanical approximations' on toric varieties.

As mentioned above, the Monge-Amp\`ere equation can be linearized
in the toric case and solved explicitly (\ref{TORPHIT}); we give a
simple new proof in \S \ref{TV}. The  geodesic arcs are easily
seen to be $C^{\infty}$ when the endpoints are $C^{\infty}$. Hence
the $C^2$-convergence result does not improve the known regularity
results on Monge-Amp\`ere geodesics of toric metrics, but pertain
only to the degree of convergence of Bergman to Monge-Amp\`ere
geodesics in a setting where the latter are known to be smooth; it
is possible that the methods can be developed to give regularity
results, but this is a distant prospect (see the remarks at the
end of this introduction).

\subsection{\label{OUTLINE} Outline of the proof}

Let us now  outline the proof of Theorem \ref{SUM}.  We start with
the fact  that the Legendre transform of the \kahler potential
linearizes the Monge-Amp\`ere equation (cf. \S \ref{LIN} and
\cite{A, G, D3}). The Legendre transform $\lcal \phi$  of the
open-orbit \kahler potential $\phi$, a convex function on $\R^m$
in logarithmic coordinates, is the so-called dual symplectic
potential
 \begin{equation} \label{SYMPOT} u_{\phi}(x) = \lcal \phi(x),\end{equation}   a convex function on
the convex polytope $P$.
  Under this Legendre transform,
  the complex Monge-Amp\`ere
equation on $\hcal_{\T}$ linearizes to the equation $\ddot{u} = 0$
and is thus solved by  \begin{equation} \label{UT} u_t =
u_{\phi_0} + t(u_{\phi_1} - u_{\phi_0}). \end{equation} Hence the
solution $\phi_t$ of the geodesic equation on $\hcal$ is solved in
the toric setting by \begin{equation} \label{TORPHIT} \phi_t = \lcal^{-1} u_t. \end{equation} Our goal is to
show that $\phi_k(t; z) \to \lcal^{-1} u_t$ as in (\ref{UT}) in a
strong sense.

The second simplifying feature of the toric setting occurs on the
quantum level. The Bergman geodesic is obtained by applying the
$FS_k$ map to the one-parameter subgroup $e^{t A_k}$. In general,
it is difficult to understand what kind of asymptotic behavior is
possessed by the operators $e^{t A_k}$.
 But on a toric variety, there  exists  a natural basis of the
space of holomorphic sections $H^0(M, L^k)$ furnished by monomial
sections $z^{\alpha}$ which are orthogonal with respect to all
torus-invariant inner products, and with respect to which all
change of basis operators $e^{t A_k}$ are diagonal;  we refer to
\S \ref{TV} or to \cite{STZ} for background.  Hence, we only need
to analyze the eigenvalues of $e^{A_k}$.  The exponents $\alpha$
of the monomials
  are lattice points $\alpha \in k P$ in the $k$th dilate
of the polytope $P$ corresponding to $M$. The eigenvalues in the
toric case are given by

\begin{equation} \label{lambdaalpha} \lambda_{\alpha} : = \frac{1}{2} \log
\left(\frac{\QQ_{h_0^k}(\alpha)}{\QQ_{h_1^k}(\alpha)}\right),
\end{equation} where $\QQ_{h_0^k}(\alpha)$ is a `norming constant' for a
toric inner product.  By a norming constant for a toric Hermitian
inner product $G$ on $H^0(M, L^k)$ we mean the associated $L^2$
norm-squares of the monomials
\begin{equation} \label{QIP} \QQ_{G}(\alpha) = ||s_{\alpha}||_{G}^2. \end{equation}
In  particular, if $h \in \hcal_{\T}$, the norming constants for
$Hilb_k(h)$ are given by
\begin{equation} \label{QHK} \QQ_{h^k}(\alpha) = ||s_{\alpha}||_{h^k}^2:= \int_{M_P}
|s_{\alpha}(z)|^2_{h^k} dV_h. \end{equation} Thus, an orthonormal
basis of $H^0(M, L^k)$ with respect to $Hilb_k(h)$ for $h \in
\hcal_T$ is given by
$\{\frac{s_{\alpha}}{\sqrt{\QQ_{h^k}(\alpha)}}, \;\; \alpha \in k
P \cap \Z^m\}$. An equivalent, and in a sense dual (cf. \S
\ref{PQ}), formulation is in terms of  the functions
\begin{equation} \label{PHK} \pcal_{h^k}(\alpha, z): =
\frac{|s_{\alpha}(z)|^2_{h^k}}{\QQ_{h^k}(\alpha)},
\end{equation} and their special values
\begin{equation} \label{PHKALPHA} \pcal_{h^k}( \alpha): = \pcal_{h^k}(\alpha,
\mu_h^{-1}(\frac{\alpha}{k}))=
\frac{|s_{\alpha}(\mu_h^{-1}(\frac{\alpha}{k}))|^2_{h^k}}{\QQ_{h^k}(\alpha)}.
\end{equation}

Given two toric hermitian metrics $h_0, h_1 \in \hcal_{\T}$, the
change of basis matrix $e^{A_k} = \sigma_{h_0, h_1, k}$ from the
monomial orthonormal basis for $Hilb_k(h_0)$ to that for
$Hilb_k(h_1)$ is diagonal, and the  eigenvalues are given by
\begin{equation} \label{EIGA} Sp(e^{A_k} e^{A_k^*}) : = \{e^{2 \lambda_{\alpha}(k) }=
\frac{\QQ_{h_0^k}(\alpha)}{\QQ_{h_1^k}(\alpha)},\;\;\; \alpha \in
k P \}.
\end{equation} Hence,
 for a
$\bcal_k$-geodesic, (\ref{intphi}) becomes
\begin{equation} \label{toricphik} \phi_k(t, z) = \frac{1}{k} \log
Z_k(t, z)
\end{equation}
where
\begin{equation} \label{toricFk} Z_k(t, z) =
\sum_{\alpha \in k P \cap \Z^m}
\left(\frac{\QQ_{h_0^k}(\alpha)}{\QQ_{h_1^k}(\alpha)}\right)^t
\frac{|s_{\alpha}(z)|^2_{h_0^k}}{\QQ_{h_0^k}(\alpha)}.
\end{equation}
It is interesting to observe that the relative \kahler potential
(\ref{toricphik}) is the logarithm of  an exponential sum, hence
 has the form of a free energy of a statistical mechanical
problem with states parameterized by $\alpha \in k P$ and with
Boltzmann weights
$\left(\frac{\QQ_{h_0^k}(\alpha)}{\QQ_{h_1^k}(\alpha)}\right)^t$.

Thus, our goal is to prove that
\begin{equation} \label{toricFka} \frac{1}{k} \log
\sum_{\alpha \in k P \cap \Z^m}
\left(\frac{\QQ_{h_0^k}(\alpha)}{\QQ_{h_1^k}(\alpha)}\right)^t
\frac{|s_{\alpha}(z)|^2_{h_0^k}}{\QQ_{h_0^k}(\alpha)} \to
\phi_t(z)\;\; \mbox{in}\;\; C^2(A \times M).
\end{equation}

\subsection{Heuristic proof}

Let us next  sketch a heuristic proof which makes the pointwise
convergence obvious. The first step is to obtain good asymptotics
of the norming constants (\ref{QHK}). As in \cite{SoZ}, they may
be expressed in terms of the symplectic potential by
\begin{equation} Q_{h^k}(\alpha) =
 \int_P  e^{ - k \left( u_{\phi}(x) + \langle \frac{\alpha}{k} - x, \nabla u_{\phi}(x) \rangle \right)}
dx
\end{equation}
As $k \to \infty$ the integral is dominated by the unique point $
x = \frac{\alpha}{k}$ where the `phase function' is maximized. The
Hessian is always non-degenerate and by complex stationary phase
we obtain the asymptotics
$$Q_{h^k}(\alpha_k) \sim k^{-m/2} e^{2 k u_{\phi} (\alpha)}. $$
The complex stationary phase (or steepest descent) method does not
apply near the boundary $\partial P$, causing serious
complications, but in this heuristic sketch we ignore this aspect.

 If we  then replace each term in
$Z_k$ by its asymptotics, we obtain
\begin{equation} \label{PHIKTASY} \phi_k(t, e^{\rho/2}) \sim
\frac{1}{k} \log \sum_{\alpha \in P\cap \frac{1}{k} \Z^m} e^{2
k\left(u_0(\alpha) + t
 (u_1(\alpha) - u_0(\alpha)) + \langle \rho, \alpha \rangle\right)}.
\end{equation}
The exponent $\left(u_0(\alpha) + t
 (u_1(\alpha) - u_0(\alpha)) + \langle \rho, \alpha \rangle\right)$
  is convex and therefore has a unique minimum point. This
  suggests applying a discrete analogue of complex stationary
  phase to the sum (\ref{PHIKTASY}),   a Dedekind-Riemann sum which is
asymptotic to the integral
$$ \int_{P} e^{2  k\left(u_0(\alpha) + t
 (u_1(\alpha) - u_0(\alpha)) + \langle \rho, \alpha \rangle \right)} d\alpha.
$$
Taking $\frac{1}{k} \log$ of the integral and applying complex
stationary phase  gives the asymptote
$$\max_{\alpha \in P}\{u_0(\alpha) + t
 (u_1(\alpha) - u_0(\alpha)) + \langle \rho, \alpha \rangle \}.$$ But this is
the
 Legendre transform of the ray of symplectic potentials
$$u_{\phi_0} (\alpha) + t
 (u_{ \phi_1}(\alpha) - u_{
\phi_0}(\alpha)),$$ and thus is  the Monge-Amp\`ere geodesic.

This is the core idea of the proof. We now give the rigorous
version.

\subsection{Outline of the rigorous proof}

The main difficulty in the proof of Theorem \ref{SUM}  is that the
norms have very different asymptotic regimes according to the
position of the normalized lattice point $\frac{\alpha}{k}$
relative to the boundary $\partial P$ of the polytope. Even in the
simplest case of $\CP^m$, the different positions correspond to
the regimes of the central limit theorem, large deviations
theorems and Poisson law of rare events for multi-nomial
coefficients. In determining the asymptotics of (\ref{toricphik}),
we face the difficulty that these Boltzmann weights might be
exponentially growing or decaying in $k$ as $k \to \infty$.

 To simplify the comparison  between the Bergman and
Monge-Amp\`ere geodesics, we  take advantage of the explicit solution (\ref{TORPHIT}) of
geodesic equation to re-write   $Z_k(t,z)$ in the form
\begin{equation} \label{REWRITE}
e^{- k \phi_t(z)} Z_k(t,z) = \sum_{\alpha \in k P \cap \Z^m}
\rcal_k(t, \alpha)
\frac{|s_{\alpha}(z)|^2_{h_t^k}}{\QQ_{h_t^k}(\alpha)} =
\sum_{\alpha \in k P \cap \Z^m} \rcal_k(t, \alpha)
 \pcal_{h_t^k}(\alpha, z),
\end{equation}
where as usual $h_t = e^{- \phi_t} h_0$ (with $\phi_t$ as in (\ref{TORPHIT})),  and where
\begin{equation} \label{QRATIO} \rcal_k(t, \alpha) : = \frac{\QQ_{h_t^k}(\alpha)}{
(\QQ_{h_0^k}(\alpha))^{1-t}(\QQ_{h_1^k}(\alpha))^{t}}
\end{equation} One of the key ideas is that $\rcal_k(t, \alpha)$
is to at least one order  a semi-classical symbol in $k$, i.e., has
at least to some extent an asymptotic expansion in powers of $k$.
Once this is established, it is possible to prove that
\begin{equation} \label{PART2} \frac{1}{k} \log  \sum_{\alpha \in k P \cap \Z^m} \rcal_k(t, \alpha)  \pcal_{h_t^k}
(\alpha, z) \to 0
\end{equation}
in the $C^2$-topology on $[0, 1] \times M$.

The proof of Theorem \ref{SUM} consists of four main ingredients:
\begin{itemize}

\item The Localization Lemma \ref{LOCALIZATION}, which  states
that the sum over $\alpha$ localizes to a ball of radius $O(
k^{-\frac{1}{2} + \delta})$ around the point $\mu_t(z)$.   Here
and hereafter, $\delta$ can be taken to be any sufficiently small
positive constant.

\item Bergman/\szego asymptotics (see \S \ref{BS}), which allow
one to make comparisons between the sum in $Z_k$ and sums with
known asymptotics.

\item The Regularity Lemma \ref{MAINLEM}, which states that the
summands $\rcal_k(t, \alpha)$  one is averaging have sufficiently
smooth asymptotics as $k \to \infty$, allowing one to Taylor
expand to order at least one around the point $\mu_t(z)$.

\item Joint asymptotics of the Fourier coefficients (\ref{PHK})
and particularly their special values $ \pcal_{h^k}( \alpha)$ in
the parameters $k$ and distance to $\partial P$ (see Proposition
\ref{MAINPCAL}). We use a complex stationary phase method in the
`interior region' far from $\partial P$ and local Bargmann-Fock
models near $\partial P$.

\end{itemize}

The  Localization  Lemma is needed not just for  $\rcal_k(t,
\alpha)$ but also for  summands which arise from differentiation
with respect to $(t, z)$:

\begin{lem} (Localization of Sums)  \label{LOCALIZATION}  Let
$B_k(t, \alpha): \Z^m \cap k P \to \C$ be a family of lattice
point functions satisfying $|B_k(t, \alpha) | \leq C_0 k^M$ for
some $C_0, M \geq 0$. Then,  there exists $C
> 0$ so that for any  $\delta > 0$,
$$ \sum_{\alpha \in k P \cap \Z^m
} B_k(t, \alpha)
\frac{|s_{\alpha}(z)|^2_{h_t^k}}{\QQ_{h_t^k}(\alpha)} =
\sum_{\alpha: |\frac{\alpha}{k} - \mu_t(z)| \leq  k^{-\frac{1}{2}
+ \delta} } B_k(t, \alpha)
\frac{|s_{\alpha}(z)|^2_{h_t^k}}{\QQ_{h_t^k}(\alpha)} \; + \;
O_{\delta} (k^{- C}). $$
\end{lem}

The proof is an integration by parts argument. One could localize
to the smaller scale $|\frac{\alpha}{k} - \mu_t(z)| \leq C
\frac{\log k}{\sqrt{k}}$ but then the argument only brings errors
of the order $(\log k)^{- M}$ for all $M$ and that complicates
later applications.

The regularity Lemma concerns  the behavior of the `Fourier
multiplier' $ R_k(t, \alpha)$ (\ref{QRATIO}). The sum
(\ref{toricFk}) formally resembles the Berezin covariant symbol of
a Toeplitz Fourier multiplier, i.e., the restriction to the
diagonal of the Schwartz kernel of the operator;  we refer to
(\cite{STZ2,Z2}) for discussion of such Toeplitz Fourier
multipliers operators on toric varieties and their Berezin
symbols. However, the resemblance is a priori just formal -- it is
not obvious that  $ R_k(t, \alpha)$ has   asymptotics  in $k$.  As
mentioned above, the nature of the asymptotics is most difficult
near $\partial P$; it is not obvious that smooth convergence holds
along $\dcal$, the divisor at infinity.

The purpose of introducing $ R_k(t, \alpha)$ is explained by the
following result. First, we make the

\begin{defin} \label{RINFTY}  We define the metric volume ratio to
be the function on $[0, 1] \times P$ defined by
$$\rcal_{\infty}(t, x):= \left(
\frac{\det \nabla^2 u _t(x)}{ (\det \nabla^2 u_0(x))^{1-t}(\det
\nabla^2 u_1(x))^{t}} \right)^{1/2}. $$
\end{defin}

\begin{lem} \label{MAINLEM} (Regularity) The volume ratio $\rcal_{\infty}(t, x) \in C^{\infty}([0, 1] \times
P)$. Further, for $0 \leq j \leq 2$,
$$ (\frac{\partial}{\partial t})^j  \rcal_k(t, \alpha) = (\frac{\partial}{\partial t})^j
\rcal_{\infty}(t, \frac{\alpha}{k})  + O(k^{-\frac{1}{3} }),
 $$ where the $O$ symbol is uniform in
$(t, \alpha)$.
\end{lem} \noindent This Lemma is  the subtlest part of the
analysis.   If the $\rcal_k$ function were replaced by a fixed
function $f(x)$ evaluated at $\frac{\alpha}{k}$  then the
convergence problem reduces to generalizations of convergence of
Bernstein polynomial approximations to smooth functions \cite{Z2},
and only requires now standard Bergman kernel asymptotics.
However, the actual $ R_k(t, \alpha)$ do not apriori have this
form, and much more is required for their analysis than
asymptotics (on and off diagonal) of Bergman kernels.   The
analysis  uses  a mixture of complex stationary phase arguments in
directions where $\frac{\alpha}{k}$ is `not too close' to
$\partial P$, while directions `close to' $\partial P$ we use an
approximation by the `linear' Bargmann-Fock model (see \S \ref{BF}
and \S \ref{BZ}).

The somewhat unexpected $k^{-1/3}$ remainder estimate has its
origin in this mixture of complex stationary phase and
Bargmann-Fock asymptotics. Both methods are  valid for $k$
satisfying $ \frac{C \log k}{k} \leq \delta_k \leq C'
\frac{1}{\sqrt{k} \log k}$. In this region, the stationary phase
remainder is of order $(k \delta_k)^{-1}$ while the Bargmann-Fock
remainder is of order $k \delta_k^2$; the two remainders agree
when $\delta_k = k^{-\frac{2}{3}}$, and then the remainder is
$O(k^{-1/3})$. For smaller  $\delta_k$ the Bargmann-Fock
approximation is more accurate and for larger $\delta_k$ the
stationary phase approximation is more accurate. This matter is
discussed in detail in \S \ref{BZ}.

The rest of the proof of the $C^2$-convergence may be roughly
outlined as follows: We calculate two logarithmic derivatives of
$e^{- k \phi_t(z)} Z_k(t,z) $ of (\ref{REWRITE}) with respect to
$(t, \rho)$. Using the Localization Lemma \ref{LOCALIZATION} we
can drop the terms in the resulting sums  corresponding to
$\alpha$ for which $|\frac{\alpha}{k} - \mu_t(z)|
> k^{-\frac{1}{2} + \delta}$. In the remaining terms we use the
Regularity Lemma \ref{MAINLEM} to approximate the summands by
their Taylor expansions to order one  around $\mu_t(z)$. This
reduces the expressions to derivatives of the diagonal \szego
kernel
\begin{equation} \label{SZEGOT}
\Pi_{h^k_t}(z, z) = \sum_{\alpha \in k P \cap \Z^m}
\frac{|s_{\alpha}(z)|^2_{h_t^k}}{\QQ_{h_t^k}(\alpha)}
\end{equation}
for the metric $h_t^k$ on $H^0(M, L^k)$ induced by  Monge-Amp\`ere
geodesic $h_t$. Here, we use the smoothness of $h_t$. The known
asymptotic expansion of this kernel (\S \ref{BS}) implies the
$C^2$-convergence of $e^{ k \phi_t(z)} Z_k(t,z) $.

As indicated in this sketch, the key problem is to analyze the
joint asymptotics of norming constants $\QQ_h^k(\alpha)$ and the
dual constants $\pcal_{h^k}( \alpha)$  (\ref{PHKALPHA}) in $(k,
\alpha)$. Norming constants are a complete set of invariants of
toric \kahler metrics.  Initial results (but not joint asymptotics
in the boundary regime) were obtained in \cite{STZ}; norms are
also an important component of Donaldson's numerical analysis of
canonical metrics \cite{D4} on toric varieties. In \cite{SoZ} the
joint asymptotics of $\QQ_h^k(\alpha)$ were studied up to the
boundary of the polytope $[0, 1]$ associated to $\CP^1$. In this
article, we emphasize the dual constants (\ref{PHKALPHA}).

\subsection{\label{BAC} Bergman approximation and complexification}

Having described our methods and results, we return to the
discussion of their relation to \kahler quantization and to the
obstacles in complexifying  $Diff_{\omega_0}(M)$.  Further
discussion  is given in \cite{RZ2}.

We may distinguish two intuitive ideas as to the nature of
Monge-Amp\`ere geodesics. The first heuristic idea, due to Semmes
\cite{S2} and  Donaldson \cite{D},  is to view HCMA geodesics as
one parameter subgroups of $G_{\C}$ where $G =
SDiff_{\omega_0}(M)$. One parameter subgroups of
$SDiff_{\omega_0}(M)$ are defined by Hamiltonian flows of initial
Hamiltonians $\dot{\phi}_0$ with respect to $\omega_0$. A
complexified one parameter subgroup is the analytic continuation
in time of such a Hamiltonian flow \cite{S2,D}. This idea is
heuristic inasmuch as  Hamiltonian flows need not possess analytic
continuations in time; moreover, no genuine complexification of
$SDiff_{\omega_0}(M)$ exists.

The second intuitive idea (backed up by the results of \cite{PS}
and this article)  is to view HCMA geodesics as classical limits
of $\bcal_k$ geodesics. The latter have a very simple extrinsic
interpretation as one parameter motions $e^{t A_k}
\iota_{\underline{s}}(M)$ of a holomorphic embedding
$\iota_{\underline{s}}: M \to \CP^{d_k}$. But the passage to the
classical limit is quite non-standard from the point of view of
\kahler quantization. The problem is that the approximating one
parameter subgroups $e^{t A_k}$ of operators on $H^0(M, L^k)$,
which change an orthonormal basis for an initial inner product to
a path of orthonormal bases for the geodesic of inner products,
are not apriori complex Fourier integral operators or any known
kind of quantization of classical dynamics.

The heuristic view taken in this article and series is that $e^{ t
A_k}$ should be approximately the analytic continuation of the
\kahler  quantization of a classical Hamiltonian flow.   To
explain this, let us recall the basic ideas of \kahler
quantization.

Traditionally, \kahler quantization refers to  the quantization of
a polarized \kahler manifold $(M, \omega, L)$ by Hilbert spaces
$H^0(M, L^k)$ of holomorphic sections of high powers of a
holomorphic line bundle $L \to M$ with Chern class $c_1(L) =
 [\omega]$. The \kahler form determines a Hermitian
metric $h$ such that $Ric(h) = \omega$. The Hermitian metric
induces  inner products $Hilb_k(h)$ on $H^0(M, L^k)$.  In this
quantization theory, functions $H$ on $M$ are quantized as
Hermitian (Toeplitz) operators $\hat{H} := \Pi_{h^k} H \Pi_{h^k}$
on $H^0(M, L^k)$, and canonical transformations of $(M, \omega)$
are quantized as unitary operators on $H^0(M, L^k)$. Quantum
dynamics is given by unitary groups $e^{i t k \hat{H}}$  (see
\cite{BBS,BSj,Z} for references).

In the case of Bergman geodesics with fixed endpoints, $H$ should
be $\dot{\phi}_0$, the initial tangent vector to the HCMA geodesic
with the fixed endpoints. The quantization of the Hamiltonian flow
of $\dot{\phi}$ should then be  $e^{i t k \hat{H}}$  and its
analytic continuation should be  $e^{ t k \hat{H}}$. The change of
basis operator $e^{t A_k}$ should then be approximately the same
as $e^{ t k \hat{H}}$. But proving this and taking the classical
limit is  necessarily non-standard when the classical analytic
continuation of the Hamiltonian flow of $\dot{\phi}$ does not
exist. Moreover, we only know that $\dot{\phi} \in C^{1, 0} $.

This picture of the Bergman approximation to HCMA geodesics is
 validated in this article in the case of the Dirichlet
problem on projective toric \kahler manifolds.  It is verified for
the initial value problem on toric \kahler manifolds in
\cite{RZ2}. In work in progress, we are investigating the same
principle for general \kahler metrics on Riemann surfaces
\cite{RZ4}.

\subsection{\label{Overview} Final remarks and further results and problems}

An obvious question within the toric setting is whether $\phi_k(t)
\to \phi_t$ in a stronger topology than $C^2$ on a toric variety.
It seems possible that the methods of this paper  could be
extended to $C^k$-convergence. The methods of this paper easily
imply $C^k$ convergence for all $k$ away from $\partial P$ or
equivalently the divisor at infinity, but the degree of
convergence along this set has yet to be investigated. As
mentioned above, we do not see why $\phi_k$ should have an
asymptotic expansion in $k$, but this aspect may deserve further
exploration. We also mention that our methods can be extended to
prove $C^2$-convergence of Berndtsson's approximations in
\cite{B}.

In subsequent articles on the toric case,  we build on the methods
introduced here to prove convergence theorems for other  geodesics
and for general harmonic maps \cite{RZ} (including the
Wess-Zumino-Witten equation).   In \cite{SoZ2}, we develop the
methods of this article to prove that the geodesic rays
constructed in \cite{PS1} from test configurations are $C^{1,1}$
and no better on a toric variety. Test configuration  geodesic
rays are solutions of a kind of initial value problem; we refer to
the articles \cite{PS1,SoZ2} for the definitions and results. For
test configuration geodesics, the analogue of $\rcal_k$ is not
even smooth in $t$. The smooth initial value problem is studied in
\cite{RZ2}.
 In a
different direction, one of the authors and Y. Rubinstein  prove a
$C^2$ convergence result for completely general harmonic maps of
Riemannian manifolds with boundary into toric varieties (see
\cite{R,RZ}). This includes the Wess-Zumino-Witten model where the
manifold is a Riemann surface with boundary.

 We believe that the techniques of this paper extend  to other
 settings with a high degree of symmetry, such as Abelian
 varieties and other settings discussed in \cite{D5}. The general
 \kahler case involves significant further obstacles.
 A basic problem in generalizing the
results is to construct a useful localized basis of sections on a
general $(M, \omega)$. In the toric case, we use the basis of
$\T$-invariant states $\hat{s}_{\alpha} = z^{\alpha}$, which
`localize' on the so-called `Bohr-Sommerfeld tori', i.e. the
inverse images $\mu^{-1}(\frac{\alpha}{k})$  of lattice points
under the moment map $\mu$. Such Bohr-Sommerfeld states also exist
on any Riemann surface; in \cite{RZ4}, we relate them to the
convergence problem for HCMA geodesics on Riemann surfaces.

 We briefly speculate on the higher dimensional general \kahler case. There are a number of plausible
substitutes for the Bohr-Sommerfeld basis on a general \kahler
manifold. A rather traditional one  is to study the asymptotics of
$e^{A_k}$ on
 a  basis of {\it coherent states}
 $\Phi_{h^k}^w$.  Here, $\Phi_{h^k}^w(z) =
 \frac{\Pi_{h^k}(z,w)}{\sqrt{\Pi_{h^k}(w,w)}}$ are $L^2$
 normalized \szego kernels pinned down in the second argument.
 Intuitively, $\Phi_{h^k}^w$
 is  like a Gaussian bump centered at $w$ with shape determined by the metric $h$. It is thus
 more localized than the monomials $z^{\alpha}$, which are only Gaussian transverse to the tori.  Under the
change of basis operators $e^{t A_k}$,  both the center and shape
should change. Like the monomials $z^{\alpha}$,  coherent states
 have some degree of orthogonality. There are
in addition other well-localized bases depending on the \kahler
metric which may be used in the analysis.

 Our  main result (Theorem \ref{SUM}) may be viewed heuristically
as showing that as $k \to \infty$  the change of basis operators
$e^{t A_k}$ tend to a path $f_t$ of diffeomorphisms changing the
initial \kahler metric $\omega_0$ into the metric $\omega_t$ along
the Monge-Amp\`ere geodesic. We conjecture that  $e^{t A_k}
\Phi_{h^k}^w \sim
 \Phi_{h_t^k}^{f_t(w)}$,
 where $h_t$ is the Monge-Amp\`ere geodesic and $f_t$ is the Moser
 path of
diffeomorphisms such that $f_t^* \omega_0 = \omega_t$. We leave
the exact degree of asymptotic similarity vague at this time since
even the regularity of the Moser path is currently an open
problem.

\bigskip

\noindent{\bf Acknolwedgements}  The  authors would like to thank
D. H. Phong and J. Sturm for their  support of this project, and
to thank  them and Y. A. Rubinstein for many detailed corrections.
The second author's collaboration with Y. A. Rubinstein subsequent
to the initial version of this article has led to a deepened
understanding of the global approximation problem, which is
reflected in the revised version of the introduction.

\section{Background on toric varieties \label{TV} }

In this section, we review the necessary background on toric
\kahler manifolds. In addition to standard material on \kahler and
symplectic potentials, moment maps and polytopes, we also present
some rather non-standard material on almost analytic extensions of
\kahler potentials and moment maps that are needed later on. We
also give a simple proof that the Legendre transform from \kahler
potentials to symplectic potentials  linearizes the Monge-Amp\'ere
equation.

Let $M$ be a complex manifold.  We use the following standard
notation: $\frac{\partial}{\partial z} = \frac{1}{2}
(\frac{\partial}{\partial x} - i \frac{\partial}{\partial y} ),
\frac{\partial}{\partial \bar{z}} = \frac{1}{2}
(\frac{\partial}{\partial x} + i \frac{\partial}{\partial y} ). $
We often find it  convenient to use the real operators  $d =
\partial + \dbar, d^c := \frac{i}{4 \pi} (\dbar -
 \partial)$ and   $dd^c = \frac{i}{2\pi}
 \ddbar$.

Let $L \to M$ be a holomorphic line bundle.  The  Chern form of a
Hermitian metric $h$ on $L$ is defined by
\begin{equation}\label{curvature} c_1(h)= \omega_h : = -\frac{\sqrt{-1}}{2 \pi}\ddbar \log \|e_L\|_h^2\;,\end{equation} where $e_L$ denotes a local
holomorphic frame (= nonvanishing section) of $L$ over an open set
$U\subset M$, and $\|e_L\|_h=h(e_L,e_L)^{1/2}$ denotes the
$h$-norm of $e_L$. We say that $(L,h)$ is positive if the (real)
2-form $\omega_h $  is a positive $(1,1)$ form,  i.e.,  defines a
\kahler metric. We write $\|e_L(z)\|_h^2 = e^{-\phi} $ or locally
$h = e^{- \phi}$,  and then refer to $\phi$ as the \kahler
potential of $\omega_h$ in $U$. In this notation,
\begin{equation} \label{DDCPHI} \omega_h = \frac{\sqrt{-1}}{2 \pi} \ddbar \phi  = dd^c
\phi. \end{equation} If we fix a Hermitian metric $h_0$ and let $h
= e^{- \phi} h_0$, and put $\omega_0 = \omega_{h_0}$, then
\begin{equation} \label{DDCPHIrel} \omega_h = \omega_0 + dd^c
\phi. \end{equation}
 The metric $h$ induces Hermitian metrics
$h^k$ on $L^k=L\otimes\cdots\otimes L$ given by $\|s^{\otimes
k}\|_{h_N}=\|s\|_h^k$.

We now specialize to toric \kahler manifolds;  for background,  we
refer to \cite{A, D3, G, STZ}. A toric \kahler manifold is a
\kahler  manifold $(M, J, \omega)$ on which the complex torus
$(\C^*)^m$ acts holomorphically with an open orbit $M^o$. Choosing
a basepoint $m_0$ on the  open orbit identifies $M^o \equiv
(\C^{*})^{m}$ and give the point $z = e^{\rho/2 +i\varphi} m_0$
the holomorphic coordinates
\begin{equation} \label{OPENORBCOORDS}
z=e^{\rho/2 +i\varphi} \in (\C^{*})^{m},\quad \rho, \varphi \in
\R^{m}. \end{equation} The real torus $\T \subset (\C^*)^m$ acts
in a Hamiltonian fashion with respect to $\omega$. Its moment map
$\mu = \mu_{\omega}: M \to P \subset {\bf t}^* \simeq \R^m$ (where
${\bf t}$ is the Lie algebra of $\T$) with respect to $\omega$
defines a singular torus fibration over a convex lattice polytope
$P$; as in the introduction, $P$ is understood to be the closed
polytope.  We recall that the moment map of a Hamiltonian torus
action with respect to a symplectic form $\omega$ is the map
$\mu_{\omega}: M \to {\bf t}^*$ defined by $d \langle
\mu_{\omega}(z), \xi \rangle = \iota_{\xi^{\#}} \omega$ where
$\xi^{\#}$ is the vector field on $M$ induced by the vector $\xi
\in {\bf t}$.  Over the open orbit one thus has a symplectic
identification
$$\mu: M^o \simeq P^o \times \T.$$
We let $x$ denote the Euclidean coordinates on $P$. The components
$(I_1, \dots, I_m)$  of the moment map are called action variables
for the torus action. The symplectically dual variables on $\T$
are called the angle variables. Given a basis of ${\bf t}$ or
equivalently of the action variables,  we denote by
$\{\frac{\partial}{\partial \theta_j}\}$ the corresponding
generators (Hamiltonian vector fields) of the $\T$ action. Under
the complex structure $J$, we also obtain generators
$\frac{\partial}{\partial \rho_j}$ of the $\R_+^m$ action.

The action variables are globally defined smooth functions but
fail to be coordinates at points where  the generators of the $\T$
action vanish. We denote the set of such points by $\dcal$ and
refer to it as the divisor at infinity. If $p \in \dcal$ and
$\T_p$ denotes the isotropy group of $p$, then the generating
vector fields of $\T_p$ become linearly dependent at $P$. Since we
are proving $C^2$ estimates, we need to replace them near points
of $\dcal$ by vector fields with norms bounded below.
 We discuss good choices of
coordinates near points of $\dcal$ below.

We assume $M$ is smooth and that $P$ is a Delzant polytope. It is
 defined by a set of linear inequalities
$$\ell_r(x): =\langle x, v_r\rangle-\lambda_r \geq 0, ~~~r=1, ..., d, $$
where $v_r$ is a primitive element of the lattice and
inward-pointing normal to the $r$-th $(m-1)$-dimensional facet
$F_r = \{\ell_r = 0\}$  of $P$. We recall that a facet is a
highest dimensional face of a polytope. The inverse image
$\mu^{-1}(\partial P)$ of the boundary of $P$ is the divisor at
infinity $\dcal \subset M$. For $x \in
\partial P$ we denote by
$$\fcal(x) = \{r: \ell_r(x) = 0\}$$
the set of facets containing $x$. To measure when $x \in P$ is
near the boundary we further define
\begin{equation} \label{FCALE} \fcal_{\epsilon} (x) = \{r: |\ell_r(x)| <
\epsilon\}. \end{equation}

The simplest  toric varieties are linear \kahler manifolds $(V,
\omega)$ carrying a linear holomorphic  torus action. They provide
local models near a corner of $P$ or equivalently near a fixed
point of the $\T$ action.  As discussed in \cite{GS, LT}, a linear
symplectic torus action is determined by a choice of $m$ elements
$\beta_j$  of the weight lattice of the Lie algebra of the torus.
The vector space then decomposes $(V, \omega) = \bigoplus (V_i,
\omega_i)$ of
 orthogonal symplectic subspaces so that the moment map has the
 form
 \begin{equation} \label{BFMM} \mu_{BF}(v_1, \dots, v_m) = \sum |v_j|^2 \beta_j.
 \end{equation}
 The  image of the moment map is the orthant
$\R_+^m$. This provides a useful local model at corners. We refer
to these as Bargmann-Fock models; they play a fundamental role in
this article (cf. \S \ref{BF}).

\subsection{\label{SOC}Slice-orbit coordinates}

We will also need local models at points  near codimension $r$
faces, and therefore supplement the coordinates
(\ref{OPENORBCOORDS}) on the open orbit with holomorphic
coordinates valid in neighborhoods of points of $\dcal$. An atlas
of coordinate charts for $M$ generalizing the usual affine charts
of $\CP^m$  is given
  in \cite{STZ}, \S 3.2 and we briefly recall the definitions.  For each vertex $v_0 \in P$,
we define  the chart $U_{v_0}$
 by
\begin{equation}
U_{v_{0}}:=\{z \in M_{P}\,;\,\chi_{v_{0}}(z) \neq 0\},
\end{equation} where
$$\chi_{\alpha}(z) =  z^{\alpha} = z_1^{\alpha_1} \cdots
z_n^{\alpha_n}.$$ Throughout the article we use standard
multi-index notation, and put $|\alpha| = \alpha_1 + \cdots +
\alpha_m$.  Since $P$ is Delzant, we can choose lattice points
$\alpha^{1},\ldots,\alpha^{m}$ in $P$ such that each $\alpha^{j}$
is in an edge incident to the vertex $v_{0}$, and the vectors
$v^{j}:=\alpha^{j}-v_{0}$ form a basis of $\Z^{m}$. We define
\begin{equation}
\label{CChange} \eta:(\C^{*})^{m} \to (\C^{*})^{m}, \quad
\eta(z)=\eta_{j}(z):=(z^{v^{1}},\ldots,z^{v^{m}}).
\end{equation}
The map $\eta$ is a $\T$-equivariant biholomorphism with  inverse
\begin{equation}  z:(\C^{*})^{m} \to (\C^{*})^{m},\quad z(\eta)=(\eta^{\Gamma
e^{1}},\ldots,\eta^{\Gamma e^{m}}),
\end{equation}
where $e^{j}$ is the standard basis for $\C^{m}$, and $\Gamma$ is
an $m \times m$-matrix with $\det \Gamma =\pm 1$ and integer
coefficients defined by
\begin{equation}\label{GAMMADEF}
\Gamma v^{j}=e^{j},\quad v^{j}=\alpha^{j}-v_{0}.
\end{equation}
 The
corner of $P$ at $v_0$ is transformed to the standard corner of
the orthant $\R_+^m$ by
 the affine linear transformation
\begin{equation}\label{GAMMATWDEF}
\tilde{\Gamma}:\R^{m} \ni u \to \Gamma u -\Gamma v_{0} \in \R^{m},
\end{equation}
which preserves $\Z^{m}$, carries $P$ to a polytope $Q_{v_0}
\subset \{x \in \R^{m}\,;\,x_{j} \geq 0\}$ and carries the facets
$F_j$ incident at $v_0$ to the coordinate hyperplanes $=\{x \in
Q_{v_0}\,;\,x_{j}=0\}$. The  map $\eta$ extends a homeomorphism:
\begin{equation}
\eta:U_{v_{0}} \to \C^{m},\quad \eta(z_{0})=0,\quad z_{0}=\mbox{
the fixed point corresponding to } v_{0}.
\end{equation}
By this homeomorphism, the set $\mu_{P}^{-1}(\bar{F}_{j})$
corresponds to the set $\{\eta \in \C^{m}\,;\,\eta_{j}=0\}$. If
$\bar{F}$ be a closed face with $\dim F=m-r$ which contains
$v_{0}$, then there are facets  $F_{i_{1}},\ldots,F_{i_{r}}$
incident at $v_0$ such that $\bar{F}=\bar{F}_{i_{1}} \cap \cdots
\cap \bar{F}_{i_{r}}$. The subvariety $\mu_{P}^{-1}(\bar{F})$
corresponding $\bar{F}$ is expressed by
\begin{equation}\label{ETAIJ}
\mu_{P}^{-1}(\bar{F}) \cap U_{v_{0}}= \{\eta \in
\C^{m}\,;\,\eta_{i_{j}}=0,\quad j=1,\ldots,r\}.
\end{equation}
When working near a point of $\mu_{P}^{-1}(\bar{F})$, we simplify
notation by writing
  \begin{equation} \label{PRIMECOORD} \eta=(\eta',\eta'') \in
\C^{m}=\C^{r} \times \C^{m-r} \end{equation}  where $\eta' =
(\eta_{i_j})$ as in (\ref{ETAIJ}) and where $\eta''$ are the
remaining $\eta_j$'s, so that  $(0,\eta'')$ is a local coordinate
of the submanifold $\mu_{P}^{-1}(\bar{F})$. When the point $(0,
\eta'') $ lies in the open orbit of $\mu_{P}^{-1}(\bar{F})$, we
often write $\eta'' = e^{i \theta'' + \rho''/2}.$ In practice, we
simplify notation by tacitly treating the corner at $v_0$ as if it
were the standard corner of $\R_+^m$,  omit mention of $\Gamma$
and always use $(z', z'')$ instead of $\eta$. It is
straightforward to rewrite all the expressions we use in terms of
the more careful coordinate charts just mentioned.

These coordinates may be described more geometrically as {\it
slice-orbit} coordinates. Let $P_0 \in \mu_{P}^{-1}(\bar{F})$ and
let $(\C^*)^m_{P_0}$ denote its stabilizer (isotropy) subgroup.
Then there always exists a local slice at $P_0$, i.e., a local
analytic subspace $S \subset M$ such that  $P_0 \in S$, $S$ is
invariant under $(\C^*)^m_{P_0}$,  and such that  the natural
$(\C^*)^m$ equivariant map of the normal bundle of the orbit
$(\C^*)^m \cdot P_0$,
\begin{equation} \label{SLICE} [\zeta, P] \in (\C^*)^m
\times_{(\C^*)^m_z} S \to \zeta \cdot P \in M
\end{equation}  is biholomorphism onto $(\C^*)^m \cdot S$. The terminology
is taken from \cite{Sj} (see Theorem 1.23).  The  slice $S$ can be
taken to be the image of a ball in  the hermitian normal space
$T_{P_0} ((\C^*)^m P_0)^{\perp}$ to the orbit under any local
holomorphic embedding  $w: T_{P_0} ((\C^*)^m P_0)^{\perp} \to M$
with $w(P_0) = P_0, dw_{P_0} = Id.$  The affine coordinates
$\eta''$ above define the slice $S = \eta^{-1} \{(z', z''(P_0)):
z' \in (\C^*)^{r}\}$.  The local `orbit-slice' coordinates are
then defined by
\begin{equation} \label{OS} P = (z', e^{i \theta'' + \rho''/2})
\iff \eta(P) = e^{i \theta'' + \rho''/2} (z', 0)
\end{equation} where $(z', 0) \in S$ is  the point on the slice  with affine holomorphic
coordinates $z' = (\eta')$.

As will be seen below, toric functions are smooth functions of the
variables $e^{\rho_j}$ away from $\dcal$, and  of the variables
$|z_j|^2$ at points near $\dcal$. We introduce the following
`polar coordinates' centered at a point $P \in \dcal$:
\begin{equation} \label{R} r_j :=  |z_j|  = e^{\rho_j/2}.
\end{equation}
They are polar coordinates along the slice.  The gradient vector
field of $r_j$ is denoted $\frac{\partial}{\partial r_j}$. As with
polar vector fields, it is not well-defined at $r_j = 0$. But to
prove $C^{\ell}$ estimates of functions which are smooth functions
of $r_j^2$ it is sufficient to prove $C^{\ell}$ estimates with
respect to the vector fields $\frac{\partial}{\partial r_j}$ or $\frac{\partial}{\partial (r_j^2)}$.

\subsection{\kahler potential in the open orbit  and symplectic potential\label{KPSP}}

Now consider the \kahler metrics $\omega$ in $\hcal$ (cf.
(\ref{HCALDEF})). We recall that on any simply connected open set,
a \kahler metric may be locally expressed as $\omega =  2  i
\ddbar \phi$ where $\phi$ is a locally defined function which is
unique up to the addition $\phi \to \phi + f(z) +
\overline{f(z)}$ of the real part of  a holomorphic or
antiholomorphic function $f$. Here, $a \in \R$ is a real constant
which depends on the choice of coordinates.  Thus, a \kahler
metric $\omega \in \hcal$ has a \kahler potential $\phi$ over  the
open orbit $M^o \subset M$. In fact, there is a canonical choice
of the open-orbit \kahler potential once one fixes the image $P$
of the moment map:
\begin{equation} \label{CANKP} \phi(z) = \log  \sum_{\alpha \in
P}  |z^{\alpha}|^2 = \log  \sum_{\alpha \in P} e^{\langle \alpha,
\rho \rangle}.
\end{equation}
Invariance under the real torus action implies that $\phi$ only
depends on the $\rho$-variables, so that we may write it in the
form
\begin{equation} \label{PHIVSF} \phi(z) = \phi(\rho) =  F(e^{\rho}). \end{equation}
The notation  $\phi(z) = \phi(\rho)$ is an abuse of notation,  but
is rather standard since \cite{D3}.  For instance, the
Fubini-Study \kahler potential is $\phi(z) =\log (1 + |z|^2) =
\log (1 + e^{\rho}) = F(e^{\rho})$. Note that the \kahler
potential $\log (1 + |z|^2)$ extends  to $\C^m$ from the open
orbit $(\C^*)^m$, although the coordinates $(\rho, \theta)$  are
only valid on the open orbit. This is a typical situation.

On the open orbit, we then have
\begin{equation} \label{OMHESSPHI} \omega_{\phi} = \frac{i}{2}  \sum_{j,
k} \frac{\partial^2 \phi(\rho)}{\partial \rho_k  \partial \rho_j}
\frac{dz_j}{z_j} \wedge \frac{d\bar{z}_k}{\bar{z}_k}
\end{equation} Positivity of $\omega_{\phi}$ implies that
$\phi(\rho) = F(e^{\rho})$ is a strictly convex function of $\rho
\in \R^n$. The moment map with respect to $\omega_{\phi}$ is given
on the open orbit by
\begin{equation} \label{MMDEF} \mu_{\omega_{\phi}} (z_1, \dots, z_m) =
\nabla_{\rho} \phi(\rho) =  \nabla_{\rho} F (e^{\rho_1}, \dots,
e^{\rho_m}), \;\;\; (z = e^{\rho/2 + i \theta}).
\end{equation}
 Here, and henceforth, we subscript moments
maps either by the Hermitian metric $h$ or by a local \kahler
potential $\phi$. The formula (\ref{MMDEF}) follows from the fact
that the generators $\frac{\partial}{\partial \theta_j}$ of the
$\T$ actions are Hamiltonian vector fields with respect to
$\omega_{\phi}$ with Hamiltonians $\frac{\partial \phi(\rho)
}{\partial \rho_j}$, since
\begin{equation} \label{HAMS}  \iota_{\frac{\partial}{\partial
\theta_j} } \omega_{\phi}  = d \frac{\partial \phi }{\partial
\rho_j}.
\end{equation}
 The moment map is a homeomorphism from
$\rho \in \R^m$ to the interior $P^o$ of $P$ and  extends as a
smooth map from $M \to \bar{P}$ with critical  points on the
divisor at infinity $\dcal$. Hence, the Hamiltonians (\ref{HAMS})
extend to $\dcal$.

Note that the  local \kahler potential on the open orbit  is not
the same as the global smooth {\it relative \kahler potential } in
(\ref{HCALDEF}) with respect to a background \kahler metric
$\omega_0$. That is,  given a reference metric $\omega_0$ with
\kahler potential $\phi_0$, it follows by the $\ddbar$ lemma that
$\omega = \omega_0 + dd^c \phi$ with $\phi
 \in C^{\infty}(M)$.   As discussed in \cite{D3}
(see Proposition 3.1.7),  the \kahler potential $\phi$ on the open
orbit defines a singular potential on $M$ which satisfies $dd^c
\phi = \omega + H$ where $H$ is a fixed current supported on
$\dcal$. We generally denote \kahler potentials by $\phi$ and in
each context explain which type we mean.

 By (\ref{OMHESSPHI}), a  $\T$-invariant \kahler potential  defines a real convex function
on $\rho \in \R^m$. Its Legendre dual is the  {\it symplectic
potential} $u_{\phi}$: for $x \in P$ there is a unique $\rho$ such
that $\mu_{\phi}(e^{\rho/2}) = \nabla_{\rho} \phi = x$. Then the
Legendre transform is defined to be the convex function
\begin{equation} \label{SYMPOTDEF} u_{\phi}(x) = \langle x, \rho_x \rangle -
\phi(\rho_x), \;\;\; e^{\rho_x/2 } = \mu_{\phi}^{-1}(x) \iff
\rho_x = 2 \log   \mu_{\phi}^{-1}(x)
\end{equation}  on $P$.  The
gradient $\nabla_x u_{\phi}$ is an inverse to
$\mu_{\omega_{\phi}}$ on $M_{\R}$ on the open orbit, or
equivalently on  $P$, in the sense that $\nabla u_{\phi}
(\mu_{\omega_{\phi}}(z)) = z$ as long as $\mu_{\omega_{\phi}}(z)
\notin \partial P$.

 The symplectic potential has canonical logarithmic
singularities on $\partial P$. According to
 \cite{A} (Proposition 2.8) or \cite{D3} ( Proposition 3.1.7),  there is a one-to-one
correspondence between $\T_{\R}$-invariant \kahler potentials
$\psi$ on $M_P$ and symplectic potentials $u$ in the class $S$ of
continuous convex functions on $\bar{P}$ such that $u - u_0$ is
smooth on $\bar{P}$ where \begin{equation} \label{CANSYMPOT}
u_0(x) = \sum_k \ell_k(x) \log \ell_k(x).  \end{equation} Thus,
$u_{\phi}(x) = u_0(x) + f_{\phi}(x)$ where $f_{\phi} \in
C^{\infty}(\bar{P})$. We note that $u_0$ and $u_{\phi}$ are
convex, that $u_0 = 0$ on $\partial P$ and hence $u_{\phi} =
f_{\phi}$ on $\partial P$. By convexity, $\max_{P} u_0 = 0$.

 We denote by $G_{\phi} =
\nabla^2_x u_{\phi}$ the Hessian of the symplectic potential. It
has simple poles on $\partial P$. It follows that $\nabla^2_{\rho}
\phi$ has a kernel along  $\dcal$.  The kernel of
$G_{\phi}^{-1}(x)$ on $T_x
\partial P$ is the linear span of the normals $\mu_r$ for $r \in
\fcal(x)$. We also denote by $H_{\phi}(\rho) = \nabla^2_{\rho}
\phi(e^{\rho})$ the Hessian of the \kahler potential on the open
orbit in $\rho$ coordinates. By Legendre duality, \begin{equation}
\label{GINV} H_{\phi}(\rho) = G_{\phi}^{-1}(x), \;\; \mu(e^{\rho})
= x. \end{equation}  This relation may be extended  to $\dcal \to
\partial P$. The kernel of the left side is  the Lie algebra of
the isotropy group $G_p$ of any point $p \in \mu^{-1}(x)$. The
volume density has the form \begin{equation} \label{VOLDEN}
\det(G_{\phi}^{-1}) = \delta_{\phi}(x)\cdot\prod_{r=1}^{d}
\ell_{r}(x), \end{equation} for some positive smooth function
$\delta_{\phi}$ \cite{A}. We note that $\log \prod_{r=1}^{d}
\ell_{r}(x)$ is known in convex optimization as the logarithmic
barrier function of $P$.

\subsection{\kahler potential near $\dcal$}

We also need    smooth local \kahler potentials in neighborhoods
of points $z_0 \in \dcal$. We note that the open orbit \kahler
potential (\ref{CANKP}) is well-defined near $z = 0$. Local
expressions for the \kahler potential at other points of $\dcal$
essentially amount to making an affine transformation of $P$ to
transform a given corner of $P$ to $0$, and in these coordinates
the local \kahler potential near any point of $\dcal$ can be
expressed in the form (\ref{CANKP}).  For instance, on $\CP^1$, a
\kahler potential valid at $z = \infty$ is given in the
coordinates $w = \frac{1}{z}$ by $\log (1 + |w|^2)$. It  differs
on the open orbit from the canonical \kahler potential $\log (1 +
|z|^2)$ by the term $\log |z|^2$ whose $i \ddbar$ is a delta
function at $z = 0$, supported on $\dcal$ away from the point $w =
0$ that one is studying. In \cite{Song} the reader can find
further  explicit examples of toric \kahler potentials in affine
coordinate charts. Hence, in what follows, we will always use
(\ref{CANKP}) as the local expression of the \kahler potential,
without explicitly writing in the affine change of variables.

We will however need to be explicit about the use of slice-orbit
coordinates $z_j', \rho_j''$ (\ref{OS}) in the local expressions
of the \kahler potential. The coordinates  near $z_0$ depend  on
$\fcal_{\epsilon}(z_0)$ from (\ref{FCALE}). For each $z_0 \in
\dcal$ corresponding to a codimension $r$ face of $P$,  after an
affine transformation changing the face to $x' = 0$, we may write
the  \kahler potential as the canonical one in slice-orbit
coordinates,  $F(|z'|^2, e^{\rho''})$  \S \ref{SOC}  (\ref{OS}).
Since $0 \in P$,   $F$ is smooth up to the boundary face $z' = 0$.
The fact that  $F$ is smooth up to the boundary also follows from
the general fact that a smooth $\T$-invariant function  $g \in
C^{\infty}_{\T} (M)$ may be expressed in the form $g(z) =
\hat{F}_g(\mu_{\phi}(z))$ where as $\hat{F}_g \in
C^{\infty}(\R^m)$. This is known as the divisibility property of
$\T$-invariant smooth functions  (cf. \cite{LT}). It implies that
$F$ is a smooth function of the polar coordinates $r_j^2$ near
points of $\dcal$ in the sense of (\ref{R}).

\subsection{\label{AAEa}Almost analytic extensions}

In analyzing the Bergman/\szego kernel and the functions
(\ref{PHK}), we make use of the  {\it almost analytic extension}
$\phi(z,w)$  to $M \times M$ of a \kahler potential for a \kahler
$\omega$; for background on almost analytic extensions, see
\cite{BSj,MSj}. It is defined near the totally real anti-diagonal
$(z, \bar{z}) \in M \times M$ by
\begin{equation} \label{AAE} \phi_{\C} (x + h, x + k) \sim
 \sum_{\alpha, \beta} \frac{\partial^{\alpha + \beta}
\phi}{\partial z^{\alpha}
\partial \bar{z}^{\beta}} (x) \frac{h^{\alpha}}{\alpha!}
\frac{k^{\beta}}{\beta!}. \end{equation} When $\phi$ is real
analytic on $M$, the almost analytic extension $\phi(z,w)$ is
holomorphic in $z$ and anti-holomorphic in $w$ and is the unique
such function for which $\phi(z) = \phi(z,z)$. In the general
$C^{\infty}$ case, the almost analytic extension   is a smooth
function with the right side of (\ref{AAE}) as its $C^{\infty}$
Taylor expansion along the anti-diagonal, for which  $\dbar
\phi(z,w) = 0$ to infinite order on the anti-diagonal.   It is
only defined in a small neighborhood $(M \times M)_{\delta} =
\{(z,w): d (z,w)< \delta\} $ of the anti-diagonal in $M \times M$,
where $d (z,w)$ refers to the distance between $z$ and $w$ with
respect to the \kahler metric $\omega$.  It is well defined up to
a smooth function vanishing to infinite order on the diagonal; the
latter
 is negligible for our purposes (cf. Proposition 1.1 of
\cite{BSj}.)

The analytic continuation $\phi(z,w)$ of the \kahler potential was
used by Calabi \cite{Ca} in the analytic case to define a
\kahler distance function, known as the
 `Calabi diastasis
function' \begin{equation} \label{DIASTASIS} D(z,w): =  \phi(z,w)
+ \phi(w,z) - (\phi(z) + \phi(w)).
\end{equation}
Calabi showed that
\begin{equation}\label{HESSDIA}   D(z,w) = d (z,w)^2 +
O(d(z,w)^4),\;\; dd^c_w D(z,w)|_{z = w} = \omega.
\end{equation}
 One has the same notion in the almost analytic sense.

The gradient of the almost analytic extension of the \kahler
potential in the toric case defines the almost analytic extension
$\mu_{\C}(z,w)$ of the moment map.  We are mainly interested in
the case where $w = e^{i \theta} z$ lies on the $\T$-orbit of $z$,
and by (\ref{MMDEF}) we have,
\begin{equation} \label{MMDEFa} i \mu_{\C} (z, e^{i \theta} z) =
\nabla_{\theta} \phi_{\C}(z, e^{i \theta} z)  = \nabla_{\theta}
F_{\C}( e^{i \theta} |z|^2),
\end{equation}
where $F$ is defined in (\ref{PHIVSF}).  We sometimes drop the
subscript in $F_{\C}$ and $\mu_{\C}$ since there is only one
interpretation of their extension;  but we emphasize that $\phi(z,
e^{i \theta} z) = F_{\C}( e^{i \theta} |z|^2)$ is very different
from $\phi(e^{i \theta} z) = F(|e^{i \theta} z|^2) = F(|z|^2)$.
For example,  the moment map of the Bargman-Fock model  $(\C^m,
|z|^2)$ is $\mu(z) = (|z_1|^2, \dots, |z_m|^2)$, whose analytic
extension is $(z_1 \bar{w}_1, \dots, z_m \bar{w}_m)$. Similarly
that of the Fubini-Study metric on $\CP^m$ is (in multi-index
notation) $\mu_{FS, \C}(z,w) = \frac{z \cdot \bar{w}}{1 + z \cdot
\bar{w}}.$ In \S \ref{BF} we further  illustrate the notation in
the basic examples of Bargmann-Fock and Fubini-Study models. We
also observe that (\ref{MMDEFa}) continues to hold for the \kahler
potential $F(|z'|^2, e^{\rho''})$ in slice-orbit coordinates. That
is we have,
\begin{equation} i \;  \mu(z', e^{\rho''/2}) =  \nabla_{\theta', \theta''} F_{\C}(e^{i \theta'} |z'|^2, e^{i \theta'' +
\rho''}) |_{(\theta', \theta'') = (0,0)} . \end{equation}

The complexified moment map  is a  map
\begin{equation} \label{AAMMAP} \mu_{\C} \to (M \times M)_{\delta} \to
\C^m.
\end{equation} The invariance of $\mu$ under the
torus action implies that $\mu_{\C}(e^{i \theta} z, e^{i \theta}
w) = \mu_{\C}(z,w).$ The following Proposition will clarify the
discussion of critical point sets later on (see e.g. Lemma
\ref{USESMUCEQ}).

\begin{prop} \label{MUCEQ} For   $\delta$ sufficiently small so that $\mu_{\C}(z,w)$ is
well-defined,we have

\begin{enumerate}

\item $ \Im \mu_{\C}(z, e^{i \theta} z) = \frac{1}{2}
\nabla_{\theta} D(z, e^{i \theta} z). $

\item $\mu_{\C}(z, e^{i \theta} z) = \mu_{\C}(z,z)$ with $(z,e^{i
\theta} z) \in (M \times M)_{\delta}$ if and only if $e^{i \theta}
z = z.$

\end{enumerate}
\end{prop}

\begin{proof} The proof of the identity (1) is immediate from the definitions;  we only note that  the
diastasis function is a kind of real part, and that the imaginary
part originates  in the factor of $i$ in (\ref{MMDEFa}). One can
check the factors of $i$ in the Bargmann-Fock model, where
$\mu_{\C} (z, e^{i \theta} z) = e^{i \theta} |z|^2$ while $D(z,
e^{i \theta} z) =  2 (\cos \theta - 1) |z|^2 + 2 i (\sin \theta)
|z|^2$ (in vector notation) .

By (\ref{HESSDIA}), $D(z,w)$ has a strict global minimum at $w =
z$ which is non-degenerate. It is therefore isolated for each $z$.
Since its Hessian at $w = z$ is the identify with respect to
$\omega$, the isolating neighborhood has a uniform size as $z$
varies. Thus, there exists a $\delta > 0$ so that $\mu_{\C}(z,w) =
\mu_{\C}(z,z)$ in $(M \times M)_{\delta}$ if and only if $z = w$.
This is true both in the real analytic case and the
almost-analytic case.

\end{proof}

\subsection{Hilbert spaces of holomorphic sections}

On the `quantum level', a toric \kahler variety $(M, \omega)$
induces the sequence of  spaces $H^0(M, L^k)$ of holomorphic
sections of  powers of the   holomorphic toric line bundle $L$
with $c_1(L) =\frac{1}{2 \pi}  \left[\omega\right]$. The
$(\C^*)^m$ action lifts to $H^0(M, L^k)$ as a holomorphic
representation which is unitary on $\T$.  Corresponding to the
lattice points $\alpha \in k P$, there is a natural basis
$\{s_{\alpha}\}$ (denoted $\chi_{\alpha}^P$ in \cite{STZ}) of
$H^0(M, L^k)$ given by joint eigenfunctions of the $(\C^*)^m$
action.  It is well-known that the joint eigenvalues are precisely
the lattice points $\Z^m \cap k P$ in the $k$th dilate of $P$.  On
the open orbit $s_{\alpha}(z) = \chi_{\alpha}(z) e^k$ where $e$ is
a frame and where as above $\chi_{\alpha}(z) =  z^{\alpha} =
z_1^{\alpha_1} \cdots z_m^{\alpha_m}.$ Hence, the $s_{\alpha}$ are
referred to as monomials. For further background, we refer to
\cite{STZ}. A hermitian metric $h$ on $L$ induces Hilbert space
inner products (\ref{HILBDEF}) on $H^0(M, L^k)$.

As is evident from (\ref{PHK}), we will need formulae for the
monomials which are valid near $\dcal$. By (\ref{CChange}) and
(\ref{GAMMADEF}), we have
\begin{equation}
\chi_{\alpha^{j}}(z)=\eta_{j}(z)\chi_{v^{0}}(z),\quad z \in
(\C^{*})^{m},\end{equation} and by (\ref{GAMMATWDEF}) we then have
\begin{equation}
\label{monch} |\chi_{\alpha}(z)|^{2} =
|\eta^{\tilde{\Gamma}(\alpha)}|^2.
\end{equation}
 As mentioned above, for  simplicity of notation we  suppress the transformation
$\tilde{\Gamma}$ and coordinates $\eta$, and  we will use the
`orbit-slice' coordinates of (\ref{OS}). Thus, we  denote the
monomials cooresponding to lattice points $\alpha$ near a face $F$
by $(z')^{\alpha'} e^{ \langle (i \theta'' + \rho''/2), \alpha''
\rangle}$, where $\tilde{\Gamma}(\alpha) = (\alpha', \alpha'')$
with  $\alpha''$
 in the coordinate hyperplane corresponding under
$\tilde{\Gamma}$ to $F$ and with $\alpha'$ in the normal space.

\subsection{Examples: Bargmann-Fock and Fubini-Study models \label{BF}}

As mentioned above the Bargmann-Fock model is the linear model. It
plays a fundamental role in this article because it provides an
approximation for objects on any toric variety on balls of radius
$\frac{\log k}{\sqrt{k}}$ and also near $\dcal$. Although it and
the Fubini-Study model are elementary examples, we go over them
because the notation is used frequently later on.

The Bargmann-Fock models on $\C^m$ correspond to choices of a
positive definite
 Hermitian matrix $H$ on $\C^m$. A toric Bargmann-Fock model is
 one in which $H$ commutes with the standard $\T$ action, i.e., is
 a diagonal matrix. We denote its diagonal elements by $H_{j \bar{j}}$.   The \kahler metric on $\C^m$ is thus $i \ddbar
 \phi_{BF, H}(z)$ where the global \kahler potential is
 $$\phi_{BF, H}(z) = \sum_{j = 1}^m H_{j \bar{j}} |z_j|^2 = F(|z_1|^2, \dots, |z_m|^2), \;\; \mbox{with}\;\;
 F(y_1, \dots, y_m) = \sum_j H_{j \bar{j}} y_j. $$  For simplicity we often only consider the case
 $H = I$.
Putting $|z_j|^2 = e^{\rho_j}$ and using (\ref{MMDEF}), it follows
that  $\mu_{BF, H}(z_1, \dots, z_m) = (H_{1 \bar{1}} |z_1|^2,
\dots, H_{m \bar{m}} |z_m|^2): \C^m \to \R_+^m$ as in
(\ref{BFMM}). The  symplectic potential Legendre dual to
$\phi_{BF, H}$  is given by
\begin{equation}\label{BFVH}  u_{BF, H}(x) = - \phi_{BF, H}(\mu_{BF}^{-1}(x)) +2  \langle
\log \mu_{BF, H}^{-1}(x), x \rangle = - \sum_j x_j + \sum_{j =
1}^m x_j \log (\frac{x_j}{H_{j \bar{j}}}).  \end{equation}
 In this case, $G_{BF, H}$ is
the diagonal matrix with entries $\frac{1}{x_j H_{j \bar{j}}}$, so
$\det G_{BF, H} = \frac{1}{\det H}  \Pi_j \frac{1}{x_j}$.

The off-diagonal analytic extension of the \kahler potential in
the sense of (\ref{AAE}) is then
 $$\phi_{BF, H}(z, \bar{w}) = \sum_{j = 1}^m H_{j \bar{j}} z_j \bar{w}_j = F(z_1 \bar{w}_1, \dots,
 z_m \bar{w}_m)$$
 and in particular,
 $$\phi_{BF, H}(z, e^{i \theta} z) = \sum_{j = 1}^m H_{j \bar{j}} e^{i \theta_j} |z_j|^2 =  F(e^{i \theta_1} |z_1|^2, \dots,
 e^{i \theta} |z_m|^2).$$
 Henceforth we often write the the right side in the multi-index notation  $F_{\C}(e^{i \theta}
 |z|^2)$. We observe, as claimed in (\ref{MMDEFa}), that
 $\nabla_{\theta} F_{BF,\C}(e^{i \theta} |z|^2) |_{\theta = 0} = i
 \mu_{BF} (z).$

 Quantization of the Bargmann-Fock model with $H = I$  produces the Bargmann-Fock
 (Hilbert) space  $$\hcal^2(\C^m, (2 \pi)^{-m} k^m  e^{- k |z|^2} dz
\wedge d\bar{z})$$ of entire functions which are $L^2$ relative to
the weight $e^{- k |z|^2/2}$. It is infinite dimensional and a
basis is given by the monomials $z^{\alpha}$ where $\alpha \in
\R_+^m \cap \Z^m$. In \S \ref{BFNORMS} we compute their $L^2$
norms. For $H \not= I$ one uses the volume form $e^{- k \langle H
z, z \rangle} (i \ddbar \langle H z, z \rangle)^m/m! = e^{- k
\langle H z, z \rangle} (\det H) dz \wedge d\bar{z}. $

Toric Fubini-Study metrics provide compact models which are
similar to Bargmann-Fock models. In a local analysis we always use
the latter. A Fubini-Study metric on $\CP^m$  is determined by a
positive Hermitian form $H$ on $\C^{m+1}$ and a toric Fubini-Study
metric is a diagonal one $\sum_{j = 0}^m  H_{j \bar{j}} |Z_j|^2$.
In the affine chart $Z_0 \not= 0$ (e.g.) a local   Fubini-Study
\kahler potential is $\phi_{FS, H}(z_1, \dots, z_m) = \log (1 +
\sum_j h_{j \bar{j}}|z_j|^2)$ where $h_{j \bar{j}} = \frac{H_{j
\bar{j}}}{H_{0\bar{0}}}$. This is a valid \kahler potential near
$z = 0$ but of course has logarithmic singularities on the
hyperplane at infinity. The almost analytic extension of the
Fubini-Study \kahler potential is given in the affine chart by
$\log (1 + \sum_j h_{j \bar{j}}z_j \bar{w}_j)$. Thus
(\ref{MMDEFa}) asserts that
$$i \; \frac{\sum_j h_{j \bar{j}}
|z_j|^2 }{1 + \sum_j h_{j \bar{j}} |z_j|^2} = \nabla_{\theta} \log
(1 + \sum_j h_{j \bar{j}}e^{i \theta_j} |z_j|^2) |_{\theta = 0}.
$$

Quantization produces the Hilbert spaces $H^0(\CP^m, \ocal(k))$,
where  $\ocal(k) \to \CP^m$ is the $kth$ power of the hyperplane
section bundle. Sections lift to homogeneous holomorphic
polynomials on $\C^{m + 1}$, and correspond to lattice points in
$k \Sigma$ where $\Sigma$ is the unit simplex in $\R^m$.

\subsection{\label{LIN}Linearization of the Monge-Amp\`ere equation }

It is known that the Legendre transform linearizes the
Monge-Amp\`ere geodesic equation. Since it is important for this
article, we present a simple proof that does not seem to exist in
the literature.

\begin{prop} \label{LINEAR} Let $M_P^c$ be a toric variety. Then under the
Legendre transform $\phi \to u_{\phi}$, the complex Monge-Amp\'ere
equation on $\hcal_{\T}$  linearizes to the equation $u'' = 0$.
Hence  the Legendre transform of a geodesic $\phi_t$ has the form
  $u_t = u_0 + t(u_1 - u_0). $
  \end{prop}

  \begin{proof}
It suffices to show that the energy functional \begin{equation}
\label{EF} E = \int_0^1 \int_M \dot{\phi}_t^2 d\mu_{\phi_t} dt
\end{equation} is Euclidean on paths of symplectic potentials.  For each $t$
let us pushforward the integral $\int_M \dot{\phi}_t^2 d\mu_{\phi}
$ under the moment map $\mu_{\phi_t}$. The integrand is by
assumption invariant under the real torus action, so the
pushforward is a diffeomorphism on the real points.   The volume
measure $d\mu_{\phi_t} $ pushes forward to $dx$. The function
$\partial_t \phi_t(\rho)$ pushes forward to the function
$\psi_t(x) = \dot{\phi}_t(\rho_{x, t})$ where
$\mu_{\phi_t}(\rho_{x, t}) = x$.  By (\ref{SYMPOTDEF}), the
symplectic potential at time $t$ is
$$u_t(x) = \langle x, \rho_{x, t}  \rangle -
\phi_t(\rho_{x, t}).$$

We note that \begin{equation} \label{UDOTPHIDOT} \begin{array}{l}
\dot{u}_t = \langle x,
\partial_t \rho_{x, t} \rangle -
\dot{\phi}_t(\rho_{x, t}) - \langle \nabla_{\rho} \phi_t (\rho_{x,
t}),
\partial_t \rho_{x, t} \rangle.\end{array}
\end{equation}  The outer terms cancel, and thus, our integral is
just
$$\int_0^1 \int_P |\dot{u}_t|^2 dx. $$ Clearly the Euler-Lagrange
equations are linear.

  \end{proof}

\section{The functions $\pcal_{h^k}$ and $\QQ_{h^k}$ \label{PQ}}

We  now introduce the key players in the analysis, the norming
constants $\QQ_{h^k}(\alpha)$ (\ref{QHK})  and the dual constants
$\pcal_{h^k}(\alpha)$ of (\ref{PHKALPHA}).  The duality is given
in the following:

\begin{prop} \label{QPINV} We have:
$$Q_{h_k}(\alpha) = \frac{e^{ k u_{\phi}(\frac{\alpha}{k})} }{\pcal_{h^k}(\alpha)}, $$
\end{prop}

\begin{proof}

By (\ref{SYMPOTDEF}), it follows that \begin{equation} \label{LHS}
||s_{\alpha} (\mu_h^{-1}(\frac{\alpha}{k})) ||_{h^k}^2 =
|\chi_{\alpha}(\mu_h^{-1}(\frac{\alpha}{k}))|^2 e^{- k
\phi_h(\mu_h^{-1}(\frac{\alpha}{k}))} = e^{k
u_{\varphi_h}(\frac{\alpha}{k})}.
\end{equation}

\end{proof}

\begin{cor}\label{RP}
$$\rcal_k(t, \alpha)  = \frac{\left(\pcal_{h_0^k}(\alpha)\right)^{1 - t}
\left( \pcal_{h_1^k}(\alpha) \right)^t}{\pcal_{h_t^k}(\alpha)}
$$

\end{cor}

\begin{proof}

We need to show that   \begin{equation} \label{QVSP}
\begin{array}{l} \frac{\QQ_{h_t^k}(\alpha)}{
(\QQ_{h_0^k}(\alpha))^{1-t}(\QQ_{h_1^k}(\alpha))^{t}} =
\frac{\left(\pcal_{h_0^k}(\alpha)\right)^{1 - t} \left(
\pcal_{h_1^k}(\alpha) \right)^t}{\pcal_{h_t^k}(\alpha)}.
\end{array} \end{equation}
By Proposition \ref{QPINV}, the left side of  (\ref{QVSP}) equals
$$\begin{array}{l}
 \frac{|\chi_{\alpha}(\mu_t^{-1}(\frac{\alpha}{k}))|^2 e^{- k
\phi_t(\mu_t^{-1}(\frac{\alpha}{k}))}}{\pcal_{h_t^k}(\alpha)}
\times \left( \frac{\pcal_{h_0^k}(\alpha)
}{|\chi_{\alpha}(\mu_0^{-1}(\frac{\alpha}{k}))|^2 e^{- k
\phi_0(\mu^{-1}_0(\frac{\alpha}{k}))}} \right)^{1-t} \times \left(
\frac{\pcal_{h_1^k}(\alpha)
}{|\chi_{\alpha}(\mu_1^{-1}(\frac{\alpha}{k}))|^2 e^{- k
\phi_1(\mu^{-1}_1(\frac{\alpha}{k}))}} \right)^{t}
\end{array}$$
By (\ref{LHS}), the left side of (\ref{QVSP}) equals
$$\begin{array}{l}
= e^{k\left(u_t(\frac{\alpha}{k}) + (1 - t) u_0(\frac{\alpha}{k})
+ t u_1(\frac{\alpha}{k}) \right)} \times
\frac{\left(\pcal_{h_0^k}(\alpha)\right)^{1 - t} \left(
\pcal_{h_1^k}(\alpha) \right)^t}{\pcal_{h_t^k}(\alpha)}.
\end{array}$$
But $u_t(x) + (1 - t) u_0(x) + t u_1(x) = 0$ on a toric variety,
and this gives the stated equality.

\end{proof}

Further, we relate the full $\pcal_{h^k}(\alpha, z)$ to the \szego
kernel. The \szego (or Bergman) kernels of a positive Hermitian
line bundle $(L, h) \to (M, \omega)$ over a \kahler manifold are
the kernels of the orthogonal projections $\Pi_{h^k}: L^2(M, L^k)
\to H^0(M, L^k)$ onto the spaces of holomorphic sections with
respect to the inner product $Hilb_k(h)$ (\ref{HILBDEF}). Thus, we
have
\begin{equation} \Pi_{h^k} s(z) = \int_M \Pi_{h^k}(z,w) \cdot s(w)
\frac{\omega_h^m}{m!}, \end{equation} where the $\cdot$ denotes
the $h$-hermitian inner product at $w$.
 Let $e_L$ be a local holomorphic  frame for $L \to M$ over an
open set $U \subset M$ of full measure,  and let $\{s^k_j=f_j
e_L^{\otimes k}:j=1,\dots,d_k\}$ be an orthonormal basis for
$H^0(M,L^k)$ with $d_k = \dim H^0(M, L^k)$.  Then the \szego
kernel can be written in the form
\begin{equation}\label{szego}  \Pi_{h^k}(z, w): = F_{h^k}
(z, w)\,e_L^{\otimes k}(z) \otimes\overline {e_L^{\otimes
k}(w)}\,,\end{equation} where
\begin{equation}\label{FN}F_{h^k}(z, w)=
\sum_{j=1}^{d_k}f_j(z) \overline{f_j(w)}\;.\end{equation}

  Since
the \szego kernel is a section of the bundle $(L^k) \otimes
(L^k)^* \to M \times M$, it often simplifies the analysis  to lift
it to a scalar kernel $\hat{\Pi}_{h^k}(x,y)$ on  the associated
unit circle bundle $X \to M$ of $(L, h)$.  Here, $X =
\partial D^*_h $ is the boundary of the unit disc bundle with respect to $h^{-1}$  in the dual line
bundle $L^*$. We use local product coordinates $x = (z, t) \in M
\times S^1$ on $X$
 where $x=e^{i t }\|e_L(z)\|_he_L^*(z)\in X$.   To avoid confusing the $S^1$ action on $X$ with
the $\T$ action on $M$ we use $e^{i t}$ for the former and $e^{i
\theta}$ (multi-index notation) for the latter.  We note that the
$\T$ action lifts to $X$ and combines with the $S^1$ action to
produce a $(S^1)^{m + 1}$ action.  We refer to \cite{Z,SZ,Z2} for
background and for more on lifting the \szego kernel of a toric
variety.

 The equivariant lift of
a section $s=fe_L^{\otimes k}\in H^0(M,L^k)$ is given explicitly
by
\begin{equation}\label{lift}\hat s(z,t) =
e^{ik t} \|e_L^{\otimes k}\|_{h^k} f(z) = e^{k\left[-\half \phi(z)
+i t \right]} f(z)\;.\end{equation} The \szego kernel thus lifts
to $X \times X$ as the scalar kernel
\begin{equation}\label{szegolift} \hat{\Pi}_k(z,t ;w,t') =
e^{k\left[-\half \phi (z)-\half \phi (w) +i(t -t')\right]}F_k(z,
w)\;.\end{equation} Since it is $S^1$- equivariant we often put $t
= t' = 0$.

\begin{prop} \label{INTEGRAL}

  We have

$$ \pcal_{h^k}(\alpha, z) = (2 \pi)^{-m} \int_{\T} \hat{\Pi}_{h^k}(e^{i \theta} z, 0;   z, 0) e^{- i \langle
\alpha,\theta \rangle} d\theta.
$$

\end{prop}

\begin{proof} We recall that
$\chi_{\alpha}(z) = z^{\alpha}$ is the local representative of
$s_{\alpha}$ in the open orbit with respect to an invariant frame.
Since $\{\frac{\chi_{\alpha}}{\sqrt{\QQ_{h^k}(\alpha)}}\}$ is the
local expression of an  orthonormal basis, we have
$$F_{h^k}(z,w) = \sum_{\alpha \in k P \cap \Z^m} \frac{\chi_{\alpha}(z)
\overline{\chi_{\alpha}(w)} }{\QQ_{h^k}(\alpha)} $$ hence
$$\hat{\Pi}_{h^k}(z,0; w, 0) = \sum_{\alpha \in k P \cap \Z^m} \frac{\chi_{\alpha}(z)
\overline{\chi_{\alpha}(w)}e^{- k (\phi(z) + \phi(w))/2}
}{\QQ_{h^k}(\alpha)}.
$$ It follows that
$$\Pi_{h^k}(e^{i \theta} z, 0;  z, 0) = \sum_{\alpha \in k P \cap \Z^m}
\frac{|\chi_{\alpha}(z)|^2 e^{- k \phi(z)}  e^{ i \langle \alpha,
\theta \rangle}}{\QQ_{h^k}(\alpha)}.
$$
Integrating against $e^{- i \langle \alpha, \theta \rangle}$ sifts
out the $\alpha$ term.

\end{proof}

\begin{cor} \label{INTEGRALCOR}

We have

\begin{equation}  \pcal_{h^k}(\alpha) = (2 \pi)^{-m}
\int_{\T} \hat{\Pi}_{h^k}(e^{i \theta} \mu_h^{-1}(
\frac{\alpha}{k}), 0; \mu_h^{-1}(\frac{\alpha}{k}), 0) e^{-  i
\langle \alpha,\theta \rangle} d\theta.
\end{equation}

\end{cor}

\subsubsection{\label{BFNORMS}Bargmann-Fock model}

 As discussed in \S \ref{BF}, the Hilbert space in this model has the orthogonal
 basis $z^{\alpha}$ with $\alpha \in \R_+^m \cap \Z^m$.  The Bargmann-Fock
norming constants when $H = I$ are given by
$$Q_{h_{BF}^k}(\alpha) = k^{-|\alpha| - m} \alpha!,  \;\;\; (\alpha! := \alpha_1! \cdots \alpha_m!) $$ It follows
that an orthonormal basis of holomorphic monomials  is given by
$\{k^{\frac{|\alpha| + m}{2}} \frac{z^{\alpha}}{\sqrt{\alpha!}}
\}.$

We therefore have
\begin{equation} \label{PCALZBF} \frac{|s_{\alpha}(z)|^2_{h_{BF}^k}}{\QQ_{h_{BF}^k}(\alpha)} =
k^{|\alpha| + m} \frac{|z^{\alpha}|^2}{\alpha!} e^{- k |z|^2},
\end{equation}
and in particular,
\begin{equation} \label{PCALBF} \pcal_{h_{BF}^k}(\alpha) = k^m
e^{-|\alpha|} \frac{\alpha^{\alpha}}{\alpha!}, \end{equation}
where $\alpha^\alpha = 1$ when $\alpha = 0$.  Here, we use that
$u_{BF}(\frac{\alpha}{k}) = \frac{\alpha}{k} \log \frac{\alpha}{k}
- \frac{\alpha}{k},  $ so that $ e^{ k u_{BF}(\frac{\alpha}{k})} =
e^{- |\alpha|} \frac{k^{-|\alpha|}}{\alpha^{\alpha}} $ and that
$\QQ_{h_{BF}^k}(\alpha) = k^{-m -|\alpha|} \alpha!$. We observe
that $\pcal_{h_{BF}^k}(\alpha)$ depends on $k$ only through the
factor $k^m$.

Precisely the same formula holds if we replace $I$ by a positive
diagonal $H$ with elements $H_{j \bar{j}}$. By  a  change of
variables, $\QQ_{h_{BF, H}^k}(\alpha) = \Pi_{j = 1}^m H_{j
\bar{j}}^{- \alpha_j} \QQ_{h_{BF}^k}(\alpha)$, and also by
(\ref{BFVH}) $u_{BF, H}(x) = u_{BF}(x) + \sum_j x_j \log H_{j
\bar{j}}.$ Hence, by Proposition \ref{QPINV},
$$\pcal_{h_{BF, H}^k}(\alpha) = \pcal_{h_{BF}^k}(\alpha) \Pi_{j = 1}^m H_{j \bar{j}}^{- \alpha_j} e^{ \sum_j \alpha_j \log H_{j \bar{j}}} =
\pcal_{h_{BF}^k}(\alpha). $$

\subsubsection{ $\CP^m$}

In the Fubini-Study model, a basis of $H^0(\CP^m, \ocal(k))$ is
given by monomials with $\alpha \in k \Sigma$ (see \S \ref{BF}),
and the norming constants are given by
\begin{equation}\label{l2}
\QQ_{h_{FS}^k}(\alpha) = {k \choose \alpha}:=  {k  \choose
\alpha_1, \dots, \alpha_m}^{-1}.
\end{equation}
Recall that multinomial coefficients are defined for $\alpha_1 +
\cdots + \alpha_m \leq k$ by
$${k \choose \alpha_1, \dots, \alpha_m} = \frac{k!}{\alpha_1! \cdots \alpha_m! (k - |\alpha|)!}, $$
where as above, $|\alpha| = \alpha_1 + \cdots + \alpha_m$.

We further have $|s_{\alpha}(z)|_{h_{FS}^k}^2 = |z^{\alpha}|^2
e^{- k \log (1 + |z|^2)} $ and therefore,

$$\begin{array}{l}
  \pcal_{h_{FS}^k}(
\alpha, z) = {k  \choose  \alpha_1, \dots, \alpha_m}
|z^{\alpha}|^2 e^{- k \log (1 + |z|^2)},
\end{array}$$
and since
$$ e^{-
ku_{FS}(\frac{\alpha}{k})} =
|s_{\alpha}(\mu_{FS}^{-1}(\frac{\alpha}{k}))|_{h_{FS}^k}^2 =
(\frac{\alpha}{k})^{\alpha} (1 - \frac{|\alpha|}{k})^{k -
|\alpha|}
$$ we have
$$ \pcal_{h_{FS}^k}(\alpha) = \frac{k!}{\alpha_1! \cdots \alpha_m! (k - |\alpha|)!}
(\frac{\alpha}{k})^{\alpha} (1 - \frac{|\alpha|}{k})^{k -
|\alpha|}.
$$

\section{\szego kernel of a toric variety \label{SZKER}}

We will use Proposition \ref{INTEGRAL} to  reduce the  joint
asymptotics of $\pcal_{h^k}\alpha, z)$ in $(k, \alpha)$ to
asymptotics of the Bergman-\szego kernel off the diagonal. We now
review  some general facts about diagonal and off-diagonal
expansions of these kernels, for which complete details can be
found in  \cite{SZ}, and we also consider some special properties
of toric Bergman-\szego kernels which are very convenient for
calculations; to some extent they derive from \cite{STZ}, but the
latter only considered \szego kernels for powers of Bergman
metrics.

The \szego kernels $\hat{\Pi}_{h^k} (x,y)$ are the Fourier
coefficients of the total \szego projector
$\hat{\Pi}_h(x,y):\lcal^2(X)\to \hcal^2(X)$, where $\hcal^2(X)$ is
the Hardy space of boundary values of holomorphic functions on
$D^*$  (the kernel of $\dbar_b$ in $L^2(X)$). Thus,
$$\hat{\Pi}_{h^k}(x,y)=\frac 1{2\pi}\int_0^{2\pi}  e^{-ik t
}\hat{\Pi}_h(e^{i t }x,y)\,dt. $$ The properties we need of
$\hat{\Pi}_{h^k}(x, y)$ are  based on the  Boutet de
Monvel-Sj\"ostrand construction of an oscillatory integral
parametrix for the \szego kernel (\cite{BSj}):
\begin{equation}\label{oscint}\begin{array}{c}\hat{\Pi} (x,y) =
S(x,y)+E(x,y)\;,\\[12pt] \mbox{with}\;\; S(x,y)=
\int_0^{\infty} e^{i \lambda \psi(x,y)} s(x,y, \lambda )
d\lambda\,, \qquad E(x,y)\in \ccal^\infty(X \times
X)\,.\end{array}
\end{equation}  The
phase function $\psi$ is of positive type and  is given in the
local coordinates above by
\begin{equation}\label{LITTLEPHASE}  \psi(z, t;  w, t') = \frac{1}{i} \left[1 -
e^{ \phi(z,w) - \frac{1}{2}(\phi(z) + \phi(w))} e^{i (t -
t')}\right]\;. \end{equation} Here, $\phi(z,w)$ is the almost
analytic extension of the local \kahler potential with respect to
the frame, i.e., $h = e^{- \phi(z)}$; see  (\ref{AAE}) for the
notion of almost analytic extension. The amplitude $s\big( z, t; w
, t', \lambda \big) $ is a semi-classical amplitude as in
\cite{BSj} (Theorem 1.5), i.e., it admits a polyhomogeneous
expansion  $s \sim \sum_{j = 0}^{\infty} \lambda^{m -k}
s_j(x,y)\in S^m(X\times X\times \R^+)$.

The phase $\psi(z, t; w, t')$ is the generating function for the
graph of the identity map along the symplectic cone $\Sigma
\subset T^*X$ defined by $\Sigma = \{(x, r \alpha_x): r > 0\}$
where $\alpha_x$ is the Chern connection one form. Hence the
singularity of $\hat{\Pi}(x,y)$ only occurs on the diagonal and
the symbol $s$ is understood to be supported in a small
neighborhood $(M \times M)_{\delta}$ of the anti-diagonal. It will
be useful to make the cutoff explicit by introducing  a smooth
cutoff function $\chi(d(z,w))$ where $\chi$ is a smooth even
function on $\R$ and $d(z,w)$ denotes the distance between $z, w$
in the base \kahler metric.

As above, we  denote the $k$-th Fourier coefficient of these
operators relative to the $S^1$ action by $\hat{\Pi}_{h^k} =
S_{h^k} + E_{h^k}$. Since $E$ is smooth, we have $E_{h^k}(x,y) =
O(k^{-\infty})$, where $O(k^{-\infty})$ denotes a quantity which
is uniformly $O(k^{-n})$ on $X\times X$ for all positive $n$.
Hence  $E_{h^k}(z,w)$ is negligible for all the calculations and
estimates of this article, and further it is only necessary to use
a finite number of terms of the symbol $s$. For simplicity of
notation, we will use the entire symbol.

It follows that (with $x = (z, t), y = (w, 0)$ and with $\chi(d(z,
w))$ as above ),
\begin{equation} \label{OSC}
\begin{array}{lll} \hat{\Pi}_{h^k}(x,y) & = & S_{h^k} (x,y) +
O(k^{- \infty})  \\ &&
\\ &=&   k \int_0^{\infty} \int_0^{2\pi}
 e^{ i k \left( -t + \lambda \psi( z, t;
w, 0)\right)} \; \chi(d(z, w)) \;s\big( z, t; w , 0, k
\lambda\big) dt  d \lambda + O( k^{-\infty}) \end{array}
\end{equation} The integral is a damped complex oscillatory
integral since (\ref{HESSDIA}) implies that \begin{equation}
\label{DAMPED1} \Im \psi(x,y) \geq C d(x, y)^2, \;\; (x, y \in X),
\end{equation} for $(x,y)$ sufficiently close to the diagonal, as
one sees by Taylor expanding the phase around the diagonal (cf.
\cite{BSj}, Corollary 1.3). It follows from (\ref{OSC}) and from
(\ref{DAMPED1}) that the \szego kernel $\Pi_{h^k}(z,w)$ on $M$ is
`Gaussian' in small balls $d(z,w) \leq \frac{\log k}{\sqrt{k}}$,
i.e., \begin{equation} \label{GAUSSIAN}
 | \; \hat{\Pi}_{h^k}(z, \phi; w,
\phi') |\;  \leq C k^m e^{- k d(z, w)^2} + O(k^{- \infty}),\;\;
(\mbox{when}\; d(z,w) \leq  \frac{\log k}{\sqrt{k}}),
\end{equation} and on the complement $d(z,w)
\geq  \frac{\log k}{\sqrt{k}}$ it is rapidly decaying.  This rapid
decay  can be improved to  long range (sub-Gaussian) exponential
decay off the diagonal given by the global Agmon estimates,
\begin{equation} \label{AGMON}
| \; \hat{\Pi}_{h^k}(z, \phi; w, \phi') |\;  \leq C  k^m e^{-
\sqrt{k} d(z, w)}.
\end{equation} We refer to \cite{Chr, L} for background and
references.

It is helpful to  eliminate the integrals in (\ref{OSC}) by
complex stationary phase. Expressed in a local frame and local
coordinates  on $M$, the result is

\begin{prop} \label
{PIKZW} Let  $(L, h)$ be a  $C^{\infty}$  positive hermitian line
bundle, and let $h = e^{-\phi}$ in a local frame. Then in this
frame, there exists a semi-classical amplitude $A_k(z, w) \sim k^m
a_0(z,w) + k^{m-1} a_1(z, w) + \cdots $  in the parameter $k^{-1}$
such that,
$$\hat{\Pi}_{h^k} (z, 0; w, 0) = e^{ k (\phi(z,w) - \frac{1}{2}(\phi(z) +
\phi(w)))} \chi_k(d(z,w)) \; A_k(z,w) + O(k^{- \infty}), $$ where
as above,  $\chi_k(d(z,w)) = \chi(\frac{k^{1/2}}{\log k} d(z,w)) $
is a cutoff to $ \frac{\log k}{\sqrt{k}}$ - neighborhood of the
diagonal.
\end{prop}

\begin{proof}

This follows from the  scaling asymptotics of \cite{SZ} or from
Theorem 3.5 of \cite{BerSj}.  We refer there for a detailed proof
of the scaling asymptotics and only sketch a somewhat intuitive
proof.

 The integral (\ref{OSC}) is a complex oscillatory
integral with a positive complex phase. With no loss of generality
we may set $\phi' = 0$.  Taking the $\lambda$-derivative gives one
critical point equation
$$1 -
e^{ \phi(z,w) - \frac{1}{2}(\phi(z) + \phi(w))} e^{i t} = 0$$ and
the critical point equation in $t$ implies that $\lambda = 1$.
 The $\lambda$-critical point  equation can only be satisfied for complex $
t$ with imaginary part equal to the negative of the `Calabi
diastasis function' (\ref{DIASTASIS}), i.e.,
$$\Im  t = D(z,w),$$ and with real  part equal to $- \Im \phi(z,w). $ To
obtain asymptotics, we therefore have to deform the integral over
$S^1$ to the circle $|\zeta| = e^{- D(z,w)}$. Since $d(z,w) \leq C
\frac{\log k}{\sqrt{k}}$ by assumption, the deformed contour is a
slightly re-scaled circle by the amount $\frac{\log k}{\sqrt{k}}$;
in the complete proofs, the contour is held fixed and the
integrand is rescaled  as in \cite{SZ}. The contour deformation is
possible modulo an error $O(k^{-M})$ of arbitrarily rapid
polynomial decay because the integrand may be replaced by  the
parametrix (up to any order in $\lambda$) which has a holomorphic
dependence on the $\C^*$ action on $L^*$, hence in $e^{i \theta}$
to a neighborhood of $S^1$ in $\C$. This is immediately visible in
the phase and with more work is visible in the amplitude (this is
the only incompleteness in the proof; the statement can be derived
from  \cite{SZ} and also \cite{Chr}). We need to use a cutoff to a
neighborhood of the diagonal of $M \times M$, but it may be chosen
to be independent of $\theta$.

 By deforming the circle of integration from the unit circle to $|\zeta| = e^{D(z,w)}$
 and then changing variables $t \to t + i D(z,w)$ to bring it back to the unit circle,  we obtain
\begin{equation} \label{OSC2}
\hat{\Pi}_{h^k}(x,y) \sim k \int_0^{\infty} \int_0^{2\pi}
 e^{ i k \left( -t - i D(z,w) -  \lambda \psi( z, t + i D(z,w);
w, 0)\right)} s\big( z, t+ i D(z,w); w , 0, k \lambda\big) dt d
\lambda \;\; \mbox{mod}\;\; k^{-\infty}. \end{equation} The new
critical point equations state that $\lambda = 1$ and that $ e^{i
\Im \phi (z,w)} e^{i t} = 1. $ The calculation shows that $\psi =
0$ on the critical set so the phase factor on the critical set
equals $e^{\phi(z,w) - \frac{1}{2}(\phi(z)+ \phi(w))}$.  The
Hessian of the phase on the critical set is
$ \left( \begin{array}{ll} 0 & 1  \\
1 & i  \end{array} \right)$ as in the diagonal case and the rest
of the calculation proceeds as in \cite{Z}. (As mentioned above, a
complete proof  is contained in  \cite{SZ}).

\end{proof}

\subsection{Toric Bergman-\szego kernels}

 In the toric case, we may simplify the expression for the \szego kernels in Proposition
 \ref{PIKZW} using the almost analytic extension (cf. \S \ref{AAEa}); (\ref{AAE})) of the \kahler
 potential $\phi(z,w)$ to $M \times M$, which has the form
 \begin{equation} \label{AAEF} F_{\C} (z \cdot \bar{w}) = \; \mbox{the almost analytic extension of }\;
     F(|z|^2)\; \mbox{to}\; M \times M.  \end{equation}
The almost analytic extension  will be illustrated in some
analytic examples  below, where it is the analytic continuation.

 Thus, we have:
\begin{prop}\label{SZKTV}  For any hermitian  toric positive line bundle over a toric
variety, the \szego kernel for the metrics $h_{\phi}^k$ have the
asymptotic expansions in a local frame on $M$,
$$\Pi_{h^k}(z, w) \sim e^{k \left(F_{\C} (z \cdot \bar{w}) - \frac{1}{2} (F(||z||^2)
+ F(||w||^2) \right) } A_k(z,w) \;\; \mbox{mod} \; k^{- \infty},
$$ where $A_k(z,w) \sim k^m \left(a_0(z,w) + \frac{a_1(z,w)}{k} + \cdots\right) $ is a semi-classical symbol of order $m$. \end{prop}

As an example,  the  Bargmann-Fock(-Heisenberg) \szego kernel with
$ k = 1$  and $H = I$ is given (up to a constant $C_m$ depending
only on the dimension) by
$$\hat{\Pi}_{h_{BF}}(z, \theta, w, \phi) =  e^{z \cdot \bar{w} -  \frac{1}{2}(|z|^2 +
|w|^2)} e^{i( \theta - \phi)} = \sum_{\alpha \in {\bf N}^n}
\frac{z^{\alpha} \overline{w^{\alpha}}}{\alpha!} e^{-
\frac{1}{2}(|z|^2 + |w|^2)} e^{i( \theta - \phi)}. $$ The higher
\szego kernels  are Heisenberg dilates of this kernel:
\begin{equation}\label{szegoheisenberg} \hat{\Pi}_{h_{BF}^k}(x,y)  =\frac{1}{\pi^m} k^m
e^{i k (t-s )} e^{ k(\zeta\cdot\bar \eta -\half |\zeta|^2
-\half|\eta|^2) }, \end{equation} where $x=(\zeta,t)\,,\
y=(\eta,s)\,$. In this case, the almost analytic extension is
analytic and  $F_{BF, \C}(z,w) = z \cdot \bar{w}$.

A second example is the Fubini-Study \szego kernel on $\ocal(k)$,
which lifts to $S^{2m - 1} \times S^{2m - 1}$ as
\begin{equation}\label{szegosphere} \hat{\Pi}_{h_{FS}^k}(x,y)=\sum_J
\frac{(k+m)!}{\pi^mj_0!\cdots j_m!}x^J \bar y^J =  \frac{(k+m)!}
{\pi^mk!}\langle x,y\rangle^k\,.\end{equation} Recalling that $x =
e^{i \theta} \frac{e(z)}{||e(z)||}$ in a local frame $e$ over an
affine chart, the \szego kernel has the local form on $\C^m \times
\C^m$ of
\begin{equation}\label{szegoCPM} \hat{\Pi}_{h_{FS}^k}(z, 0; w, 0)= \frac{(k+m)!}
{\pi^mk!} e^{k \log \frac{(1 + z \cdot \bar{w})}{\sqrt{1 + |z|^2}
\sqrt{1 + |w|^2}}} \,.\end{equation} Thus, $F_{FS, \C}(z,w) = \log
(1 + z \cdot \bar{w}). $

\subsection{Asymptotics of derivatives of toric  Bergman/\szego kernels \label{BS}}

One of the key ingredients in  of Theorem \ref{SUM} is the
asymptotics of derivatives of the contracted Bergman/\szego kernel
\begin{equation} \label{CONTSZEGO} \Pi_{h_t^k}(z,z) = F_{h_t^k}(z,z) ||e_L^k(z)||_{h^k}^2 =
\hat{\Pi}_{h^k}(z, 0; z, 0) \end{equation} in $(t,z)$. (The
notation is slightly ambiguous since in (\ref{szego}) it is used
for the un-contracted kernel, but it is standard and we hope no
confusion will arise since one is scalar-valued  and the other is
not.)  These derivatives allow us to make simple comparisions to
derivatives of $\phi_k(t, z)$. Since we ultimately interested in
$C^k$ norms we need asymptotics of derivatives with respect to
non-vanishing vector fields. We can use the vector fields
$\frac{\partial}{\partial \rho_j}$ away from $\dcal$ and the
vector fields $\frac{\partial}{\partial r_j}$ near $\dcal$. The
calculations are very similar, but we carry them both out in some
detail here. Later we will tend to suppress the calculations with
$\frac{\partial}{\partial r_j}$ to avoid duplication; the reader
can check in this section that the calculations and estimates are
valid.

Only the leading coefficient and the order of asymptotics are
relevant. The undifferentiated  diagonal asymptotics are of the
following form: for any $h \in P(M, \omega)$,
\begin{equation} \label{TYZ}  \Pi_{h^k}(z,z) = \sum_{i=0}^{d_k} ||s_i(z)||_{h_k}^2 = a_0
k^m + a_1(z) k^{m-1} + a_2(z) k^{m-2} + \dots \end{equation} where
$a_0$ is constant and as above $d_k + 1 = \dim H^0(M, L^k)$.

 We first consider derivatives with respect
to $\rho$. Calculating $\rho$ derivatives of
$\Pi_{h^k}(e^{\rho/2}, e^{\rho/2})$ is equivalent to calculating
$\theta$-derivatives of $\Pi_{h_t^k}(e^{ i \theta} z,  z)$. Using
(\ref{MMDEFa}) we have
$$\Pi_{h_t^k}(e^{ i \theta} z, z)=
\sum_{\alpha \in k P \cap \Z^m} \frac{e^{i \langle \alpha, \theta
\rangle} |z^{\alpha}|^2 e^{- k F_t(e^{i \theta}
|z|^2)}}{\QQ_{h_t^k}(\alpha)}.$$  The results are globally valid
but are not useful near $\dcal$ since on each stratum some of the
vector fields generating the $(\C^*)^m$ action vanish.

In the following, we use the tensor product notation
$(\frac{\alpha}{k} - \mu_t(e^{\rho/2}))^{\otimes 2}_{ij}$ for
$(\frac{\alpha_i}{k} - \mu_t(e^{\rho/2})_i) (\frac{\alpha_j}{k} -
\mu_t(e^{\rho/2})_j).$

\begin{prop} \label{COMPAREPIT} For $i, j = 1, \dots, m$ we have,

\begin{enumerate}

\item $k^{-m} \sum_{\alpha \in k P \cap \Z^m} (\frac{\alpha}{k} -
\mu_t(e^{\rho/2})) \frac{e^{\langle \alpha, \rho \rangle  - k
\phi_t(e^{\rho/2})} }{\QQ_{h_t^k}(\alpha)} = O(k^{-2}); $

\item $\frac{1}{\Pi_{h_t^k}(z,z)} \left( - \sum_{\alpha \in k P
\cap \Z^m}
 (\frac{\partial}{\partial t} \log  \QQ_{h_t^k}(\alpha))\;  \frac{e^{\langle \alpha,
\rho \rangle  - k \phi_t(e^{\rho/2})}
}{\QQ_{h_t^k}(\alpha)}\right) - k \frac{\partial}{\partial t}
\phi_t  = O(k^{-1}); $

\item $\frac{1}{\Pi_{h_t^k}(z,z)} \left( k^2  \sum_{\alpha \in k P
\cap \Z^m} (\frac{\alpha}{k} - \mu_t(e^{\rho/2}))^{\otimes 2}_{ij}
\frac{e^{\langle \alpha, \rho \rangle  - k \phi_t(e^{\rho/2})}
}{\QQ_{h_t^k}(\alpha)} \right) - k \frac{\partial^2
\phi_t}{\partial \rho_i
\partial \rho_j}   = O(k^{-1}) $;

\item $\frac{1}{\Pi_{h_t^k}(z,z)} \left(k \sum_{\alpha \in k P \cap
\Z^m } (\frac{\alpha}{k} - \mu_t(e^{\rho/2}))_i
(\frac{\partial}{\partial t} \log \QQ_{h_t^k}(\alpha))
\frac{e^{\langle \alpha, \rho \rangle - k \phi_t(e^{\rho/2})}
}{\QQ_{h_t^k}(\alpha)}\right) - k \frac{\partial^2
\phi_t}{\partial \rho_i
\partial t}  = O(k^{-1}) $.

\end{enumerate}
\end{prop}

\begin{proof}

 To prove (1), we differentiate and use (\ref{MMDEF})-(\ref{MMDEFa}) and  (\ref{TYZ}) to obtain
$$\begin{array}{lll} O(k^{m - 1}) =  \nabla_{\rho} \Pi_{h_t^k}(e^{\rho/2}, e^{ \rho/2} ) & = &
k \sum_{\alpha \in k P \cap \Z^m} (\frac{\alpha}{k} -
\mu_t(e^{\rho/2})) \frac{e^{\langle \alpha, \rho \rangle  - k
\phi_t(e^{\rho/2})} }{\QQ_{h_t^k}(\alpha)}.
\end{array}$$

To prove (2) we differentiate $$\log \Pi_{h_t^k}(e^{\rho/2}, e^{
\rho/2} ) = \log \sum_{\alpha \in k P \cap \Z^m}
  \frac{e^{\langle \alpha, \rho \rangle  - k
\phi_t(e^{\rho/2})} }{\QQ_{h_t^k}(\alpha)}$$ with respect $t$ to
produce the left side. Since the leading coefficient of
(\ref{TYZ}) is independent of $t$, the $t$ derivative has the
order of magnitude of the right side of (2).

To prove (3), we take a second  derivative of (1)  in $\rho$ (or
$\theta$) to get
$$\begin{array}{lll} \nabla_{\rho}^2  \Pi_{h_t^k}(e^{\rho/2}, e^{ \rho/2}
)&&=  - k \nabla \mu_t(e^{\rho/2})  \Pi_{h_t^k}(e^{\rho/2}, e^{\rho/2}) \\&& \\
&&+   k^2 \sum_{\alpha \in k P \cap \Z^m} (\frac{\alpha}{k} -
\mu_t(e^{\rho/2}))^{\otimes 2} \frac{e^{\langle \alpha, \rho
\rangle - k \phi_t(e^{\rho/2})} }{\QQ_{h_t^k}(\alpha)}.
\end{array}$$ Then (3) follows from (\ref{TYZ}) and the fact that
$\nabla \mu_t(e^{\rho/2})) = \nabla^2 \phi$. Similar calculations
show (4).

\end{proof}

In our applications, we actually need asymptotics of logarithmic
derivatives. They follow in a straightforward way  from
Proposition \ref{COMPAREPIT}, using that $\Pi_{h^k}(z,z) \sim
k^m$.  We record the results for future reference.

\begin{prop} \label{ONESZEGODER} We have:
\begin{itemize}

\item $\frac{1}{k} \nabla_{\rho} \log \sum_{\alpha \in k P \cap
\Z^m} \frac{|S_{\alpha}(z)|^2_{h_t^k}}{\QQ_{h_t^k}(\alpha)} =
\frac{\sum_{\alpha }(\frac{\alpha}{k}  - \mu_t(z))\;\;
 \frac{e^{\langle\alpha, \rho
\rangle}}{\QQ_{h_t^k}(\alpha)}}{\left(\sum_{\alpha}
\frac{e^{\langle\alpha, \rho
\rangle}}{\QQ_{h_t^k}(\alpha)}\right)} = O(\frac{1}{k^2}) $
\medskip

\item $\frac{1}{k}\frac{\partial} {\partial t}\log \sum_{\alpha
\in k P \cap \Z^m
}
\frac{|S_{\alpha}(z)|^2_{h_t^k}}{\QQ_{h_t^k}(\alpha)} =
\frac{\sum_{\alpha}
\partial_t \log  \left(
\frac{1}{\QQ_{h_t^k}(\alpha)} \right) \frac{e^{\langle\alpha, \rho
\rangle}}{\QQ_{h_t^k}(\alpha)} }{\left(\sum_{\alpha}
\frac{e^{\langle\alpha, \rho
\rangle}}{\QQ_{h_t^k}(\alpha)}\right)}  - \frac{\partial \phi_t}{
\partial t} =  O(\frac{1}{k^2}). $

\end{itemize}

\end{prop}

\begin{prop} \label{TWOSZEGODER} We have:
\begin{enumerate}

\item $\frac{1}{k} \nabla^2_{\rho} \log \sum_{\alpha \in k P \cap
\Z^m} \frac{|S_{\alpha}(z)|^2_{h_t^k}}{\QQ_{h_t^k}(\alpha)} =
\frac{1}{ k}  \sum_{\alpha, \beta}(\alpha - \beta)^{\otimes 2}\;\;
 \frac{e^{\langle\alpha, \rho
\rangle}}{\QQ_{h_t^k}(\alpha)}\frac{e^{\langle \beta, \rho
\rangle}}{\QQ_{h_t^k}(\beta)}{\left(\sum_{\alpha}
\frac{e^{\langle\alpha, \rho
\rangle}}{\QQ_{h_t^k}(\alpha)}\right)^{-2}} -
\left(\frac{\partial^2 \phi_t}{\partial \rho_i \partial
\rho_j}\right) = O(\frac{1}{k^2}) $
\medskip

\item $\frac{1}{k}\frac{\partial} {\partial t} \nabla_{\rho} \log
\sum_{\alpha \in k P \cap \Z^m}
\frac{|S_{\alpha}(z)|^2_{h_t^k}}{\QQ_{h_t^k}(\alpha)} =
\frac{1}{k}
 \frac{\sum_{\alpha, \beta}(\alpha - \beta )\;\;
\partial_t \log  \left(
\frac{\QQ_{h_t^k} (\beta)}{\QQ_{h_t^k} (\alpha)} \right)
\frac{e^{\langle\alpha, \rho
\rangle}}{\QQ_{h_t^k}(\alpha)}\frac{e^{\langle \beta, \rho
\rangle}}{\QQ_{h_t^k}(\beta)} }{\left(\sum_{\alpha}
\frac{e^{\langle\alpha, \rho
\rangle}}{\QQ_{h_t^k}(\alpha)}\right)^2}  - \left(\frac{\partial^2
\phi_t}{\partial \rho_i \partial t}\right) = O(\frac{1}{k^2}). $

 \item
$$\begin{array}{l} \frac{1}{k}\frac{\partial^2} {\partial t^2}\log
\sum_{\alpha \in k P \cap \Z^m }
\frac{|S_{\alpha}(z)|^2_{h_t^k}}{\QQ_{h_t^k}(\alpha)} \\ \\  =
\frac{1}{k} \frac{\sum_{\alpha, \beta} \left( \partial_t^2 \log
\frac{1}{\QQ_{h_t^k(\alpha)}} + \;\;(\partial_t \log
\frac{1}{\QQ_{h_t^k}}) (\partial_t \log  ( \frac{\QQ_{h_t^k}
(\beta)}{\QQ_{h_t^k} (\alpha)}) \right) \frac{e^{\langle\alpha,
\rho \rangle}}{\QQ_{h_t^k}(\alpha)}\frac{e^{\langle \beta, \rho
\rangle}}{\QQ_{h_t^k}(\beta)} }{\left(\sum_{\alpha}
\frac{e^{\langle\alpha, \rho
\rangle}}{\QQ_{h_t^k}(\alpha)}\right)^2}
 - \left(\frac{\partial^2
\phi_t}{\partial \rho_i \partial t} \right) = O(\frac{1}{k^2}).
\end{array}$$

\end{enumerate}

\end{prop}

Finally, we consider the analogous derivatives with respect to the
radial coordinates $r_j$ near $\dcal$. We assume $z$ is close to
the component of $\dcal$ given in local slice orbit coordinates by
$z' = 0$ and let $r' = (r_j)_{j = 1}^p$ denote polar coordinates
in this slice as discussed in \S \ref{TV}. The \szego kernel then
has the form \begin{equation} \label{PIWITHR} \Pi_{h_t^k}(z,  z)=
\sum_{\alpha \in k P \cap \Z^m} \frac{ \Pi_{j = 1}^p r_j^{2
\alpha_j} e^{\langle \rho'', \alpha''\rangle} e^{- k F_t(r_1^2,
\dots, r_p^2, e^{\rho_{p + 1}}, \dots,
e^{\rho_m})}}{\QQ_{h_t^k}(\alpha)}. \end{equation}  The
coefficients of the expansion (\ref{TYZ}) are smooth functions of
$r_j^2$ and the expansion may be differentiated any number of
times.

The behavior of $\Pi_{h_t^k}(z,  z)$ for $z \in \dcal$ has the new
aspect that many of the terms vanish.  The extreme case is where
$z$ is a fixed point. We choose the slice coordinates so that it
has coordinates $z = 0$. We observe that only the term with
$\alpha = 0$ in (\ref{PIWITHR}) is non-zero, and the $\alpha$th
term vanishes to order $|\alpha|$.

Since $\frac{\partial}{\partial r_j} = \frac{2}{ r_j}
\frac{\partial}{\partial \rho_j}$ where both are defined, the
calculations  above are only modified by the presence of new
factors of $\frac{2}{r_j}$ in each space derivative. Since we are
applying the derivative to functions of $r_j^2$, it is clear that
the apparent poles will be cancelled. Indeed, the $r_j$ derivative
removes any lattice point $\alpha$ with vanishing $\alpha_j$
component.  Comparing these derivatives with derivatives of
(\ref{PIWITHR}) gives the following:

\begin{prop} \label{ONESZEGODERr} For $n = 1, \dots, p$, we have:
\begin{itemize}

\item $\frac{1}{k} \frac{\partial}{\partial r_n}  \log
\Pi_{h_t^k}(z,z)  = \frac{\sum_{\alpha: \alpha_n \not= 0 }\frac{2
(\frac{\alpha_n}{k} - \mu_{t n} (z))}{ r_n} \;\; \frac{ \Pi_{j =
1}^p r_j^{2 \alpha_j} e^{\langle \rho'', \alpha''\rangle} e^{- k
F_t(r_1^2, \dots, r_p^2, e^{\rho_{p + 1}}, \dots,
e^{\rho_m})}}{\QQ_{h_t^k}(\alpha)}}{\sum_{\alpha} \frac{ \Pi_{j =
1}^p r_j^{2 \alpha_j} e^{\langle \rho'', \alpha''\rangle} e^{- k
F_t(r_1^2, \dots, r_p^2, e^{\rho_{p + 1}}, \dots,
e^{\rho_m})}}{\QQ_{h_t^k}(\alpha)}} = O(\frac{1}{k^2}) $

\end{itemize}

\end{prop}

In effect, the exponent $\alpha$ is taken to $ \alpha - (0, \dots,
1_n, \dots)$ in the sum or removed if $\alpha_n = 0$, where $(0,
\dots, 1_n, \dots)$ is the lattice point with only a $1$ in the
$n$th coordinate. There are similar formulae for the second
derivatives $\frac{\partial^2}{\partial r_n
\partial r_i}, \frac{\partial^2}{\partial r_n \partial t},
\frac{\partial^2}{\partial r_n \partial \rho_i}$. The only
important point to check  is that the modification changing
$\alpha$ to $ \alpha - (0, \dots, 1_n, \dots)$ does not affect the
proofs in \S \ref{C1}-\ref{C2}.

\section{\label{LOCALIZATIONB}Localization of Sums: Proof of Lemma \ref{LOCALIZATION}}

 The following
Proposition immediately implies Lemma \ref{LOCALIZATION}:

\begin{prop}\label{LOCALSUM}  Given $(t, z)$, and for any $\delta, C >
0$, there exists $C' > 0$ such that
$$\frac{|s_{\alpha}(z)|^2_{h_t^k}}{\QQ_{h_t^k}(\alpha)} =
\pcal_{h_t^k}(\alpha, z)
 = O(k^{-C}), \;\;\; \mbox{if}\;\; |\frac{\alpha}{k} - \mu_t(z)| \geq C' k^{- \frac{1}{2} + \delta}. $$
\end{prop}

\begin{proof}

The proof is based on  integration by parts.   All of the
essential issues occur in the  Bargmann-Fock model, so we first
illustrate with that case.

\subsection{Bargmann-Fock case}

 To analyze the decay of $\pcal_{h_{BF}^k}(\alpha, z)$ as a function of lattice points
$\alpha$, it seems simplest to use the integral formula
(suppressing the factor $k^m$ and normalizing the volume of $\T$
to equal one),
\begin{equation} \label{BFPZA} k^{|\alpha|}
\frac{|z^{\alpha}|^2}{\alpha!} e^{- k |z|^2} = (2 \pi)^{-m}
\int_{T^m} e^{- k \left(|z|^2 (1 - e^{i \theta}) - i \langle
\frac{\alpha}{k}, \theta \rangle \right)} d \theta = e^{- k |z|^2}
(2 \pi)^{-m} \int_{T^m} e^{ k \left(|z|^2  e^{i \theta}) - i
\langle \frac{\alpha}{k}, \theta \rangle \right)} d \theta.
\end{equation}

We  observe that the  rightmost expression in (\ref{BFPZA})  is
$e^{- k |z|^2}$ times a complex oscillatory integral with phase
$$\Phi_{z, \frac{\alpha}{k}}(\theta) =    |z|^2  (e^{i \theta} - 1) -  i \langle
\frac{\alpha}{k}, \theta \rangle.  $$ We observe that (consistent
with Proposition \ref{MUCEQ}),
$$\nabla_{\theta} \Phi_{z, \frac{\alpha}{k}} (\theta) =  i( |z|^2  e^{i \theta}  -
\frac{\alpha}{k})  = 0 \iff e^{i \theta} |z|^2 = |z|^2 =
\frac{\alpha}{k}.
$$ Further, we claim that
\begin{equation} \label{GRADLB} |\nabla_{\theta} \Phi_{z, \frac{\alpha}{k}} (\theta)| \geq |
|z|^2 - \frac{\alpha}{k}|.  \end{equation}  Indeed,  the function
$$f_{z, \alpha}(\theta) : = \left| e^{i \theta} |z|^2 -
\frac{\alpha}{k} \right|^2 = \sum_{j = 1}^m \left(\cos \theta_j \;
|z_j|^2 - \frac{\alpha_j}{k} \right)^2 + \left(\sin \theta_j\;
|z_j|^2 \right)^2
$$ on $\T$ has a strict global minimum at $\theta = 0$ as long as
$|z_j|^2 \not= 0, \frac{\alpha_j}{k} \not= 0$ for all $j$. It
still has a global minimum without these restrictions, but the
minimum is no longer strict. We note that this discussion of
global minima is possible only because the \kahler potential
admits a global analytic continuation in $(z,w)$;  in general,
one can only analyze critical points near the diagonal.

We integrate  by parts with the operator
\begin{equation} \label{ELL} \lcal = \frac{1}{k} \frac{1}{|\nabla_{\theta}  \Phi_{z,
\frac{\alpha}{k}|^2}} \overline{\nabla_{\theta} \Phi_{z,
\frac{\alpha}{k}}} \cdot \nabla_{\theta},
\end{equation} i.e., we apply its transpose
\begin{equation} \label{ELLT} \lcal^t = - \frac{1}{k} \frac{1}{|\nabla_{\theta}  \Phi_{z, \frac{\alpha}{k}}|^2}
\nabla_{\theta} \Phi_{z, \frac{\alpha}{k}} \cdot \nabla_{\theta} -
\frac{1}{k} \nabla_{\theta} \cdot \frac{1}{|\nabla_{\theta}
\Phi_{z, \frac{\alpha}{k}}|^2} \nabla_{\theta} \Phi_{z,
\frac{\alpha}{k}}
\end{equation}
to the amplitude. The second (divergence) term is $-1$ times
\begin{equation} \label{DIVTERM} \frac{1}{k} \frac{\nabla \cdot \nabla \Phi_{z,
\frac{\alpha}{k}} }{|\nabla \Phi_{z, \frac{\alpha}{k}}|^2} +
\frac{1}{k} \frac{ \langle \nabla^2 \Phi_{z,
\frac{\alpha}{k}}\cdot \nabla \Phi_{z, \frac{\alpha}{k}}, \nabla
\Phi_{z, \frac{\alpha}{k}} \rangle}{|\nabla \Phi_{z,
\frac{\alpha}{k}}|^4}. \end{equation}

We will need to take into account the  $k$-dependence of the
coefficients, and therefore   introduce some standard spaces of
semi-classical symbols.
 We denote by $S^n_{\delta}(\T)$ the class of smooth
functions $a_k(\theta)$ on $\T \times \N$ satisfying
\begin{equation} \label{SYMBCAL}  \sup_{e^{i \theta} \in \T} \; \left|
D_{\theta}^{\gamma} a_k(\theta) \right| \leq C k^{n + |\gamma|
\delta}. \end{equation} Here we use multi-index notation
$D_{\theta}^{\gamma} = \prod_{j = 1}^m (\frac{\partial}{i
\partial \theta_j})^{\gamma_j}$. Thus, each $D_{\theta_j}$
derivative gives rise to an extra order of $k^{\delta}$ in
estimates of $a_k$. We note that products of symbols satisfy
\begin{equation} \label{MULTSYMB} S_{\delta}^{n_1} \times S_{\delta}^{n_2} \subset
S_{\delta}^{n_1 + n_2}. \end{equation}

We now claim that (with $\delta$ the same as in the statement of
the Proposition),
\begin{enumerate}

\item $\frac{\nabla_{\theta} \Phi_{z,
\frac{\alpha}{k}}}{|\nabla_{\theta} \Phi_{z, \frac{\alpha}{k}}|^2}
  \in S^{\frac{1}{2} - \delta}_{\frac{1}{2} - \delta}$

\item (\ref{DIVTERM}) lies in $ S^{- 2 \delta}_{\frac{1}{2} -
\delta}$.
\end{enumerate}
In (2), we note the pre-factor $\frac{1}{k}$.  To prove the claim,
we first observe that the sup norm estimates are correct by
(\ref{GRADLB}) and from the fact that $\frac{ \nabla \Phi_{z,
\frac{\alpha}{k}} }{|\nabla \Phi_{z, \frac{\alpha}{k}}|}$ is a
unit vector.
 We further  consider derivatives of (1)-(2). Each $\theta$ derivative essentially introduces one more factor
of $k |\nabla_{\theta} \Phi_{z, \frac{\alpha}{k}}|$ and hence
raises the order by $k^{\frac{1}{2} - \delta}$. This continues to
be true for iterated derivatives, proving the claim.

 Now we observe that
\begin{equation} \label{LACTS} \lcal^t: S^n_{\frac{1}{2} - \delta}
\to S^{n - 2 \delta}_{\frac{1}{2} - \delta}. \end{equation}
Indeed, the first term of   $\lcal^t$ is the composition of (i)
$\nabla_{\theta}$, which raises the order by $\frac{1}{2} -
\delta$, (ii) multiplication by an element of $S^{\frac{1}{2} -
\delta}_{\frac{1}{2} - \delta}$ which again raises the order by
$\frac{1}{2} - \delta$ (iii)  times $\frac{1}{k}$ which lowers the
order by $1$. The second term is a multiplication by $\frac{1}{k}$
times an element of $ S^{1- 2 \delta}_{\frac{1}{2} - \delta}$ and
thus also lowers the order by $2 \delta$.

It follows that each partial integration by $\lcal$ introduces
decay of $k^{-2 \delta}$, hence for any $M > 0$,

$$\begin{array}{lll} (\ref{BFPZA}) & = &  e^{- k |z|^2}
(2 \pi)^{-m} \int_{T^m} e^{ k \left(|z|^2  e^{i \theta}) - i
\langle \frac{\alpha}{k}, \theta \rangle \right)} ((\lcal^t)^M 1
)d \theta
\\ &&
\\ && = O \left( k^{-2 \delta} \right)^{ M} e^{- k |z|^2} \int_{T^m} e^{ k
\Re  (|z|^2  e^{i \theta})} d\theta  = O(k^{- 2 \delta M})
\end{array}$$ in this region.

\subsection{\label{PCALALZ}General case}

We now generalize this argument from the model case to the general
one. With no loss of generality we may choose coordinates so that
$z$ lies in a fixed compact subset of  $\C^m$, where the open
orbit is identified with $(\C^*)^m$. In the open orbit we continue
to write $|z|^2 = e^{\rho}$.  The first step is to obtain a useful
oscillatory integral formula for $\pcal_{h^k}(\alpha, z) $. By
Proposition \ref{INTEGRAL} and Proposition \ref{SZKTV}, we have
\begin{equation} \label{TORICPCAL} \pcal_{h^k}(\alpha, z) = (2 \pi)^{-m}
\int_{T^{m}} e^{ k (F_{\C}(e^{i \theta} |z|^2) - F(|z|^2))}\;
\chi(d(z,e^{i \theta} z)) \; A_k \big( z,
 e^{i \theta} z , 0 \big) e^{i \langle \alpha, \theta \rangle} d
\theta \; + O(k^{-\infty}).\end{equation} The  phase is given by
\begin{equation} \label{PHIDEFIN} \Phi_{z,
\frac{\alpha}{k}}(\theta) = F_{\C}(e^{i \theta} |z|^2) - F(|z|^2)
- i \langle \frac{\alpha}{k}, \theta \rangle \end{equation} where
as above, $F_{\C}(e^{i \theta} |z|^2)$ is the almost analytic
continuation of the \kahler potential $F(|z|^2)$ to $M \times M$.
By (\ref{DAMPED1}) and (\ref{HESSDIA}), it  satisfies
\begin{equation} \label{LBD} \Re (F_{\C}(e^{i \theta} |z|^2) - F(|z|^2))) \leq - C  d(z, e^{i
\theta} z)^2, \;\; (\mbox{for some}\;\; C > 0).
\end{equation} Hence, the integrand (\ref{TORICPCAL}) is  rapidly decaying on the set
of $\theta$  where $d(z,e^{i \theta} z)^2 \geq C \frac{ \log
k}{k}$   (see also (\ref{AGMON})),  and  we may replace
$\chi(d(z,e^{i \theta} z))$ by $\chi (k^{\frac{1}{2} - \delta'}
d(z,e^{i \theta} z)) \in S^{0}_{\frac{1}{2} - \delta'}$, since
 the contribution from $1 - \chi (k^{\frac{1}{2}
- \delta'} d(z,e^{i \theta} z))$ is rapidly decaying. Here,
$\delta'$ is an arbitrarily small constant and we may choose it so
that $\delta' < \delta$ in the Proposition.  (We did not use such
cutoffs in the Bargmann-Fock case since the real analytic
potential had a global analytic extension with obvious properties,
but as in \S \ref{AAEa}, it is necessary for almost analytic
extensions).

The set  $d(z,e^{i \theta} z) \leq C \frac{
k^{\delta'}}{\sqrt{k}}$ depends strongly on the position of $z$
relative to $\dcal$, or equivalently on the position of $\mu_h(z)$
relative to $\partial P$. For instance, if $z$ is a fixed point
then $d(z, e^{i \theta} z) = 0$ for all $\theta$. However, we will
not need to analyze these sets until the next  section.

We now generalize the integration by parts argument. Our goal is
to prove that $\pcal_{h_t^k}(\alpha, z)
 = O(k^{-C})$  if $ |\frac{\alpha}{k} - \mu_t(z)| \geq C k^{- \frac{1}{2} +
 \delta}$.
Now, the  gradient in $\theta$  of the phase of (\ref{TORICPCAL})
is given by
\begin{equation}\label{NABLA}  \nabla_{\theta} \Phi(z, \frac{\alpha}{k})(\theta)
= \nabla_{\theta} F_{\C} (e^{i \theta} |z|^2) - i \frac{\alpha}{k}
= i (\mu_{\C}(z, e^{i \theta} z) - \frac{\alpha}{k}),
\end{equation} where $\mu_{\C}(z, e^{i \theta} z)$ is the almost
analytic extension of the moment map (see \S \ref{AAEa}). The
following Lemma is obvious, but we display it to highlight the
relations between the small parameters $\delta$ of the Proposition
and $\delta'$ in our choice of cutoffs.

\begin{lem}\label{USESMUCEQ}  If  $ |\frac{\alpha}{k} - \mu_t(z)| \geq C k^{- \frac{1}{2} +
 \delta}$, and if $d(z,e^{i \theta} z) \leq C k^{- \frac{1}{2} + \delta'}$ with
 $\delta' < \delta$ , then $|(\mu(z, e^{i \theta} z) -
 \frac{\alpha}{k})| \geq C' k^{- \frac{1}{2} +
 \delta}$.
\end{lem}

\begin{proof} By Proposition \ref{MUCEQ},
$$\begin{array}{lll}  |(\mu(z, e^{i \theta} z) -
 \frac{\alpha}{k})|^2&  = & |(\Re \mu(z, e^{i \theta} z) -
 \frac{\alpha}{k})|^2 + |\frac{1}{2}
\nabla_{\theta} D(z, e^{i \theta} z)|^2  \\ && \\ & \geq &  |(
\mu(z) -
 \frac{\alpha}{k})|^2  + O(d(e^{i \theta} z, z)). \end{array} $$

 \end{proof}

It follows that under the assumption  $ |\frac{\alpha_k}{k} -
\mu_t(z)| \geq C k^{- \frac{1}{2} +
 \delta}$ of the Proposition,
  we  may  integrate by parts with the
operator \begin{equation} \label{LCAL2}
\begin{array}{lll}\lcal& = &\frac{1}{k} |\nabla_{\theta} \Phi_{z, \frac{\alpha}{k}} |^{-2}\;  \nabla_{\theta}
 \Phi_{z, \frac{\alpha}{k}}  \cdot \nabla_{\theta}
\end{array}. \end{equation}
 The transpose $\lcal^t$ has the same
form (\ref{ELLT}) as for the Bargmann-Fock example, the only
significant change being  that it is now applied to a non-constant
amplitude $A_k$  and to the cutoff $\chi (k^{\frac{1}{2} -
\delta'} d(z,e^{i \theta} z)) \in S^{0}_{\frac{1}{2} - \delta'}$
as well as to its own coefficients. Differentiations of  $A_k$ do
preserve the orders of terms; the  only significant change in the
symbol analysis in the Bargmann-Fock case is that differentiations
of $\chi (k^{\frac{1}{2} - \delta'} d(z,e^{i \theta} z))$ bring
only improvements of order $k^{-\delta'}$ rather than $k^{-
\delta}$. However, the order still decreases  by at least $2
\delta'$ on each partial integration, and therefore
 repeated integration by parts
again gives the estimate
\begin{equation} |\pcal_{h^k}(\alpha, z)| = O \left( ( k^{-\delta'})^{M}  \int_{\T} e^{k (\Re F(e^{i \theta}
|z|^2) - F(|z|^2) } d\theta \right)  = O \left( (k^{-\delta'}
)^{M} \right).
\end{equation}

\end{proof}

\begin{rem} It is natural to use integration by parts in this
estimate since the decay in $\mu_t(z) - \frac{\alpha}{k}$ must use
the imaginary part of the phase and is not a matter of being far
from the center of the Gaussian.
\end{rem}

\subsection{Further details on the phase}

For future reference (see Lemma \ref{FIRSTCASE}),  we Taylor
expand the phase (\ref{PHIDEFIN}) in the $\theta$ variable  to
obtain
\begin{equation} \label{PHASEANALYSIS} \Phi_{z, \frac{\alpha}{k}}(\theta)  = i \langle \mu(z) - \frac{\alpha}{k},
\theta \rangle  + \langle H_{\frac{\alpha}{k}} \theta, \theta
\rangle + R_3(k, e^{i \theta} \mu^{-1}(\frac{\alpha}{k})),
\end{equation} where $R_3 = O(|\theta|^3)$. Here, $H_{\frac{\alpha}{k}} = \nabla^2 F (\mu^{-1}(\frac{\alpha}{k}))$ denotes the Hessian
of $\phi$ at $\frac{\alpha}{k}$ (see (\ref{GINV}) of \S
\ref{KPSP})   Indeed, we have

\begin{equation}  \begin{array}{lll}
F_{\C}( e^{i \theta} |z|^2) - F( |z|^2) &=& \int_0^1 \frac{d}{dt}
F_{\C}( e^{i t \theta} |z|^2) dt \\ &&\\ & = & \int_0^1 \langle
\nabla_{\theta} F(e^{i t \theta} |z|^2), i \theta \rangle
dt \\ && \\
& = & \langle \nabla_{\rho} F(e^{\rho})),  (i \theta) \rangle
 + \int_0^1
(t - 1) \nabla_{\rho}^2 (F(e^{i t \theta +  \rho})) (i \theta)^2/2
dt  \\ && \\
& = & i \langle \mu(z),  \theta \rangle + \nabla_{\rho}^2
(F(e^{\rho})) (i \theta)^2  + R_3(k, e^{i \theta} \mu^{-1}(\frac{\alpha}{k}), \theta) \\ && \\
& = & i \langle \mu(z),  \theta \rangle +  \langle H_{z} \theta,
\theta \rangle + R_3(k, \theta, \mu^{-1}(\frac{\alpha}{k}) ),
\end{array} \end{equation}
in the notation (\ref{GINV}),  where $H_z = \nabla^2_{\rho}
F(|z|^2)$ and where
\begin{equation} \label{REMAINDERb} R_3( k, \theta, \rho) : = \int_0^1 (t - 1)^2
\langle \nabla_{\rho}^3 (F(e^{i t \theta +  \rho})), (i
\theta)^3/3!\rangle dt .
\end{equation}

\section{Proof of Regularity Lemma \ref{MAINLEM} and joint asymptotics of $\pcal_{h^k}(\alpha)$}

The first statement that $\rcal_{\infty}(t, x)$ is $C^{\infty}$ up
to the boundary follows from  (\ref{VOLDEN}),
\begin{equation}\begin{array}{lll}  \rcal_{\infty}(t, x)& = & \left(
\frac{\delta_{\phi_t}(x)\cdot\prod_{r=1}^{d}
\ell_{r}(x)}{\left(\delta_{\phi_0}(x)\cdot\prod_{r=1}^{d}
\ell_{r}(x))\right)^{1-t} \left(
\delta_{\phi_1}(x)\cdot\prod_{r=1}^{d} \ell_{r}(x)\right)^{t}}
\right)^{1/2} \\ && \\
& = & \left( \frac{\delta_{\phi_t}(x)}{\delta_{\phi_0}(x)^{1-t}
 \delta_{\phi_1}(x)^{t}} \right)^{1/2}, \end{array}
 \end{equation}
where the functions $\delta_{\phi}$ are positive, bounded below by
strictly positive constants and  $C^{\infty}$ up to $\partial P$.

We now consider the asymptotics of $\rcal_k(t, \alpha)$. We
determine the asymptotics of the ratio by first determining the
asymptotics of the factors of the ratio. We  could use either the
expression (\ref{QRATIO})  in terms of norming constants
$\QQ_h^k(\alpha)$ for   the dual expression in terms of
$\pcal_{h^k}(\alpha)$ in Corollary \ref{RP}. Each approach has its
advantages and each seems of interest in the geometry of toric
varieties, but for the sake of simplicity we only consider
$\pcal_{h^k}(\alpha)$ here. In \cite{SoZ} we take the opposite
approach of focusing on the norming constants.  The advantage of
using $\pcal_{h^k}(\alpha)$ is that it may be represented by a
smooth complex oscillatory integral up to the boundary, while
$\QQ_h^k(\alpha)$ are singular oscillatory integrals over $P$. A
disadvantage of $\pcal_{h^k}(\alpha)$ is that it does not extend
to a smooth function on $\bar{P}$ and has singularities on
$\partial P$.

The asymptotics of $\pcal_{h^k}(\alpha)$ are straightforward
applications of steepest descent in compact subsets of $M
\backslash \dcal$ but become non-uniform at $\dcal$. To gain
insight into the general problem we again consider first   the
 Bargmann-Fock model, where by (\ref{PCALBF}) we have
\begin{equation} \label{BFPZAa} \pcal_{h_{BF}^k}(\alpha) = k^m
e^{-|\alpha|} \frac{\alpha^{\alpha}}{\alpha!} = (2 \pi)^{-m} k^m
\int_{\T} e^{ k \langle e^{i \theta} - 1 - i \theta,
\frac{\alpha}{k} \rangle } d \theta.\end{equation} As observed
before, the factors of $k$ cancel so `asymptotics' means
asymptotics as
 $\alpha \to \infty$. This indicates that we do not have
asymptotics when $\alpha$ ranges over a bounded set, or
equivalently when $\frac{\alpha}{k}$ is $\frac{C}{k}$-close to a
corner. On the other hand, steepest descent asymptotics applies in
a coordinate $\alpha_j$ as long as $\alpha_j \to \infty$.  Our aim
in general is to obtain steepest descent asymptotics of
$\pcal_{h^k}(\alpha)$ in directions far from facets and
Bargmann-Fock asymptotics in directions near  a facet.

\subsection{Asymptotics of  $\pcal_{h^k}(\alpha)$}

The analysis of  $\pcal_{h^k}(\alpha)$ is closely related to the
analysis of $\pcal_{h^k}(\alpha, z)$ in \S \ref{PCALALZ}, and in a
sense is a continuation of it. But the arguments are now more than
integrations-by-parts.  We obtain the asymptotics of
$\pcal_{h^k}(\alpha)$ from the integral representation analogous
to (\ref{TORICPCAL}) (see also Proposition \ref{SZKTV} and
Corollary \ref{INTEGRALCOR}). Modulo rapidly decaying functions in
$k$, we have (in the notation of Proposition \ref{SZKTV}),

\begin{equation} \label{INTEGRALa}\begin{array}{lll} \pcal_{h^k}(\alpha)
&\sim & (2 \pi)^{-m}  \int_{\T } e^{ - k (F_{\C}( e^{i \theta}
\mu_h^{-1}(\frac{\alpha}{k})) - F( \mu_h^{-1}(\frac{\alpha}{k})))}
A_k \big( e^{i \theta}  \mu_h^{-1}(\frac{\alpha}{k}),
 \mu_h^{-1}(\frac{\alpha}{k}) , 0, k\big) e^{i \langle \alpha, \theta \rangle} d
\theta. \end{array}
\end{equation}
This largely reduces the asymptotic calculation of $\pcal_{
h^k}(\alpha)$ to facts about the off-diagonal asymptotics of the
\szego kernel (cf. Proposition \ref{SZKTV}).

The integral (\ref{INTEGRALa}) is the oscillatory integral
(\ref{TORICPCAL}) but with $z = \mu^{-1}(\frac{\alpha}{k})$.
Hence, as in (\ref{PHIDEFIN}), its phase is  \begin{equation}
\label{PHASEALPHA} \Phi_{\frac{\alpha}{k}}(\theta) = F_{\C}(e^{i
\theta} \mu^{-1}(\frac{\alpha}{k})) -
F(\mu^{-1}(\frac{\alpha}{k})) - i \langle \frac{\alpha}{k}, \theta
\rangle.  \end{equation}  As in (\ref{DAMPED1}) and (\ref{LBD})
(but with $i$ included in as part of the phase),
\begin{equation} \label{LBDb} \Re \Phi_{\frac{\alpha}{k}}(\theta) \leq - C  d(\mu^{-1}(\frac{\alpha}{k}), e^{i
\theta} \mu^{-1}(\frac{\alpha}{k}))^2, \;\; (\mbox{for some}\;\; C
> 0).
\end{equation}
Specializing  (\ref{NABLA}) to our $z =
\mu^{-1}(\frac{\alpha}{k})$, we get
\begin{equation}\label{NABLAa}  \nabla_{\theta}
\Phi_{\frac{\alpha}{k}}(\theta)   = \nabla_{\theta} F_{\C} (e^{i
\theta} \mu^{-1}(\frac{\alpha}{k})) - i \frac{\alpha}{k} = i
(\mu_{\C}(\mu^{-1}(\frac{\alpha}{k}), e^{i \theta}
\mu^{-1}(\frac{\alpha}{k})) - \frac{\alpha}{k}). \end{equation} By
Proposition \ref{MUCEQ},  the complex phase has  a critical point
 at values of  $\theta $ such that $d(z, e^{i \theta} z) \leq \delta, $ and $e^{i \theta}
\mu^{-1}(\frac{\alpha}{k}) = \mu^{-1}(\frac{\alpha}{k})$.  For
$\frac{\alpha}{k} \notin \partial P$, the only critical point is
therefore  $\theta = 0$. The phase then  equals zero, and hence at
the critical point the real part of the phase is at its maximum of
zero.

For $\frac{\alpha}{k} \notin \partial P$,  the critical point
$\theta = 0$ is non-degenerate. Specializing (\ref{PHASEANALYSIS})
to $z = \mu^{-1}(\frac{\alpha}{k})$, we have
\begin{equation} \label{TAYLOREXP} \begin{array}{lll}
F_{\C}( e^{i \theta} \mu_h^{-1}(\frac{\alpha}{k})) - F(
\mu_h^{-1}(\frac{\alpha}{k})) &=& \int_0^1 \frac{d}{dt} F_{\C}(
e^{i t \theta} \mu_h^{-1}(\frac{\alpha}{k})) dt\\ && \\ & = & i
\langle \frac{\alpha}{k} , \theta \rangle + i \langle
H_{\frac{\alpha}{k}} \theta, \theta \rangle + R_3(k, \theta,
\mu^{-1}(\frac{\alpha}{k}),
\end{array} \end{equation} where $R_3$ is defined in
(\ref{REMAINDERb}). Hence,
\begin{equation} \Phi_{\frac{\alpha}{k}}(\theta)  = \langle H_{\frac{\alpha}{k}} \theta, \theta
\rangle + R_3(\theta, k, \mu^{-1}(\frac{\alpha}{k})),
\end{equation} and finally,
 \begin{equation}\label{FINALCASEONE}
\begin{array}{lll}
 \pcal_{h^k}(\alpha) &\sim&
(2 \pi)^{-m} \int_{\T }  e^{ - k \langle H_{\frac{\alpha}{k}}
\theta, \theta \rangle} e^{k  R_3(\theta, k,
\mu^{-1}(\frac{\alpha}{k})) } A_k \big(
\mu_h^{-1}(\frac{\alpha}{k}),
 e^{i \theta}  \mu_h^{-1}(\frac{\alpha}{k}) , 0, k\big) d
\theta \end{array}.
\end{equation}
Non-degeneracy of the phase is the statement that
$H_{\frac{\alpha}{k}}$ is a non-degenerate symmetric matrix, and
this follows from strict convexity of the \kahler potential or
symplectic potential, see (\ref{GINV}). But as discussed in \S
\ref{KPSP}, the $H_{\frac{\alpha}{k}}$ has a kernel when
$\frac{\alpha}{k} \in \partial P$. Hence the stationary phase
expansion is non-uniform for $\frac{\alpha}{k} \in P$ and is not
possible when $\frac{\alpha}{k} \in \partial P$. This accounts for
the fact that we need to break up the analysis into several cases,
and that we cannot rely on the complex stationary phase method for
all of them.

 Specializing (\ref{GAUSSIAN}) and (\ref{AGMON}), we have
\begin{equation} \label{AGMON2}
| \; \Pi_{h^k}(e^{i \theta } \mu_h^{-1}(\frac{\alpha}{k}),
\mu_h^{-1}(\frac{\alpha}{k})) |\;  \leq C  k^m e^{- C k
d(\frac{\alpha}{k}), e^{i \theta} \frac{\alpha}{k}))^2} + O(e^{- C
\sqrt{k} d(z,e^{i \theta} z)}).
\end{equation} Hence, the integrand of (\ref{INTEGRALa}) is
negligible off the set of $\theta$ where
$d(\mu^{-1}(\frac{\alpha}{k}) ,e^{i \theta}
\mu^{-1}(\frac{\alpha}{k}) ) \leq C \frac{\log k}{\sqrt{k}}.$ We
now observe that for  $d(z,e^{i \theta} z) \leq C \frac{
k^{\delta}}{\sqrt{k}}$,
\begin{equation} \label{DISTANCEq}  d(e^{i \theta} z,
z)^2 \sim \sum_j (1 - \cos \theta_j) \ell_j(\mu(z)),
\end{equation}
where we sum over $j$ such that $|\ell_j(\mu(z))| << 1$ (we will
make this precise in Definition \ref{CLOSE}). In particular,
\begin{equation} \label{DISTANCEqa}  d(e^{i \theta} \mu_h^{-1}(\frac{\alpha}{k}),
\mu_h^{-1}(\frac{\alpha}{k}))^2 \sim \sum_j (1 - \cos \theta_j)
\ell_j(\frac{\alpha}{k}).  \end{equation}  Indeed, both in small
balls in the interior and  near the boundary, the calculation is
universal and hence is accurately reflected in the  Bargmann-Fock
model with all $H_j = 1$, where  the distance squared equals
\begin{equation} \label{DISTSQUARE}  \sum_{j = 1}^m |e^{i
\theta_j} z_j - z_j|^2 = 2 \sum_{j = 1}^m |z_j|^2 (1 - \cos
\theta_j) = 2 \sum_{j = 1}^m \ell_j(\mu(z))  (1 - \cos \theta_j).
\end{equation}

This motivates the following terminology:

\begin{defin} \label{CLOSE} Let $0 <
 \delta_k << 1$ . We say:

\begin{itemize}

\item  $x \in P$ is  $\delta_k$-close to (resp. $\delta_k$-far
from) the facet $F_j = \{\ell_j = 0\}$ if $\ell_j (x) \leq
\delta_k$ (resp. $\geq \delta_k$).

\item  $x$ is a $\delta_k$-interior point  if it is $\delta_k$-far
from all facets.

\end{itemize}

\end{defin}

There are $m$ possible cases according to the number of facets to
which $x$ is $\delta_k$-close.  Of course, $x$ can be
$\delta_k$-close to at most $m$ facets, in which case it is
$\delta_k$-close to the corner defined by the intersection of
these facets. We thus define
\begin{equation}\label{FCALDELTAK} \fcal_{\delta_k} (x) = \{r: |\ell_r(x)| <
\delta_k\}.\end{equation} We also let \begin{equation}
\label{CARDFCAL} \delta_k^{\#}(x) = \# \fcal_{\delta_k}(x)
\end{equation} denote the number of $\delta_k$-close facets to $x$.
Dual to the sets $\fcal_{\delta_k}$ above are the sets
\begin{equation} \fcal_{F_{i_1}, \dots, F_{i_r}} = \{x:
\fcal_{\delta_k}(x) = \{i_1, \dots, i_r\} \}. \end{equation}

The asymptotics of $\pcal_{h^k}(\alpha)$ depend to the leading
order on the determinant of the inverse of the  Hessian of the
phase of (\ref{INTEGRALa}) (see also  (\ref{TORICPCAL}))  at
$\theta = 0$. This Hessian is the same as the Hessian of the
\kahler potential discussed in \S \ref{KPSP}, and we recall that
its inverse is  the Hessian $G$ of the symplectic potential.
Hence, the asymptotics are in terms of the determinant of $G$,
which has first order poles on $\partial P$. This indicates that
the asymptotics are not uniform up to $\partial P$. We saw this as
well in the explicit example of the Bargmann-Fock case. We define
\begin{equation}\label{GCALDEF} \gcal_{\phi, \delta_k} (x) =  \left(
\delta_{\phi}(x)\cdot\prod_{j \notin \fcal_{\delta_k}(x)}
\ell_{j}(x) \right)^{-1},
\end{equation}
where the functions $\delta_{\phi}$ are defined in \S \ref{KPSP}.
When $x$ is $\delta_k$-far from all facets, then $\gcal_{\phi}(x)
= \det G_{\phi}$ (cf. \ref{VOLDEN}).  We also define $
\pcal_{h_{BF}^k}(k  \ell_j(x))$ to be the unique real analytic
extension of (\ref{PCALBF}) to all $x \in [0, \infty)$.   We then
consider Bargmann-Fock type functions of type (\ref{PCALBF}) which
are adapted to the corners of our polytope $P$:
\begin{equation}\label{PCALBFK}  \pcal_{P, k,  \delta_k} ( x) =
\prod_{j \in \fcal_{\delta_k}(x)} \pcal_{h_{BF}^k}(k \ell_j(
\frac{\alpha}{k}))
\end{equation}

and

\begin{equation}  \label{PTWIDDLE} \tilde{\pcal}_{P, k}(x) =
\prod_{j = 1}^d k^{-1}(2 \pi \ell_j(x ))^{ 1/2} \pcal_{h_{BF}^k} (k
\ell_j(x)),
\end{equation}

 When we
straighten out the corners by affine maps to be standard octants
and separate variables $x = (x', x'')$ into directions near and
far from $\partial P$, then {$\pcal_{P, k, \delta_k} (x)$} is by
definition a function of the near variables $x'$ and $\gcal_{\phi, \delta_k}
(x) $ is by definition a function of the far variables $x''$.

The main result of this section is:

\begin{prop}\label{MAINPCAL}   We have
\medskip

\begin{equation} \label{pasym1} \pcal_{h^k}(\alpha) \; = \;
 C_m k^{\frac{m}{2}}  \sqrt{  \det G_{\phi}
(\frac{\alpha}{k})} \tilde{\pcal}_{P, k}(\frac{\alpha}{k}) \;
\left(1 + R_k(\frac{\alpha}{k}, h) \right),
\end{equation}
where $R_k = O(k^{-\frac{1}{3}})$ and $C_m$ is a positive constant depending only on $m$.  The expansion
is uniform in the metric  $h$ and may be differentiated in the
metric parameter $h$ twice with a remainder of the same order.

Equivalently, with $ \delta_k^{\#}$ defined in
(\ref{CARDFCAL}) and by letting  $\delta_k = k^{-\frac{2}{3}}$,
\medskip

\begin{equation}\label{pasym2}\pcal_{h^k}(\alpha) \; = \;
 C_m k^{\frac{1}{2} \;(m - \delta_k^{\#}(\frac{\alpha}{k}))}  \sqrt{  \gcal_{\phi, \delta_k}
(\frac{\alpha}{k})}\; \pcal_{P, k,  \delta_k} (\frac{\alpha}{k})
\; \left(1 + R_k(\frac{\alpha}{k}, h) \right),
\end{equation}
where again $R_k = O(k^{- \frac{1}{3}})$.

\end{prop}

The factor $ k^{\frac{1}{2} \;(m -
\delta_k^{\#}(\frac{\alpha}{k}))}$ is due to the fact that we
apply complex stationary phase in $m -
\delta_k^{\#}(\frac{\alpha}{k})$ variables to a complex
oscillatory integral with symbol of order $k^{(m -
\delta_k^{\#}(\frac{\alpha}{k}))}$.

 As a check, let us
 consider the $m$-dimensional Bargmann-Fock case where
$\delta_k^{\#}(\frac{\alpha}{k})= r$, and with no loss of
generality we will assume that the first $r$ facets are the close
ones. The factor $k^m$ in the symbol of the \szego kernel is then
split into $k^r$ (absorbed in $\pcal_{P, k,
\delta_k} $) and $k^{m-r}$ in the far factor. As discussed in \S
\ref{BFNORMS}, the far factor should have the form
$$k^{m-r} \prod_{j = r + 1}^m \; e^{- \alpha_j} \frac{
\alpha_j^{\alpha_j}}{\alpha_j!} \sim k^{m-r} \prod_{j = r + 1}^m
\alpha_j^{-\frac{1}{2}}.$$ The asymptotic factor in Proposition
\ref{MAINPCAL},
$$k^{\frac{1}{2} \;(m - \delta_k^{\#}(\frac{\alpha}{k}))}  \; \left( \prod_{j = r +
1}^m \frac{k}{\alpha_j} \right)^{\frac{1}{2}}, $$ matches this
expression. Here, and throughout the proof, we always straighten
out the corner to a standard octant when doing calculations in
coordinates.

Secondly, as a check on the remainder, we note that it arises from
two sources. As will be seen in the proof, in `far directions' the
stationary phase remainder has the form $ O(\frac{1}{k
d(\frac{\alpha}{k}, \partial P)})$ while in the near directions it
has the form  $O( k (d(\frac{\alpha}{k}, \partial P))^2)$. When  $
d(\frac{\alpha}{k}, \partial P) \sim k^{-\frac{2}{3}}$ the
remainders match.

We break up the proof into cases according to the distance of
$\frac{\alpha}{k}$ to the various  facets as $k \to \infty$. Since
 we are studying joint asymptotics in $(\alpha,
k)$,  $\alpha$ may change with $k$.

\bigskip

\subsection{Interior asymptotics } \label{interiorasy}~

\medskip

\noindent  { $\bullet$ \bf $\displaystyle \frac{\alpha}{k}$ is $\delta$-far from all facets}

We first consider the case where $\frac{\alpha}{k}$ is
$\delta$-far from all facets as an introduction to the problems we
face. In this case, we obtain  asymptotics of the integral
(\ref{INTEGRALa}) by  a complex stationary phase argument. But it
is not quite standard even in this interior case. In the next
section, we goo on to consider the same expansion when $\delta$
depends on $k$.

\begin{lem} \label{FIRSTCASE}  Assume that there exists $\delta > 0$ such that
 $\ell_j (\frac{\alpha}{k}) \geq \delta$ for all $j$, i.e., that $\frac{\alpha}{k}$ is $\delta$-far from all facets.  Then
there exist bounded smooth functions $A_{-j}(x)$ on $\bar{P}$
such that
$$\pcal_{h^k}(\alpha) \sim  \; C_m   k^{\frac{m}{2}} \; \sqrt{\det
G_{\phi} (\frac{\alpha}{k})}\; \left(1 +
\frac{A_{-1}(\frac{\alpha}{k})}{k} +
\frac{A_{-2}(\frac{\alpha}{k})}{k^2} + \cdots +
O_{\delta}(k^{-M})\right).$$ Here, $G_{\phi} = \nabla^2 u$ (\S
\ref{KPSP}) and  $ G_{\phi}(\frac{\alpha}{k})$ is its value at
$\frac{\alpha}{k}$; its norm is $O(\delta^{-1})$ and its determinant is $O(\delta^{-m})$.

\end{lem}

Before going into the proof, we note that the only assumption on
the limit points of $\frac{\alpha}{k}$ is that they are
$\delta$-far from facets. The lattice points $\alpha$ are
implicitly allowed to vary with $k$.  Asymptotics of the left side
clearly depend on the asymptotics of the points
$\frac{\alpha}{k}$, and the Lemma  states how they do so.

\begin{proof}

We now apply the complex stationary phase method, or more
precisely its proof. The usual complex stationary phase theorem
applies to exponents $k \Phi(\theta)$ where $\Phi(\theta)$ is a
positive phase function with a non-degenerate critical point at
$\theta = 0$. In our case, the phase is also $k$-dependent since
it depends on   $\frac{\alpha}{k}$ and the asymptotics of
(\ref{FINALCASEONE}) therefore depend on the asymptotics of
$\frac{\alpha}{k}$ in the domain $d(\frac{\alpha}{k}, \partial P)
\geq \delta$. Our stated asymptotics also depend on the behavior
of $\frac{\alpha}{k}$ in the same way.

Although the exact statement of complex stationary phase \cite{Ho}
 (Theorem 7.7.5) does not apply, the proof applies
without difficulty in this region. Namely, we introduce a cutoff
$\chi_{\delta}(\theta) = \chi(\delta^{-1} \theta) \in
C^{\infty}(\T)$ which $\equiv 1$ in a $\delta$-neighborhood of
$\theta = 0$ and which vanishes outside a $2 \delta$-neighborhood
of $\theta = 0$. We decompose the integral into its
$\chi_{\delta}$ and $1- \chi_{\delta}$ parts. A standard
integration by parts argument, essentially the same as in Lemma
\ref{LOCALIZATION} shows that the $1 - \chi_{\delta}$ term is  $ =
O(\delta^{-M} k^M)$ for all $M
> 0$. In the $\chi_{\delta}$ part the integral may be viewed as an integral
over $\R^m$ and we may apply the Plancherel theorem as in the
standard stationary phase argument to obtain \begin{equation}
\label{FOURIER}
\begin{array}{lll}
 \pcal_{h^k}(\alpha) &\sim&
\frac{C_m}{\sqrt{\det (k H_{\frac{\alpha}{k}}}) } \int_{\R^m }
e^{ - \langle (k H_{\frac{\alpha}{k}})^{-1} \xi, \xi \rangle}
\fcal_{\theta \to \xi} \left(e^{k  R_3(\theta, k,
\mu^{-1}(\frac{\alpha}{k}))) } A_k \big(
\mu_h^{-1}(\frac{\alpha}{k}),
 e^{i \theta}  \mu_h^{-1}(\frac{\alpha}{k}) , 0\big)\right)(\xi) d
\xi \end{array},
\end{equation}
where $\fcal_{\theta \to \xi}$ is the Fourier transform.

The stationary phase expansion  (see \cite{Ho}, Theorem 7.7.5) is
the following:
\begin{equation} \label{SPEXP}   \sim (\frac{2\pi}{k})^{m/2} \frac{e^{\frac{i \pi}{4}
 \mbox{sgn} H_{\frac{\alpha}{4}}}}{\sqrt{|\det H_{\frac{\alpha}{k}} |} } \sum_{j}^{\infty} k^{-j} {\mathcal
P}_{\frac{\alpha}{k}, j} A_k   \big( \mu_h^{-1}(\frac{\alpha}{k}),
 e^{i \theta}  \mu_h^{-1}(\frac{\alpha}{k}) , 0\big)|_{\theta = 0} \end{equation} where
\begin{equation} \label{MSPJTERM} {\mathcal P}_{\frac{\alpha}{k}, j}  A_k(0) = \sum_{\nu - \mu = j}
 \sum_{2 \nu \geq 3\mu} \frac{i^{-j} 2^{-\nu}}{\mu! \nu!}
\langle H_{\frac{\alpha}{k}}^{-1} D_{\theta}, D_{\theta}
\rangle^{\nu} ( A_k R_3^{\mu} ) |_{\theta = 0}
\end{equation}

  The only
change in the standard argument  is that we have a family of
quadratic forms $H_{\frac{\alpha}{k}}$ depending on parameters
$(\alpha, k)$ rather than a fixed one. But the standard proof is
valid for this modification. As in the standard proof, we expand
the exponential in (\ref{FOURIER}) and evaluate the terms and the
remainder of the exponential factor just as in \cite{Ho} Theorem
7.7.5, to obtain (\ref{SPEXP}), which becomes
\begin{equation}\label{MSPa}
\begin{array}{l} \left( \det (k^{-1} G_{\phi}(\frac{\alpha}{k}))\right)^{1/2}
\sum_{j = 0}^M  k^{-j} \left(\langle
G_{\phi}(\frac{\alpha}{k})(D_{\theta}, D_{\theta}\rangle \right)^j
\chi_{\delta} e^{k R_3(k, \theta, \mu^{-1}(\frac{\alpha}{k})) }
A_k \big( \mu_h^{-1}(\frac{\alpha}{k}),
 e^{i \theta}  \mu_h^{-1}(\frac{\alpha}{k}) , 0, k\big) |_{\theta = 0}\\ \\
+ O(k^{-M}  \sup_{\theta \in supp \chi_{\delta}} \left| \langle
G_{\phi}(\frac{\alpha}{k}) (D_{\theta}, D_{\theta}\rangle^M\;
\chi_{\delta} e^{k R_3(k, \theta, \mu^{-1}(\frac{\alpha}{k})) }
A_k \big( \mu_h^{-1}(\frac{\alpha}{k}),
 e^{i \theta}  \mu_h^{-1}(\frac{\alpha}{k}) , 0, k\big))\right|.
\end{array} \end{equation}
 Here, $G_x$ is the Hessian of the symplectic
potential, i.e., the inverse of  $H_{\mu^{-1}(x)}$. (cf.
\ref{KPSP}). We recall that  $G_x$ has poles $x_j^{-1}$  of order
one  when $x \in \partial P$. When  $d(\frac{\alpha}{k},
\partial P) \geq \delta$,  its  norm is therefore $O(\delta^{-1})$ and its determinant is $O(\delta^{-m})$.  Since $R_3$
vanishes to order $3$ at the critical point,  the terms of the
expansion can be arranged into terms of descending order as in the
standard proof. If we recall that the leading term of $S$ is
$k^{m}$, we obtain the statement of Proposition \ref{MAINPCAL} in
the $\delta$-interior case.

\end{proof}

\noindent {\bf $\bullet$ $\displaystyle \frac{\alpha}{k}$ is
$\delta_k$-far from facets with $k \delta_k \to \infty$
\label{IZ}}

We continue to study the complex oscillatory integral
(\ref{INTEGRALa}) but now allow $\frac{\alpha}{k}$ to become
$\delta_k$-close to some facet, and obtain a stationary phase
expansion (with very possibly slow decrease in the steps) under
the condition that $k \delta_k \to \infty$. This should be
feasible since the phase $k \Phi_{\frac{\alpha}{k}}$ is still
rapidly oscillating in this region, albeit at different rates in
different directions according to the proximity of
$\frac{\alpha}{k}$ to a particular facet. The principal
complication is as as follows:

\begin{itemize}

\item The Hessian $G_{\phi}(\frac{\alpha}{k})$ now has components
which blow up like $\delta_k^{-1}$   near the close facets. In the
stationary phase expansion, we get factors of
$$k^{-j} \langle G_{\phi}(\frac{\alpha}{k}) D_{\theta}, D_{\theta}
\rangle^j A_k \big( \mu_h^{-1}(\frac{\alpha}{k}),
 e^{i \theta}  \mu_h^{-1}(\frac{\alpha}{k}) , 0, k\big) R_3(k, \theta, \mu^{-1}(\frac{\alpha}{k}))^{\mu}$$ both
in the expansion and remainder. We must verify that these terms
still are of descending order.

\end{itemize}

As a guide, we note that by (\ref{BFPZA}),  the Bargmann-Fock
phase with $\mu(z) = \frac{\alpha}{k} $  is given by
$$\Phi_{BF, \frac{\alpha}{k}}(\theta)  =   \langle  \frac{\alpha}{k},    e^{i \theta} -  i  \theta \rangle
= \langle \cos \theta + i (\sin \theta - \theta), \frac{\alpha}{k}
\rangle,   $$ while the amplitude is constant. In this case, the
phase factors into single-variable factors and one can employ the
complex stationary phase method separately to each. In the general
case, we will roughly split the variables $\theta$ into two groups
$(\theta', \theta'')$, depending on $\frac{\alpha}{k}$, so that
the $\theta'$ variables are paired with the small components of
$\frac{\alpha}{k}$ while the $\theta''$ variables are paired with
its large components.  The complex stationary phase method applies
equally to either $d \theta'$ or $d\theta''$ integral, but the
orders of the terms are determined by the proximity of
$\frac{\alpha}{k}$ to the facets.

\begin{lem} \label{FIRSTCASEa}  Let $\{\delta_k\}$ be a sequence such that $k \delta_k \to \infty$.
Assume that  $\ell_j(\frac{\alpha}{k}) \geq \delta_k$
for all $j$, i.e.,, that $\frac{\alpha}{k}$ is $\delta_k $ far from
all facets.  Then in the notation of   Lemma \ref{FIRSTCASE}, we
have
$$\pcal_{h^k}(\alpha) \sim C_m k^{\frac{m}{2}} \;  \sqrt{\det
G_{\phi}(\frac{\alpha}{k}) }\; \left(1 +
\frac{A_{-1}(\frac{\alpha}{k})}{k} +
\frac{A_{-2}(\frac{\alpha}{k})}{k^2} + \cdots +
\frac{A_{-M}(\frac{\alpha}{k})}{k^2}+ O(k \delta_k)^{-M}
\right),$$ where now
$$A_{- j} (\frac{\alpha}{k}) \leq D \delta_k^{-1} =  C d(\frac{\alpha}{k}, \partial
P)^{-j}.$$

\end{lem}

\begin{rem} One may regard this as an
expansion in the semi-classical parameter $(k \delta_k)^{-1} = (k
d(\frac{\alpha}{k}, \partial P))^{-1}$. \end{rem}

\begin{proof}

We need to prove that the expansion (\ref{MSPa}) may be
re-arranged into terms of decreasing order and that the remainder
can be made to have an arbitrarily small order $k^{-M}$ by taking
sufficiently many terms.

 To analyze the expansion (\ref{MSPa}), we begin with a decomposition of the inverse Hessian
$G_{\frac{\alpha}{k}}$, which  is the Hessian of the symplectic
potential, which has the form $u_0 + g$ where $g \in
C^{\infty}(\bar{P})$ and where $u_0$ is the canonical symplectic
potential (\ref{CANSYMPOT}).  We continue to fix a small  $\delta
> 0$ as in the previous section, and consider the facets to which
$\frac{\alpha}{k}$ is $\delta$-close. We use the affine
transformation to map these $\delta$-close facets  to the
hyperplanes $x'_j = 0$. In these coordinates, we may write the
symplectic potential as
 \begin{equation} \label{CANSYMPOTa}
u_{\phi}(x) =  \sum_{j \in \fcal_{\delta}} x_j' \log x_j' + g(x),
\end{equation}
where the Hessian of $g$ is bounded with bounded derivatives near
$\frac{\alpha}{k}$. The Hessian $G_{\frac{\alpha}{k}}$  then
decomposes into the sum,
\begin{equation} \label{BLOCKS}
G_{\phi}(x) =  \sum_{j \in \fcal_{\delta}(\frac{\alpha}{k})}
\frac{1}{x_j'}\;  \delta_{jj} + \nabla^2 g: = G_x^s +  \nabla^2 g,
\end{equation}
where $\nabla^2 g$ is smooth up to the boundary in a neighborhood
of $\fcal_{\delta_k}(\frac{\alpha}{k})$. The notation
$G_{\phi}(x)^s$ refers to the `singular part' of $G_x$. The choice
of $\delta$ is not important; we are allowing $\frac{\alpha}{k}$
to become $\delta_k$ close to some facets, and for any choice of
$\delta$, the sum will include such facets.

The decomposition (\ref{BLOCKS}) of the inverse Hessian induces a
block decomposition of the Hessian operator  $\langle
G_{\frac{\alpha}{k}} D_{\theta}, D_{\theta} \rangle$.  The chage
of variables to $x$ above induces an affine change of the $\theta$
variables, as follows:  We are using the  coordinates $(x', x'')$
on $P$ with $x'$ denoting the linear coordinates in the directions
of the normals to the facets $\fcal_{\delta}(\frac{\alpha}{k})$.
The normals corresponding to $\fcal_{\delta}(\frac{\alpha}{k})$
generate the isotropy algebra of the sub-torus $(\T)' $ fixing the
near facets. We  have $\T = (\T)' \times (\T)''$, and denote the
corresponding coordinates by $(\theta', \theta'')$.

The Hessian operator in these coordinates has the form
\begin{equation}\label{DECOMPG}  \langle G_{\phi}(\frac{\alpha}{k}) D_{\theta}, D_{\theta}
\rangle =  \sum_{j \in \fcal_{\delta}(\frac{\alpha}{k})}
\frac{k}{\alpha_j'}\;  D^2_{\theta'_j \theta'_j} + \langle
G_{\phi}(\frac{\alpha}{k})'' D_{\theta}, D_{\theta} \rangle,
\end{equation}
where the second term has bounded coefficients. Evidently, the
change to the interior stationary phase expansion is entirely due
to the singular part of the Hessian operator,
\begin{equation} \langle G^s{\frac{\alpha}{k}} D_{\theta}, D_{\theta}
\rangle : = \sum_{j \in \fcal_{\delta}(\frac{\alpha}{k})}
\frac{k}{\alpha_j'}\; D^2_{\theta'_j \theta'_j}. \end{equation}

We now consider the order of magnitude of the terms in the $j$th
term (\ref{MSPJTERM}), which has the form
\begin{equation}\label{NEWJ}
k^{-\nu} \;  ( \langle G_{\phi}(\frac{\alpha}{k}) D_{\theta},
D_{\theta} \rangle^{\nu} A_k \big( \mu_h^{-1}(\frac{\alpha}{k}),
 e^{i \theta}  \mu_h^{-1}(\frac{\alpha}{k}) , 0, k\big) R_3(k, \theta, \mu^{-1}(\frac{\alpha}{k}))^{\mu}) |_{\theta
= 0} \end{equation} with $\nu - \mu = j$ and with $2 \nu \geq
3\mu$. The latter constraint is evident from the fact that $R_3$
vanishes to order $3$.

Using (\ref{DECOMPG}), $\langle G_{\phi}(\frac{\alpha}{k})
D_{\theta}, D_{\theta} \rangle^{\nu}$ becomes a sum of terms of
which the most singular is $$\langle
G^s_{\phi}(\frac{\alpha}{k})D_{\theta}, D_{\theta} \rangle^{\nu} :
= (\sum_{j \in \fcal_{\delta}(\frac{\alpha}{k})}
\frac{k}{\alpha_j'}\; D^2_{\theta'_j \theta'_j})^{\nu}. $$ We will
only discuss the terms generated by this operator; the discussion
is similar but simpler for the other terms. In the extreme case of
$ \langle G''_{\frac{\alpha}{k}} D_{\theta}, D_{\theta}
\rangle^{\nu}$, the discussion is essentially the same as in the
previous section; in particular, (\ref{NEWJ}) has order $k^{-j}$.

The problem with each application of $\langle
G^s_{\phi}(\frac{\alpha}{k})  D_{\theta}, D_{\theta} \rangle$ is
that it raises the order by the maximum of $\frac{k}{\alpha_j'}$,
which may be as large as $k \delta_k. $ Although we have an
overall $k^{-j}$ and constraints $\nu - \mu = j, 2 \nu \geq 3
\mu$, it is not hard to check that these are not sufficient to
produce negative exponents of $k$.

The key fact which saves the situation is that the phase
$\Phi_{\frac{\alpha}{k}}$ and amplitude $S$ depend on $\theta$ as
functions of  $e^{i \theta} |\mu^{-1}(\frac{\alpha}{k})|^2$.
Although $R_3$ has a more complicated $\theta$-dependence, its
third and higher derivatives are the same as those of
$\Phi_{\frac{\alpha}{k}}$, and it is obvious that only these
contibute to (\ref{NEWJ}). Hence derivatives in $\theta$ bring in
factors of $|\mu^{-1}(\frac{\alpha}{k})|^2$ by the chain rule. Due
to the behavior of the moment map near a facet, these chain rule
factors cancel a square root of the blowing up factor in
$G_{\phi}(\frac{\alpha}{k})$. This turns out to be sufficient for
a descending series due to the power $k^{-j}$ and constraint $2
\nu \geq 3 \nu$.

Before giving all the details, let us consider what should be the
`worst' terms of (\ref{NEWJ}), i.e., the ones with the least decay
in $k$. Each factor of $R_3$ comes with a factor of $k$, so one
would expect terms with large $\mu$ to be `worst'.  The `worst'
term will be one with a maximum $\mu$ and where a maximum number
of applications on operator $\langle G^s_{\frac{\alpha}{k}}
D_{\theta}, D_{\theta} \rangle^{\nu}$ is applied to the
`chain-rule' factors $(e^{i \theta}
|\mu^{-1}(\frac{\alpha}{k}))_j|^2$ (the $j$th component of this
vector),  obtained from an application of some $D_{\theta_j'}$ to
$S$ or to $R_3$. If instead we differentiate $S$ or $R_3$ again,
we pull out another chain rule factor, which cancels more of the
bad coefficient $\frac{k}{\alpha'_j}$.

 We now give the rigorous argument. The terms of (\ref{NEWJ}) have
 the form,

\begin{equation} \label{AMPLITUDE} k^{- \nu + \mu} \; G_{\phi}(\frac{\alpha}{k})^{i_1 j_1}
\cdots
 G_{\phi}(\frac{\alpha}{k})^{i_{\nu} j_{\nu}} D^{\beta_1} R_3 \cdots D^{\beta_{\mu}} R_3 D^{\beta_{\mu+1}}
S,\end{equation}
 where $|\beta| = 2 \nu$ and where $D^{\beta_q}$ denote
universal constant multiples of the  multinomial differential
operators  $ \frac{\partial^{\beta_q} }{\partial \theta^{n_1}
\cdots
\partial \theta^{n_{\beta_q}}}$ where the union of the indices
agrees with $\{i_1, j_1, \dots, i_{\nu}, j_{\nu} \}$. We need each
$|\beta_q| \geq 3$ for $q \leq \mu$ to remove the zero of $R_3$.
If we only consider the most singular term, then we need $i_{q} =
j_{q} \in \fcal_{\delta}(\frac{\alpha}{k})$. In this case our term
becomes
\begin{equation} \label{AMPLITUDEs} k^{- \nu + \mu} \; \left(\prod^{\nu}_{j = 1; q_j\in
\fcal_{\delta}(\frac{\alpha}{k})}\frac{k}{\alpha'_{q_j}}\right)
 D^{\beta_1} R_3 \cdots D^{\beta_{\mu}} R_3 D^{\beta_{\mu+1}}
S,\end{equation} For each factor $\frac{k}{\alpha'_{q_j}}$, there
exist two factors of the associated differential operator
$\frac{\partial}{\partial \theta_{q_j}}$. When one is applied to
either $R_3$ or $S$ it pulls out a chain rule factor $e^{i
\theta_{q_j}} |\mu^{-1}(\frac{\alpha}{k}))_j|^2$. The second could
be applied to this factor, hence need not introduce any new
factors of  $|\mu^{-1}(\frac{\alpha}{k}))_j|^2$. We now estimate
(\ref{AMPLITUDEs}) by
\begin{equation} \label{AMPLITUDEsest} \left| (\ref{AMPLITUDEs}) \right| \leq k^{- \nu + \mu} \; \left(\prod^{\nu}_{j = 1; q_j\in
\fcal_{\delta}(\frac{\alpha}{k})} \frac{k}{\alpha'_{q_j}} \right)
\prod_{j = 1}^{\mu}
  |\mu^{-1}(\frac{\alpha}{k})_{q_j}|^2.\end{equation}
  Now $\mu^{-1}(x) = \nabla u_{\phi} (x)$ in $\rho$ coordinates. So the square of the $q_j$th component
  of $\mu^{-1}(\frac{\alpha}{k})$ equals  $\log \frac{\alpha_{q_j}}{k}$ plus a bounded remainder  in $\rho$
  coordinates; here as above we are using the $x_j$ coordinates
  adapted to $\frac{\alpha}{k}$. It follows that in the $z$
  coordinates adapted to the facets of $\dcal$ corresponding to
  the hyperplanes $x'_j = 0$, with $|z_j|^2 = e^{\rho_j}$,
  $|\mu^{-1}(\frac{\alpha}{k})_{q_j}|^2 \leq C
  \frac{\alpha_j}{k}. $ The constant $C$ comes from the smooth part of the symplectic potential
  and has a uniform bound.  As a check on the square root, we
  note that for the approximating Bargmann-Fock model we have
  $|z_j|^2 = \frac{\alpha_j}{k}$. It follows from
  (\ref{AMPLITUDEsest}) and  $\frac{k}{\alpha_j'} \leq C
d(\frac{\alpha}{k},
\partial P)^{-1}$  that
  \begin{equation} \label{AMPLITUDEsesta} \begin{array}{lll} \left| (\ref{AMPLITUDEs}) \right| \leq C k^{- \nu + \mu} \;
  \left(\prod^{\nu}_{j = 1; q_j\in
\fcal_{\delta}(\frac{\alpha}{k})} \frac{k}{\alpha'_{q_j}} \right)
\prod_{j = 1}^{\mu}
  \frac{k}{\alpha'_{q_j}}
& \leq & C k^{- \nu + \mu}   d(\frac{\alpha}{k}, \partial P)^{-\nu
+ \mu}\\ && \\
& = & C   (k  d(\frac{\alpha}{k}, \partial P))^{- j},
\end{array} \end{equation}

Effectively, the `semi-classical parameter' has changed from
$k^{-1}$ to $k^{-1} d(\frac{\alpha}{k}, \partial P)^{-1}$, a
natural parameter in boundary problems. As long as $k
d(\frac{\alpha}{k}, \partial P) \to \infty$ at some fixed rate, we
obtain a descending expansion.

\end{proof}

\subsection{Boundary zones: Corner zone \label{CZ}}

Having dealt with the case where $|\frac{\alpha_j}{k}| \geq
\delta_k,$ we now turn to the complementary cases where $d(\mu(z),
\partial P) \leq \delta_k$, i.e.,  at least
for one $j$, $|\frac{\alpha_j}{k}| \leq \delta_k$  or
equivalently, $\frac{\alpha}{k}$ is $\delta_k$-close to at least
one facet. The choice of the scale $\delta_k$ is so that it  it is
small enough to justify   the Bargmann-Fock approximation  in the
`near' variables.

In this section, we consider the extreme
`corner' case where $\mu(z)$  lies in a $\delta_k$-corner, i.e., where  there
exists a vertex $v \in \partial P$ so that $d(\mu(z), v) \leq
\delta_k$.  Putting $v = 0$,
the assumption becomes that $|\mu(z)| \leq C \delta_k$ . Our main object is to determine the scale $\delta_k$
so that the Bargmann-Fock approximation is valid. That is, for
$z = \mu^{-1}(\frac{\alpha}{k})$ we should have in the multi-index  notation  of \S \ref{BF}
(see (\ref{BFPZAa})),

\begin{equation} \label{INTEGRALBFa}\begin{array}{lll} \pcal_{h^k}(\alpha)
&\sim &  \pcal_{h^k_{BF} }(\alpha) = k^m (2 \pi)^{-m} \int_{\T }
e^{ - k \left( \sum_{j = 1}^m H_{j \bar{j}} (e^{i\theta_j} - 1  +
i \theta_j, \frac{\alpha_j}{k} \rangle \right)} d \theta.
\end{array}
\end{equation}

\begin{lem} \label{CORNER} If $\mu(z)$ lies in a
$\delta_k$-corner, then

$$\begin{array}{l} \pcal_{h^k}(\alpha) = C_m  \pcal_{h_{BF}^k}(\alpha)
(1  + O(\delta_k) + O(k \delta_k^2)) =  C_m
\pcal_{h_{BF}^k}(\alpha) (1 + O(k \delta_k^2)).
\end{array}$$

\end{lem}

\begin{proof} We may assume that $v = 0$ and that the corner is a
standard octant. The phase is \begin{equation} \label{CORNERPHASE}
k \left( F_{\C}(|z|^2 e^{i \theta}) - F(|z|^2) - \langle
\frac{\alpha}{k}, \theta \rangle \right). \end{equation} We Taylor
expand $F(w)$ at $w = 0$:
$$F_{\C}(e^{i \theta}|z|^2) = F(0) + F'(0) e^{i \theta}|z|^2 +
O(|z|^4), $$  so that
$$ F_{\C}(|z|^2 e^{i \theta}) - F(|z|^2) = F'(0) |z|^2 (e^{i \theta} -
1)) + O(|z|^4). $$ Since $|z|^2 = O(\delta_k)$, it follows that $k$ times
the quartic remainder is $O(k \delta_k^2) = o(1)$ as long as $\delta_k = o(\frac{1}{\sqrt{k}})$.
 Hence this part of the exponential is a symbol of order zero
and may be absorbed into the amplitude.
 Further we note that $F'(0) |z|^2 = \mu(z) + O(|z|^4)$
and therefore we have
$$k \left( F_{\C}(|z|^2 e^{i \theta}) - F(|z|^2) - i \langle \frac{\alpha}{k},
\theta \rangle \right) = k \mu(z)  ( (1 - \cos \theta) + i (\sin
\theta - \theta)) + O(|z|^4)).$$ It follows that when $\mu(z) =
\frac{\alpha}{k} = O(\delta_k), $ the phase equals
$$\alpha  ( (1 - \cos \theta) + i (\sin
\theta - \theta)) + O(k \delta_k^2).$$ Absorbing the $e^{O(k \delta_k^2)} = 1 + O(k \delta_k^2)$ term into the
amplitude produces an oscillatory integral with the same phase
function as for the Bargmann-Fock  kernel.

Now let us consider the amplitude of the integral. We continue to
use the notation of Proposition \ref{SZKTV}. The  amplitude has a
semi-classical expansion $A_k(z, w) \sim k^m a_0(z,w) + k^{m-1}
a_1(z, w) + \cdots. $ Further, the $\T$-invariance implies that
$A_k(e^{i \theta} z, e^{i \theta} w) = A_k(z,w)$. The leading
order amplitude equals $1$ when $z = w$ and thus
$$a_0(z, e^{i \theta} w) = 1 + C e^{i \theta} |z|^2 + O(|z|^4), $$
hence the full symbol satisfies
$$A_k(z, e^{i \theta} z) = k^m (1 + C e^{i \theta} |z|^2 + \cdots) +
O(\delta_k^2). $$ When $\mu(z) = \frac{\alpha}{k} = O(\delta_k)$ we thus have
$$A_k(z, e^{i \theta} z) = k^m \left(1 + C e^{i \theta}
\frac{\alpha}{k} + O(\delta_k^2) \right). $$
 Therefore, $\pcal_{h^k}(\alpha) = \pcal_{h_{BF}^k} (\alpha)(1  + O(\delta_k) + O(k \delta_k^2))$ in the corner region.

\end{proof}

 \subsection{Boundary zones: Mixed boundary zone \label{BZ}}

Now let us consider the general case where $d(\mu(z),
\partial P) \leq \delta_k$, but where $\mu(z)$ is not necessarily in a corner.
Thus, at least one component $\frac{\alpha_j}{k} = O(\delta_k)$ but not all components need
to satisfy this condition. We refer to this case as `mixed' since some components are small and some
are not.

The basic idea to handle this case is to split the components into
`near' and `far' parts,  to use Taylor expansions and
Bargmann-Fock approximations in the near components, and to use
complex stationary phase in the far components. By \S
\ref{interiorasy}, complex stationary phase works for any
sequence $\delta_k$ satisfying $k \delta_k \to \infty$, and by \S
\ref{CZ} the Taylor-Bargmann-Fock approximation works whenever
$\delta_k = o(\frac{1}{\sqrt{k}})$, so we have some flexibility in
choosing $\delta_k$.

\begin{rem} In fact, we see that both the complex
stationary phase and the Bargmann-Fock approximations are valid
for $k$ satisfying (for instance) $ \frac{C \log k}{k} \leq
\delta_k \leq C' \frac{1}{\sqrt{k} \log k}$, although the
remainder estimates will not be equally sharp by both methods. In
fact, the stationary phase remainder is of order $(k
\delta_k)^{-1}$ while the Bargmann-Fock remainder is of order $k
\delta_k^2$; the two remainders agree when $\delta_k = k^{-\frac{2}{3}}$
and for small $\delta_k$ the Bargmann-Fock remainder is smaller.

\end{rem}

We first choose linear coordinates  so that $\mu(z) =
\frac{\alpha}{k}$ is $\delta_k'$ close to the first $r$ facets and
$\delta_k'$ far from the $p: = m - r$ remaining facets, and by an
affine map we position the first $r$ facets as the first $r$
coordinate hyperplanes at $x = 0$, and the remaining facets as the
remaining coordinate hyperplanes. We use coordinates $(x', x'')$
relative to this splitting.  We also write the $z$ variables as
$(z', z'')$ in the corresponding
 slice-orbit coordinates and $(\theta', \theta'')$ as the
 associated coordinates on $\T$.

 We now introduce two small scales, a smaller one $\delta_k'$ to
 define
 the nearest facets, and a larger one $\delta_k''$. The
 Bargmann-Fock approximation will be used in $x'$ variables which
 are $\delta_k'$ close to a facet. It is sometimes advantageous to
 use the Bargmann-Fock approximation also $x''$ which are
 $\delta_k''$ small, but the complex stationary phased method is
 also applicable. In the following, we continue to use the
 notation above Proposition \ref{MAINPCAL}.

\begin{lem} \label{SECONDCASE} Assume $\mu(z)$ lies in the mixed boundary zone $\{ |x'| \leq \delta_k',
 |x''|\leq \delta_k'' \}$.   If  $$ \eta_k =  k^{-1}(\delta_k'')^{-1}  +  k (\delta_k')^2 +
  k (\delta_k')^2 \delta_k'' + \delta_k' \to 0,$$ then $\pcal_{h^k}(\alpha)$ has an
asymptotic expansion
$$\pcal_{h^k}(\alpha) = C_m  k^{m - \frac{p}{2}} \sqrt{
\gcal_{\varphi, \delta_k}   ( \frac{\alpha}{k})} \pcal_{P, k, \delta_k'}(\alpha) (1
+ O(  \eta_k) ).
$$

\end{lem}

Our strategy for obtaining asymptotics of $\pcal_{h^k}(\alpha)$ in
this case is as follows:

\begin{itemize}

\item We employ steepest descent in the $p$  directions which are
$\delta_k''$-far from all facets, i.e., in the $x''$ variables.
This removes the $x''$ variables and produces an expansion
analogous to that of  Lemma \ref{FIRSTCASE}.

\item In the remaining $x'$ variables, we Taylor expand the phase
and amplitude  in the directions $\delta_k$-close to $\partial P$
as in \S \ref{CZ}.

\item We thus obtain universal asymptotics to leading order
depending only on the number of facets to which $\frac{\alpha}{k}$
is $\delta_k$-close.

\end{itemize}

\begin{proof}

We are still working on the oscillatory integral with phase
(\ref{INTEGRALa}), but we now treat it as an iterated complex
oscillatory integral in the variables $(\theta', \theta'')$
defined above.  We first consider the $d \theta''$ integral,
\begin{equation} \label{INTEGRALb}\begin{array}{l} I_k(\theta', \frac{\alpha}{k}): =  (2 \pi)^{-p} \int_{{\bf T}^p } e^{  k (F_{\C}( e^{i \theta}
\mu_h^{-1}(\frac{\alpha}{k})) - F( \mu_h^{-1}(\frac{\alpha}{k})))}
A_k \big( e^{i \theta}  \mu_h^{-1}(\frac{\alpha}{k}),
 \mu_h^{-1}(\frac{\alpha}{k}) , 0, k\big) e^{- i \langle \alpha, \theta \rangle} d
\theta'', \end{array}
\end{equation}
where $p$ is the number of $\theta''$ variables. We also let $r =
m - p$ be the number of $\theta'$ variables. We now verify that we
may apply the complex stationary phase method to the  $d \theta''$
integral for fixed $\theta'$. Throughout this section, we put $z =
\mu^{-1}(\frac{\alpha}{k})$ and often write $|z'|^2, |z''|^2$ for
the modulus square of the associated complex coordinate components
of this point in the open orbit.

The first step is to simplify the complex phase.  As in \S
\ref{CZ}, we Taylor expand $F_{\C}(e^{i \theta'}|z'|^2, e^{i
\theta''}|z''|^2)$ in the $z'$ variable (and only in the $z'$
variable) to obtain
$$F_{\C}(e^{i \theta'}|z'|^2, e^{i \theta''}|z''|^2) = F_{\C}(0,  e^{i \theta''}|z''|^2) +
F'_{1}(0,  e^{i \theta''}|z''|^2) e^{i \theta'}|z'|^2 + O(|z'|^4),
$$
where $F_1$ is the $z'$-derivative of $F$.
 The phase  is then
\begin{equation} \label{PHASEB} \begin{array}{l} k \left(F_{\C}(e^{i \theta'}|z'|^2, e^{i
\theta''}|z''|^2) - F(|z'|^2, |z''|^2) - i \langle
\frac{\alpha'}{k}, \theta' \rangle
) -  i \langle \frac{\alpha''}{k}, \theta'' \rangle \right) \\ \\
=  k \left( F_{\C}(0,  e^{i \theta''}|z''|^2) - F(0, |z''|^2
\right) + k \left( F'_{1}(0,  e^{i \theta''}|z''|^2)
e^{i \theta'}|z'|^2 -  F'_{1}(0,  |z''|^2) |z'|^2 \right)\\  \\
-  k \left(i \langle \frac{\alpha'}{k}, \theta' \rangle  +  i
\langle \frac{\alpha''}{k}, \theta'' \rangle  \right) + O(k |z'|^4).
\end{array} \end{equation}

 We now absorb the exponentials of the
 terms $k O(|z'|^4)$, $k i \langle
\frac{\alpha'}{k}, \theta' \rangle $ of the phase (\ref{PHASEB})
into the amplitude, i.e., we take the new amplitude $A_k''$ to be
the old one $A_k$ multiplied by this factor.  The term $k
O(|z'|^4)$ is $o(1)$, while $k i \langle \frac{\alpha'}{k},
\theta' \rangle $ is constant in $\theta''$, so their exponentials
are symbols in $\theta''$ and may  be absorbed into the amplitude.
Morevoer, the term $-  F'_{1}(0,  |z''|^2) |z'|^2 $ is independent
of $\theta''$ so its exponential may also be absorbed into the
amplitude

The phase function for the $d \theta''$ integral thus simplifies
to
\begin{equation} \label{PHASEC} k  \left( F_{\C}(0,  e^{i \theta''}|z''|^2)-  F(0, |z''|^2 \right)  + k \left( F'_{1}(0,  e^{i \theta''}|z''|^2) e^{i
\theta'}|z'|^2 \right) -  k   i \langle \frac{\alpha''}{k},
\theta'' \rangle
\end{equation}
Due to the presence of $|z'|^2$,  the terms $k( F'_{1}(0, e^{i
\theta''}|z''|^2) e^{i \theta'}|z'|^2 - F'_{1}(0, |z_2|^2) |z'|^2)
$ are $O(k \delta') $, hence of much lower order than the
remaining terms. To simplify the phase further, we now argue that
their exponentials can also be absorbed into the amplitude, albeit
as exponentially growing rather than polynomially growing factors
in $k$.  Since $F'_{1}(0, |z_2|^2) |z'|^2)$ is independent of
$\theta''$, it can be factored out of the $\theta''$ integral, so
the key factor is
\begin{equation} \label{EK} E_k(\theta'') : = e^{k( F'_{1}(0, e^{i \theta''}|z''|^2) e^{i
\theta'}|z'|^2)}, \end{equation}  where in the notation for $E_k$
we omit its dependence on the parameters $|z''|^2, |z'|^2,
\theta'$. Thus we would like to show that complex stationary phase
method applies to the complex oscillatory integral with phase
\begin{equation} \label{PHASED}  \Phi^{''}(\theta''): =  \left( F_{\C}(0,  e^{i \theta''}|z''|^2) - F(0, |z''|^2
\right)  -     i \langle \frac{\alpha''}{k}, \theta'' \rangle
\end{equation}
and with the amplitude $A''_k(\theta'')$ given by the original
amplitude $A_k$ multiplied by $$\exp k   \left( F'_{1}(0,  e^{i
\theta''}|z''|^2) e^{i \theta'}|z'|^2 -  F'_{1}(0,  |z''|^2)
|z'|^2 +  i \langle \frac{\alpha''}{k}, \theta'' \rangle   + O(k
|z'|^4 \right). $$ The `amplitude' is of exponential growth but
its growth is of strictly lower exponential growth than the
`phase' factor.

The next (not very important) observation is that by (\ref{LBDb}),
the real part of complex phase damps the integral so that the
integrand is negligible on the complement of the set
\begin{equation} \label{LOCSET} |\theta''| \leq C \frac{\delta'}{ d''(\mu(z), \partial P)} \end{equation} modulo rapidly decaying errors.
 This follows
by splitting up the sum in (\ref{DISTANCEq})-(\ref{DISTANCEqa})
into the close facets to $z$ and the far facets. The integrand is
negligible unless $|\Re \Phi| \leq  C \frac{\log k}{k}$; hence  it
is negligible unless
\begin{equation} \label{DISTANCEa} \begin{array}{lll}  d(e^{i \theta} z,
z)^2 & \sim &\sum_{j \in \fcal_{\delta_k}(\mu(z))}  (1 - \cos
\theta_j'') \ell_j''(\mu(z)) + O\left(|z'|^2 \right)  \\
&& \\  & \sim &\sum_{j \in \fcal_{\delta_k}(\mu(z))}  (\theta_j'')^2 \ell_j''(\mu(z)) + O\left( \delta_k' \right)  \\
&&
\\ &\leq & C \frac{\log k}{k} \iff \theta_j^2 \leq  \frac{O(\delta_k') + O(\frac{\log k)}{k}}{d''(\mu(z), \partial P)},\; \forall j \in \fcal_{\delta_k}(\mu(z)).
\end{array} \end{equation}
Under the assumption that $d''(\mu(z),
\partial P) \geq \delta_k''$,
 the integrand  is rapidly decaying unless
$\theta_j^2  \leq C \frac{\delta_k'}{\delta_k''}. $
 We could
introduce a cutoff of the form $\chi
(\sqrt{\frac{\delta_k''}{\delta_k'}} \theta)$, but  for our
purposes, it suffices to use  a smooth cutoff $\chi_{\delta}
(\theta'')$ around $\theta'' = 0$ with a fixed small $\delta$ so
that we may use local $\theta''$ coordinates. We then break up the
integral using $1 = \chi_{\delta} + (1 - \chi_{\delta})$. The $(1
- \chi_{\delta})$ term is rapidly decaying and may be neglected.

We observe that $\nabla_{\theta''}  F_{\C}(0,  e^{i
\theta''}|z''|^2) = i \mu''_{\C}(|z''|, e^{i \theta''} |z''|)$ is
the complexified moment map for the subtoric variety $z' = 0$, and
we can use Proposition \ref{MUCEQ} to see that its only critical
point in the domain of integration is at $\theta'' = 0$. We denote
the Hessian of the phase (\ref{PHASED}) at $\theta'' = 0$ by
\begin{equation} H''_{|z''|^2} = \nabla^2_{\theta''}
\Phi^{''}(\theta'')|_{\theta'' = 0} = \nabla^2_{\theta''}
F_{\C}(0,  e^{i \theta''}|z''|^2) |_{\theta'' = 0},\end{equation}
and observe that it equals $i D  \mu''_{\C}(|z''|, e^{i \theta''}
|z''|)$, the derivative of the moment map from the subtoric
variety to its polytope. By the same calculation that led to
(\ref{BLOCKS}), the  $\theta''-\theta''$ block of the inverse
Hessian operator has the form
\begin{equation} \label{BLOCKSb}
G''_{\phi}(x'') =  \sum_{j =1}^p \frac{1}{x_j''}\;  \delta_{jj} +
\nabla^2 g: = G_{x''}^s +  \nabla^2 g,
\end{equation}
where  $|x''| \geq \delta''_k$.

We now must verify that the complex stationary phase expansion

\begin{equation}\label{MSPa}
\begin{array}{l} \left(\det k^{-1}   G''_{\phi}(|z''|^2) \right)^{1/2}
\sum_{j = 1}^M  k^{-j} \left(\langle G''_{\phi}(|z''|^2)
(D_{\theta''}, D_{\theta''}\rangle \right)^j \chi_{\delta}
A''_k(\theta'') |_{\theta'' =
 0} \end{array}\end{equation}
 is a descending expansion in well-defined steps and
that the remainder \begin{equation} \label{REMAIN} k^{-M}
\sup_{\theta'' \in supp \chi_{\delta}} \left| \langle
G''_{\phi}(|z''|^2)  (D_{\theta''}, D_{\theta''}\rangle^M
A''_k(\theta'')^M\; \chi_{\delta} A''_k(\theta'')\right|.
 \end{equation}
 is of arbitrarily small order as $M$ increases.

We first note that the Hessian operator $k^{-1} \langle
G''_{\phi}(|z''|^2) (D_{\theta''}, D_{\theta''}\rangle$ brings in
a net order of $k^{-1 } (\delta_k'')^{-1}$, since the coefficients
$\frac{1}{x''}$ in the singular part are bounded by
$(\delta_k'')^{-1}$. The maximal order terms arise from applying
the Hessian operator to the factor $E_k$. Each derivative can
bring down a factor of $k F'_{1}(0, e^{i \theta''}|z''|^2) e^{i
\theta'}|z'|^2) = O(k \delta'_k \delta_k'') $. Since there are two $\theta''$
derivatives for each $k^{-1} (\delta_k'')^{-1}$ the maximum order
in $k$ from a single factor of $k^{-1} \langle
G''_{\phi}(|z''|^2)(D_{\theta''}, D_{\theta''}\rangle$ applied to
$A_k''$ is of order $$\eta_k= k^{-1} (\delta_k'')^{-1} ((k \delta'_k)^2 (\delta_k'')^2 + k\delta_k' \delta_k'')= k (\delta_k')^2 \delta_k'' + \delta_k'. $$  In
particular this is the order of magnitude of the sub-dominant
term. Therefore, to obtain a descending expansion in steps of at
least $k^{- \epsilon_0}$, we obtain the following necessary and
sufficient condition on $(\delta_k', \delta_k'')$:
\begin{equation} \label{DELTAS}
\eta_k \leq C k^{ -\epsilon_0}.
\end{equation}
Under this condition, the series and remainder will go down in
steps of $k^{- \epsilon_0}$.

 With these
choices of $(\delta_k', \delta_k'')$, the complex stationary phase
expansion gives an asymptotic expansion in powers of $k^{-
\epsilon_0}$. Recalling that  the unique critical point occurs at
$\theta'' = 0$, the remaining $d\theta'$ integral is given by the
dimensional constant $C_m  (2 \pi)^{-r}$ times
\begin{equation} \label{BOUNDARYPCALa}\begin{array}{lll}
\pcal_{h^k}(\alpha) &\sim &  \left(  \det (k^{-1}
G''_{\phi}(|z''|^2)\right)^{ 1/2} \int_{{\bf T}^r} e^{i k \langle
\frac{\alpha'}{k}, \theta' \rangle } \sum_{j = 1}^M k^{-j}
\left(\langle G''_{\phi}(|z''|^2) (D_{\theta''},
D_{\theta''}\rangle \right)^j \chi_{\delta} A''_k(\theta', 0)
d\theta',
\end{array}\end{equation}
 plus the integral of the remainder (\ref{REMAIN}), which is
 uniform in $\theta'$ and integrates to a remainder of the same
 order. Here we wrote the amplitude as $A''_k(\theta', \theta'')$
 and set $\theta'' = 0$ after the differentiations.

The differentiations leave the factor $E_k$ (\ref{EK}) while
bringing down polynomials in the derivatives of its phase. The
same is true of the factor $e^{k O(||z'||^4)}$ that we absorbed
into the amplitude. We now collect these factors and note that the
exponent is simply the original phase (\ref{PHASEB}) evaluated at
$\theta'' = 0$: \begin{equation} \label{PHASEC}
\begin{array}{l} \Phi'(\theta'; |z'|^2, |z''|^2) := F_{\C}(e^{i \theta'}|z'|^2, |z''|^2) - F(|z'|^2, |z''|^2) - i \langle
\frac{\alpha'}{k}, \theta' \rangle )
\end{array}
\end{equation}

We also collect the derivatives of this phase  and the other
factors of $A_k$ and find that
\begin{equation} \label{AE} \sum_{j = 1}^M k^{-j} \left(\langle
G''_{\phi}(|z''|^2)  (D_{\theta''}, D_{\theta''}\rangle \right)^j
\chi_{\delta} A''_k(\theta', 0)    = e^{k \Phi'(\theta'; |z'|^2,
|z''|^2) } \tilde{A}_k(\theta'),
\end{equation}
where $\tilde{A}_k (\theta')$ is a classical  symbol in $k$ whose
order is the order $m$ of the original symbol $A_k$. The integral
(\ref{BOUNDARYPCALa}) then takes the form
\begin{equation} \label{BOUNDARYPCALb}\begin{array}{lll}
\pcal_{h^k}(\alpha) & \sim & C_m \left( \det ( k^{-1}
G''_{\phi}(|z''|^2)  ) \right)^{1/2} \int_{{\bf T}^r}  e^{k
\Phi'(\theta'; |z'|^2, |z''|^2)} \tilde{A}_k(\theta') d \theta'.
\end{array}\end{equation}
This  is a corner type integral as studied in \S \ref{CZ}, with
$|z''|^2$ as an additional parameter. The asymptotics of
(\ref{BOUNDARYPCALb}) are given by  Lemma \ref{CORNER}. It is only
necessary to keep track of the powers of $|z'|^2, |z''|^2$ and of
the parameter $k^{-1} (\delta_k'')^{-1} (k \delta'_k)^2$ in the
analysis of $\tilde{A}_k$.

To do so, we first observe that
\begin{equation} \nabla_{\theta'} F_{\C}(e^{i \theta'}|z'|^2,
|z''|^2) = i \mu'_{\C}((z', z''),  (e^{i \theta'} z', z'')),
\end{equation}
i.e., it is the $'$ component of the complexified moment map. By
definition of $(z', z'')$ it equals $\frac{\alpha'}{k}$ when
$\theta' = 0$. It follows that $ F'_1(0, |z''|^2) |z'|^2 =
\frac{\alpha'}{k}$, and the almost analytic extension satisfies
\begin{equation} F'_1(0, |z''|^2) e^{i \theta'} |z'|^2 =
e^{i \theta'} \frac{\alpha'}{k}, \end{equation} where (as
previously) the multiplication is componentwise. If we then Taylor
expand the phase, we obtain
\begin{equation} \Phi'(\theta'; |z'|^2, |z''|^2) =  F'_{1}(0, |z''|^2) |z'|^2
(1 - e^{i \theta'}) + O(|z'|^4) = \frac{\alpha'}{k} (1 - e^{i
\theta'}) + O(|z'|^4).
\end{equation}
If we absorb the $e^{k O(|z|^4)}$ factor into the amplitude, he
integral has now been converted to the form (\ref{INTEGRALBFa})
with a more complicated amplitude.

We next observe that
\begin{equation} \label{TAYLORAMPEXP} \tilde{A}_k = k^m( 1 +
O(|z'|^2)
).
\end{equation}  Hence, the
assumption $|z'|^2 = O(\delta_k')$ implies that to leading order
\begin{equation} \label{FINALBINT} \begin{array}{lll} \pcal_{h^k}(\alpha) &\sim& \sqrt{\det k^{-1}
\;G''_{\phi}(|z''|^2)}
 k^m \int_{T^r } e^{ - k \left(  ( e^{i \theta'} - 1 - i \theta)
\right) \frac{ \alpha'}{k}}  d \theta' (1 +
O(\delta_k'))\\ && \\
& = &  k^{m - \frac{p}{2}} \sqrt{\det G''_{\phi}(|z''|^2)})
\pcal_{h_{BF}^k}(\alpha') (1 + O(\delta_k')).
\end{array} \end{equation}
  This completes the proof
of the Lemma.

\end{proof}

\subsection{Completion of proof  of Proposition \ref{MAINPCAL}}

We now complete the proof of Proposition \ref{MAINPCAL}.

\subsubsection{\bf Asymptotic expansion for $\pcal_{h^k}(\alpha)$} The error terms for  the asymptotics of $\pcal_{h^k}(\alpha)$,  in the corner zone, the interior zone and the mixed zone,  are given by
$  k^{-1}(\delta_k'')^{-1}$, $  k (\delta_k')^2 $ and $  \eta_k = k^{-1}(\delta_k'')^{-1} + k (\delta_k')^2+ k (\delta_k')^2 \delta_k'' + \delta_k'$ respectively.  In order to minimize these terms, we let
$$  k^{-1}(\delta_k'')^{-1}  =  k (\delta_k')^2~~~{\rm and }~~~ 0< \delta_k' \leq \delta_k''. $$
  By elementary calculation, the optimal choice for $\delta_k'$ and $\delta_k''$ is given by
$$\delta_k'=\delta_k''=  k^{-\frac{2}{3}}~~~{\rm and }~~~
 k^{-1}(\delta_k'')^{-1}  =  k (\delta_k')^2$$ and $$k^{-1}(\delta_k'')^{-1}  =  k (\delta_k')^2= k^{-\frac{1}{3}}, ~~\eta_k \sim O(k^{-\frac{1}{3}}).$$

We let $\delta_k = k^{-\frac{2}{3}}$ and break up the estimate into four cases.

\begin{enumerate} \label{ENUMERATE}

\item  $|x'|, |x''|  \leq \delta_k$: this is the corner case
handled in Lemma \ref{CORNER} if $ k (\delta_k)^2 \to 0$.

$$\pcal_{h^k}(\alpha) = C_m  \pcal_{h_{BF}^k}(\alpha)
(1  + O(k^{-\frac{1}{3}} ) ).$$

\item $|x'|, |x''|  \geq \delta_k$. By Lemma \ref{FIRSTCASEa},
stationary phase is valid and

$$\pcal_{h^k}(\alpha) \sim C_m k^{\frac{m}{2}} \;  \sqrt{\det
G_{\phi}(\frac{\alpha}{k}) }\; \left(1 + O(k^{-\frac{1}{3}} )
\right).$$

\item $|x'| \leq \delta_k$ and  $ |x''| \geq
\delta_k $. By Lemma \ref{SECONDCASE},  $$\begin{array}{l} \pcal_{h^k}(\alpha) = C_m  k^{m - \frac{p}{2}} \sqrt{\det
G''_{\phi}(\frac{\alpha}{k})} \pcal_{P, k, \delta_k'}(\alpha') (1
+ O(  k^{-\frac{1}{3}}) ).
\end{array}$$

\item $|x''| \leq \delta_k$ and  $ |x'| \geq
\delta_k $. This case is the same as case $(3)$ by switching $x'$ and $x''$.
\end{enumerate}
Combining the formulas above, the asymptotics for $\pcal_{h^k}(\alpha)$ is given by (\ref{pasym2})

$$\pcal_{h^k}(\alpha) \; = \;
 C_m k^{\frac{1}{2} \;(m - \delta_k^{\#}(\frac{\alpha}{k}))}  \sqrt{  \gcal_{\phi, \delta_k}
(\frac{\alpha}{k})}\; \pcal_{P, k,  \delta_k} (\frac{\alpha}{k})
\; \left(1 + R_k(\frac{\alpha}{k}, h) \right),
$$
where $R_k(\frac{\alpha}{k}, h)= O(k^{-\frac{1}{3}})$.

On the other hand, equation (\ref{pasym1}) is derived by the following calculation.

\begin{eqnarray*}
k^{\frac{1}{2}( m -\delta_k^{\#}) } \sqrt{  \gcal_{\phi, \delta_k}
(\frac{\alpha}{k})}\; \pcal_{P, k,  \delta_k} (\frac{\alpha}{k}) &=& k^{\frac{m}{2}} \sqrt{ \det G_\varphi(\frac{\alpha}{k})}
 \tilde{\pcal}_{P, k}     \prod_{j\notin \fcal_{\delta_k}(x)} (2\pi k\ell_j(\frac{\alpha}{k}))^{-\frac{1}{2}}e^{|k\ell_j(\frac{\alpha}{k})|} \frac{k\ell_j(\frac{\alpha}{k})}{k\ell_j(\frac{\alpha}{k})^{k\ell_j(\frac{\alpha}{k})}} \\
&=&  k^{\frac{m}{2}} \sqrt{ \det G_\varphi(\frac{\alpha}{k})}
 \tilde{\pcal}_{P, k}   (1 + O(k^{-\frac{1}{3}})),
\end{eqnarray*}
where the last equality follows from the Stirling approximation.

\subsubsection{\bf Derivatives with respect to metric parameters \label{METRICD}}

Now suppose that $h =h_t$ is a smooth one-parameter family of
metrics. We would like to obtain asymptotics
$(\frac{\partial}{\partial t})^j \pcal_{h_t^k}(\alpha)$ for $j =
1,2$.

\begin{prop}\label{MAINPCALD}  For $j = 1,2$, there exist
amplitudes $S_j$ of order zero such that
\medskip

$$(\frac{\partial}{\partial t})^j \pcal_{h_t^k}(\alpha) \; = \;
 C_m k^{\frac{1}{2}(m - \; \delta_k^{\#}(\frac{\alpha}{k}))} \;  \sqrt{  \gcal_{\phi_t, \delta_k} (\frac{\alpha}{k})}\;
\pcal_{P, k, \delta_k} (\frac{\alpha}{k}) \; \left(S_j(t, \alpha,
k) + R_k(\frac{\alpha}{k}, h) \right),
$$
where $R_k = O(k^{- \frac{1}{3}})$. The expansion is uniform in $h$
and may be differentiated in $h$ twice  with a remainder of the
same order.

\end{prop}

\begin{proof}

Such time derivatives may also be represented in the form
(\ref{INTEGRALa})
\begin{equation} \label{INTEGRALt} (\frac{\partial}{\partial
t})^j  \pcal_{h_t^k}(\alpha) = (2 \pi)^{-m} \int_{\T } e^{ - k
(F_t( e^{i \theta} \mu_{h_t}^{-1}(\frac{\alpha}{k})) - F_t(
\mu_{h_t}^{-1}(\frac{\alpha}{k})))} A_{k, j} \big(k, t, \alpha,
\theta \big) e^{i \langle \alpha, \theta \rangle} d \theta,
\end{equation}
with a new amplitude $ A_{k, j}$ that is obtained by a combination
of differentiations of the original amplitude in $t$ and of
multiplications  by $t$  derivatives of the phase. It is easy to
see that $t$ derivatives of the amplitude do not change the
estimates above since they do not change the order in growth in
$k$ of the amplitude. However,  $t$ derivatives of the phase bring
down factors $k (\frac{\partial}{\partial t})^j  (F_t( e^{i
\theta} \mu_{h_t}^{-1}(\frac{\alpha}{k})) - F_t(
\mu_{h_t}^{-1}(\frac{\alpha}{k}))$. The second derivative can
bring down two factors with $j = 1$ or one factor with $j = 2$. We
now verify that, despite the extra factor of $k$, the new
oscillatory integral still satisfies the same estimates as before.

The key point is that, by the calculation (\ref{TAYLOREXP}),  the
phase $ F_t( e^{i \theta} \mu_{h_t}^{-1}(\frac{\alpha}{k})) - F_t(
\mu_{h_t}^{-1}(\frac{\alpha}{k})) - i \langle \frac{\alpha}{k},
\theta \rangle$ for any  metric $h$ vanishes to order two at the
 critical point $\theta = 0$; the first derivative vanishes
because $\nabla_{\theta} F( e^{i \theta} z)|_{\theta = 0} = i
\mu_h(z)$. Hence,  the $t$ derivative of the $h_t$-dependent
Taylor expansion (\ref{TAYLOREXP}) for  a one-parameter family
$h_t$ of metrics  also vanishes to order $2$, i.e.,
\begin{equation} \label{TJ}(\frac{\partial}{\partial t})^j \left(F_t( e^{i \theta}
\mu_{h_t}^{-1}(\frac{\alpha}{k})) - F_t(
\mu_{h_t}^{-1}(\frac{\alpha}{k}) \right) =
O(|\theta|^2).\end{equation} Thus, for each new power of $k$ one
obtains by differentiating the phase factor in $t$ one obtains a
factor which vanishes to order two at $\theta = 0$. As a check, we
note that in the Bargmann-Fock model,  the phase has the form
$\sum_j (e^{i \theta_j} - 1 - i \theta_j) \frac{\alpha_j}{k}. $

Let us first consider the first derivative.  We  repeat the
asymptotic analysis but with the new amplitude $S_1$. In the
`interior region' the stationary phase calculation in Proposition
\ref{FIRSTCASE} proceeds as before, but the leading term (now of
one higher order than before) vanishes since it contains the value
of (\ref{TJ}) at the critical point as a factor.  Therefore the
asymptotics start at the same order as before but with the value
of the second $\theta$-derivative of the amplitude at $\theta =
0$.

In the corner, resp.  mixed boundary,  zone we obtain an integral
of the same type as the ones studied in Lemma \ref{CORNER}, resp.
Lemma \ref{SECONDCASE},  but again with an amplitude of one higher
order given by the $t$-derivative of the phase. The only change in
the calculation is in the Taylor expansion of the amplitude in
(\ref{TAYLORAMPEXP}) in the $z'$ variable, which now has the form
\begin{equation} \label{TAYLORAMPEXPa} \tilde{A}_{k, 1}  =  k (\frac{\partial}{\partial t}) \left(F_t( e^{i \theta}
\mu_h^{-1}(\frac{\alpha}{k})) - F_t( \mu_h^{-1}(\frac{\alpha}{k}))
\right) + O(|z'|^2 ), \end{equation} so that the final integral
now has the form $$(2 \pi)^{-m} k^m \int_{T^r } e^{ - k \left(  (
e^{i \theta'} - 1 - i \theta') \right) \frac{ \alpha'}{k}} \left(
k (\frac{\partial}{\partial t}) \left(F_t( e^{i \theta}
\mu_h^{-1}(\frac{\alpha}{k})) - F_t( \mu_h^{-1}(\frac{\alpha}{k}))
\right) \right)_{\theta'' = 0}  d \theta'.$$ As  noted in
(\ref{TJ})
$$\begin{array}{lll} k (\frac{\partial}{\partial t}) \left(F_t( e^{i \theta}
\mu_t^{-1}(\frac{\alpha}{k})) - F_t( \mu_t^{-1}(\frac{\alpha}{k}))
\right) & = & k (\frac{\partial}{\partial t}) \left(F_t( e^{i
\theta} \mu_t^{-1}(\frac{\alpha}{k})) - F_t(
\mu_t^{-1}(\frac{\alpha}{k})) - i \langle  \frac{\alpha}{k}, \theta \rangle \right) \\ && \\
&=&  k \frac{\partial}{\partial t} \int_0^1 (1 - s) \;
\frac{\partial^2}{\partial s^2} \left(F_t( e^{i  s \theta}
\mu_t^{-1}(\frac{\alpha}{k})\right) ds \\ && \\
& = &  O(k |\theta|^2 \frac{\alpha}{k} ).
\end{array} $$
Since the stationary phase method applies as long as $|\alpha| \to
\infty$ we may assume that $|\alpha| \leq C$ and we see that the
factor is then bounded. Here, we have suppressed the subscript
$\C$ for the almost-analytic extension to simplify the writing.

As an independent check, we use integration by parts in $\theta'$.
We use a cutoff function $\chi$ supported near $\theta' = 0$ to
decompose the integral into a term supported near $\theta' = 0$
and one supported away from $\theta' = 0$. We use the integration
by parts operator
$$\lcal = \frac{1}{\left(  ( e^{i \theta'}
- 1)  \alpha'\right)^2 }  \left(   e^{i \theta'} -  1 \right)
\alpha' \cdot \nabla_{\theta}$$ where we note that the factors of
$k$ cancel. The operator is well defined for $\theta' \not= 0$ and
repeated partial integration gives decay in $\alpha'$ in case
$|\alpha'| \to \infty$. On the support of $\chi$ the denominator
is not well defined but the vanishing of the phase to order two
shows that $\lcal^t (S_1)$ is bounded.

Now we consider second time derivatives. The second
$\frac{\partial}{\partial t}$ could be applied to the phase factor
$e^{k \Phi_{t}}$ again or it could be applied again to
(\ref{TAYLORAMPEXPa}),  and then we have
\begin{equation} \label{TAYLORAMPEXPb} \begin{array}{lll} \tilde{A}_{k, 2} & = &  (k (\frac{\partial}{\partial t}) (F_t( e^{i \theta}
\mu_h^{-1}(\frac{\alpha}{k}))
 - F_t( \mu_h^{-1}(\frac{\alpha}{k}) ))^2  \\ && \\ &&  + (k
(\frac{\partial^2}{\partial t^2}) (F_t( e^{i \theta}
\mu_h^{-1}(\frac{\alpha}{k})) - F_t( \mu_h^{-1}(\frac{\alpha}{k})
))^2  + O(|z'|^2). \end{array}
\end{equation}
The first term contains the factor $k^2$ and after cancellation it
induces a term of order $|\alpha'|^2$. In addition this term
vanishes to order four at $\theta = 0$. Hence the stationary phase
calculation in the case of the first derivative equally shows that
the first two terms vanish and thus the factors of $k^2$ are
cancelled. In the regime where stationary phase is not applicable,
$|\alpha'|^2$ may be assumed bounded, and additionally one can
integrate by parts twice. Thus again this term is bounded.

\end{proof}

\subsection{Completion of the proof of Lemma \ref{MAINLEM}}

So far we have only considered the asymptotics of $\pcal_{h^k}(t,
z)$. We now take the ratios to complete the proof of Lemma
\ref{MAINLEM}.

\begin{lem} \label{RATIOCANCEL} With $\delta_{\phi}$ defined by
(\ref{VOLDEN}), we have
$$\begin{array}{lll} \rcal_{k}( t, \alpha) & =  &
 \left( \frac{\det \nabla^2 u _t(\frac{\alpha}{k})}{ (\det \nabla^2
u_0(\frac{\alpha}{k}))^{1-t}(\det \nabla^2 u_1(\frac{\alpha}{k}))^{t}} \right)^{1/2} \left(1 +
O(k^{- \frac{1}{3}})\right) .\end{array}$$ The asymptotic in
 may be differentiated twice with the same order
of remainder.

\end{lem}

\begin{proof} Combining  Corollary \ref{RP} and  Proposition
\ref{MAINPCAL},we have \begin{equation} \label{RATIOPCAL}
\rcal_{k}( t, \alpha) = \frac{\sqrt{ \det G_{\phi_t}
(\frac{\alpha}{k})}\; \tilde{\pcal}_{P, k}
(\frac{\alpha}{k}) }{ (\sqrt{ \det G_{\phi_0}
(\frac{\alpha}{k})}\; \tilde{\pcal}_{P, k}
(\frac{\alpha}{k}))^{1-t}(\sqrt{ \det G_{\phi_1}
(\frac{\alpha}{k})}\; \tilde{\pcal}_{P, k}
(\frac{\alpha}{k}))^{t}} (1 + O(k^{-  \frac{1}{3}})).
\end{equation}  We observe that the factors of $\tilde{\pcal}_{P, k}
$   cancel out, leaving
\begin{equation} \label{RATIOCANCELa} \rcal_{k}(t, \alpha) = \frac{\sqrt{ \det G_{\phi_t}
(\frac{\alpha}{k})}}{ (\sqrt{ \det G_{\phi_0}
(\frac{\alpha}{k}})^{1-t}(\sqrt{ \det G_{\phi_1}
(\frac{\alpha}{k})})^t} (1 + O(k^{-\frac{1}{3}})).
\end{equation}

By Proposition \ref{MAINPCALD}, the asymptotic in
(\ref{RATIOPCAL}) may be differentiated twice with the same order
of remainder, completing the proof.

\end{proof}

\begin{rem} By (\ref{VOLDEN}), we also have
$$\rcal_{k}(t,\alpha) = \left( \frac{\delta_{\phi_0}^{1 - t}
\delta_{\phi_1}^{t}}{\delta_{\phi_t}} \right) ^{- \frac{1}{2}} (1
+ O(k^{- \frac{1}{3}})).$$ Indeed, the factors of
$\ell_j(\frac{\alpha}{k})$ are independent of the metrics and
cancel out. Also $\left( \frac{\delta_{\phi_0}^{1 - t}
\delta_{\phi_1}^{t}}{\delta_{\phi_t}} \right) ^{- \frac{1}{2}}$ is smooth on $P$.
\end{rem}

The following simpler  estimate on logarithmic derivatives is
sufficient for  much of  the proof of the main results:

\begin{lem}\label{BOUNDEDRCALDER}  We have:

\begin{enumerate}

\item  $\partial_t \log \rcal_k(t,  \alpha)$ is uniformly bounded.

\item $\partial_t^2  \log \rcal_k(t,  \alpha)$ is uniformly
bounded.

\end{enumerate}
\end{lem}

\begin{proof} We first note that
\begin{equation} \label{FULLR} \partial_t \log \rcal_k(t, \alpha) = \log \pcal_{h_1^k}(
\alpha) - \log \pcal_{h_0^k}(\alpha) - \partial_t \log
\pcal_{h_t^k}(\alpha). \end{equation}
  We note that by Proposition
\ref{MAINPCAL},
 \begin{equation} \label{LOGPCAL} \log \pcal_{h^k}(\alpha) \; = \; \frac{1}{2}\log  \det
(k^{-1} G_{\phi} (\frac{\alpha}{k}))\; + \log \tilde{\pcal}_{P, k}
(\frac{\alpha}{k}) \; + \log C_m +
O(k^{- \frac{1}{3}}).
\end{equation}
As in Lemma \ref{RATIOCANCEL}, the Bargmann-Fock terms cancel
between the $h_0$ and $h_1$ terms, while the metric factors
simplify asymptotically to $\frac{1}{2} \log \left(
\delta_{\phi_1} \delta_{\phi_0} \right)$, and this is clearly
bounded. To complete the proof of (1), we need that the final
ratio is bounded. By  Lemma \ref{MAINPCALD}, we see that in the
`interior' region both numerator and denominator have asymptotics
which differ only in the value of a zeroth order amplitude at
$\theta = 0$ and that it equals $1$ in the case of the
denominator. Hence, the ratio is bounded in the interior. Towards
the boundary, the denominator is comparable with the Bargmann-Fock
model and is bounded below by one. The numerator is also bounded
by Lemma \ref{MAINPCALD}, and therefore the ratio is everywhere
bounded.

Now we consider (2), which simplifies to
\begin{equation} \label{FULLR2} \partial_t^2 \log \rcal_k(t, \alpha) = -  \frac{ \partial_t^2
\pcal_{h_t^k}(\alpha)}{\pcal_{h_t^k}(\alpha)} + \left( \frac{
\partial_t \pcal_{h_t^k}(\alpha)}{\pcal_{h_t^k}(\alpha)} \right)^2.
\end{equation}
As we  have just argued, the second factor is bounded. The same
argument applies to the first term by Lemma \ref{MAINPCALD}.

\end{proof}

\section{$C^0$ and $C^1$-convergence \label{C1} }

We begin with the rather simple proof of $C^0$-convergence with
remainder bounds.

\subsection{$C^0$-convergence}

\begin{prop} $\frac{1}{k} \log Z_k(t, z) e^{- k \phi_t(z)} =
O(\frac{\log k}{k})$ uniformly for $(t, z) \in [0, 1] \times M$
\end{prop}

 The Proposition follows from the following:

\begin{lem} \label{UBLB} (Upper/Lower bound Lemma) There
 exist $C,
c > 0$ so that
$$c \leq  \rcal_k(t, \alpha) \leq C. $$
\end{lem}

\begin{proof} This follows immediately from Lemma
\ref{RATIOCANCEL}. \end{proof}

$C^0$-convergence is an immediate consequence of the upper and
lower bound lemma:

\begin{proof} By the upper/lower bound lemma, there exist positive
constants $c, C > 0$ so that
\begin{equation} \label{PART3} c \Pi_{h_t^k}(z,z) \leq   \sum_{\alpha \in k P \cap \Z^m}  \rcal_k(t, \alpha)
\frac{|S_{\alpha}(z)|^2_{h_t^k}}{\QQ_{h_t^k}(\alpha)} \leq C
\Pi_{h_t^k}(z,z).
\end{equation}

Hence,
\begin{equation} \label{PART4} \begin{array}{lll} \frac{1}{k}  \log
\Pi_{h_t^k} (z,z) \leq \frac{1}{k} \log \sum_{\alpha \in k P \cap
\Z^m} \rcal_k(t, \alpha)
\frac{|S_{\alpha}(z)|^2_{h_t^k}}{\QQ_{h_t^k}(\alpha)}&&  \leq
\frac{1}{k} \log \Pi_{h_t^k} (z,z) + O(\frac{1}{k})\\ && \\
&& = O(\frac{\log k}{k}),\end{array}
\end{equation}
where the last estimate follows from (\ref{TYZ}).

\end{proof}

\subsection{$C^1$-convergence}

We now discuss first derivatives in $(t, z)$. In the $z$ variable
the vector fields $\frac{\partial}{\partial \rho_j}$ vanish on
$\dcal$, so can only use them to estimate $C^1$ norms in
directions $\delta_k$ far from the boundary. In directions close
to  the boundary we may choose coordinates so that derivatives in
$z'$ near $z' = 0$ define the $C^1$ norm.

The estimates in the $\rho$ and $z'$ derivatives are similar. We
carry out the calculations in detail in the $\rho$ variables and
then indicate how to carry out the analogous estimates in the $z$
variable.

We also consider $t$ derivative. The key distinction between $t$
and $z$ derivatives is the following:
\begin{itemize}

\item $z$ or $\rho$ derivatives bring down derivatives of the
phase, which have the form $k (\mu_t(z) - \frac{\alpha}{k})$. The
factor of $k$ raises the order of asymptotics while the factor $
(\mu_t(z) - \frac{\alpha}{k})$ lowers it by the Localization
Lemma.

\item $t$ derivatives do not apply to the phase and only
differentiate $\rcal_k(t, \alpha)$ and $\QQ_{h_t^k}(\alpha)$.

\end{itemize}

\begin{prop} Uniformly for $(t, z) \in [0, 1] \times M$, we have:
\begin{enumerate}

\item $\frac{1}{k}\left|\frac{\partial}{\partial \rho_i}\log
\sum_{\alpha \in k P \cap \Z^m}  \rcal_k(t, \alpha)
\frac{|S_{\alpha}(z)|^2_{h_t^k}}{\QQ_{h_t^k}(\alpha)} \right| =
O(k^{-\frac{1}{2} + \delta})$;

\item The same estimate is valid if we differentiate in
$\frac{\partial}{\partial r_n}$ in directions near $\dcal$ as in
Proposition \ref{ONESZEGODERr}.

\item    $ \frac{1}{k}\left|\frac{\partial}{\partial t}\log
\sum_{\alpha \in k P \cap \Z^m} \rcal_k(t, \alpha)
\frac{|S_{\alpha}(z)|^2_{h_t^k}}{\QQ_{h_t^k}(\alpha)}  \right| =
O(k^{-\frac{1}{3} })$.

\end{enumerate}

\end{prop}

\begin{proof}

We first prove (1).

\begin{eqnarray*}
&&\frac{1}{k}\left|\nabla_{\rho} \log \sum_{\alpha \in k P \cap
\Z^m}  \rcal_k(t, \alpha)
\frac{|S_{\alpha}(z)|^2_{h_t^k}}{\QQ_{h_t^k}(\alpha)}
\right|\\
&=& \left|\frac{\sum_{\alpha \in k P \cap \Z^m }\left(
\frac{\alpha}{k} - \mu_t(z) \right) \rcal_k(t, \alpha)
\frac{|S_{\alpha}|^2_{h_t^k} }{\QQ_{h_t^k}(\alpha)} }
{\sum_{\alpha \in k P \cap \Z^m } \rcal_k(t, \alpha)
\frac{|S_{\alpha}|^2_{h_t^k} }{\QQ_{h_t^k}(\alpha)} }\right|\\
&=& \left|\frac{\sum_{\alpha \in k P \cap \Z^m:\;
|\frac{\alpha}{k} - \mu_t(z)| \leq k^{-\frac{1}{2} + \delta}
}\left( \frac{\alpha}{k} - \mu_t(z) \right) \rcal_k(t, \alpha)
\frac{|S_{\alpha}|^2_{h_t^k} }{\QQ_{h_t^k}(\alpha)} }
{\sum_{\alpha \in k P \cap \Z^m } \rcal_k(t, \alpha)
\frac{|S_{\alpha}|^2_{h_t^k} }{\QQ_{h_t^k}(\alpha)} }\right|
+ O(k^{-M})\\
& \leq & C k^{-\frac{1}{2} + \delta}  \left|\frac{\sum_{\alpha \in
k P \cap \Z^m:\; |\frac{\alpha}{k}
 - \mu_t(z)| \leq k^{-\frac{1}{2} + \delta}} \frac{|S_{\alpha}|^2_{h_t^k} }{\QQ_{h_t^k}(\alpha)}
} {\sum_{\alpha \in k P \cap \Z^m } \frac{|S_{\alpha}|^2_{h_t^k}
}{\QQ_{h_t^k}(\alpha)} }\right|
+ O(k^{-M}) \\
& \leq  & \ C k^{-\frac{1}{2} + \delta},
\end{eqnarray*}
proving (1).  In this estimate, we use the Localization Lemma
\ref{LOCALIZATION} and the upper/lower bound Lemma \ref{UBLB} on
$\rcal_k$.

Regarding $\frac{\partial}{\partial r_n}$ derivatives in (2), the
only change to the argument is in summing only $\alpha$ with
$\alpha_n \not= 0$ and then changing $\alpha \to \alpha - (0,
\dots, 1_n, \dots, 0)$ as explained in Proposition
\ref{ONESZEGODERr}. Clearly the localization and the estimates
only change by $\frac{1}{k}$.

We now consider the $\partial_t$ derivative. By Proposition
\ref{ONESZEGODER}, we have

\begin{equation} \begin{array}{lll}
&&\frac{1}{k} \frac{\partial} {\partial t}\log \sum_{\alpha \in k P
\cap \Z^m} \rcal_k(t, \alpha)
\frac{|S_{\alpha}(z)|^2_{h_t^k}}{\QQ_{h_t^k}(\alpha)}\\
& = &
\frac{1}{k} \frac{\sum_{\alpha} \rcal_k(t, \alpha)
\partial_t \log  \left( \frac{\rcal_k(t,
\alpha)}{\QQ_{h_t^k}(\alpha)} \right) \frac{e^{\langle\alpha, \rho
\rangle}}{\QQ^k_{t}(\alpha)}}{\left(\sum_{\alpha} \rcal_k(t,
\alpha)  \frac{e^{\langle\alpha, \rho
\rangle}}{\QQ_{h_t^k}(\alpha)}\right)} - \;
\frac{\partial}{\partial t} \phi_t  \\ && \\
& = & \frac{1}{k} \frac{\sum_{\alpha} \rcal_k(t, \alpha)
\partial_t \log  \left( \frac{\rcal_k(t,
\alpha)}{\QQ_{h_t^k}(\alpha)} \right) \frac{e^{\langle\alpha, \rho
\rangle}}{\QQ_{h_t^k}(\alpha)}}{\left(\sum_{\alpha} \rcal_k(t,
\alpha) \frac{e^{\langle\alpha, \rho
\rangle}}{\QQ_{h_t^k}(\alpha)}\right)} - \frac{1}{k}
\frac{\sum_{\alpha}
\partial_t \log  \left(\frac{1}{\QQ_{h_t^k}(\alpha)} \right) \frac{e^{\langle\alpha, \rho
\rangle}}{\QQ_{h_t^k}(\alpha)}}{\left(\sum_{\alpha}
\frac{e^{\langle\alpha, \rho
\rangle}}{\QQ_{h_t^k}(\alpha)}\right)}   + O(k^{-1}) \\
&=& \frac{1}{k} \frac{\sum_{\alpha} \rcal_k(t, \alpha)
\partial_t \log \rcal_k(t,
\alpha) \frac{e^{\langle\alpha, \rho
\rangle}}{\QQ_{h_t^k}(\alpha)}}{\left(\sum_{\alpha} \rcal_k(t,
\alpha) \frac{e^{\langle\alpha, \rho
\rangle}}{\QQ_{h_t^k}(\alpha)}\right)}
+ \frac{1}{k} \left(
\frac{\sum_{\alpha}
\partial_t \log  \QQ_{h_t^k}(\alpha) \frac{e^{\langle\alpha, \rho
\rangle}}{\QQ_{h_t^k}(\alpha)}}{\left(\sum_{\alpha}
\frac{e^{\langle\alpha, \rho
\rangle}}{\QQ_{h_t^k}(\alpha)}\right)}   -  \frac{\sum_{\alpha} \rcal_k(t, \alpha)
\partial_t \log  \QQ_{h_t^k}(\alpha) \frac{e^{\langle\alpha, \rho
\rangle}}{\QQ_{h_t^k}(\alpha)}}{\left(\sum_{\alpha}
\rcal_k(t, \alpha)\frac{e^{\langle\alpha, \rho
\rangle}}{\QQ_{h_t^k}(\alpha)}\right)}   \right)+ O(k^{-1})
 \end{array}
\end{equation}

Notice that $\QQ_{h_t^k} = \rcal_k(t, \alpha) (\QQ_{h_0^k}(\alpha))^{1-t} (\QQ_{h_1^k}(\alpha))^t$ and so $\partial_t  \log \QQ_k(t, \alpha)  = \partial_t \log \rcal_k (t, \alpha) + \log \left( \frac{\QQ_{h_1^k}(t, \alpha)}{\QQ_{h_0^k}(t, \alpha) } \right) $. It follows easily from the fact proved in Lemma \ref{MAINLEM}
(or more precisely the simpler Lemma \ref{BOUNDEDRCALDER})  that
$  \rcal_k(t, \alpha) = O(1)   $  and $\partial_t \log \left( \rcal_k(t, \alpha) \right) = O(1)$.   Also $\log \frac{\QQ_{h_1^k}(t, \alpha)}{\QQ_{h_0^k}(t, \alpha) } = O(k)$ uniformly in $\alpha$. Replacing $\rcal_k$ by $\rcal_\infty$ plus an error of order $k^{-\frac{1}{3}}$,

$$\frac{1}{k} \frac{\partial} {\partial t}\log \sum_{\alpha \in k P
\cap \Z^m} \rcal_k(t, \alpha)
\frac{|S_{\alpha}(z)|^2_{h_t^k}}{\QQ_{h_t^k}(\alpha)} = O( k^{ - \frac{1}{3}} ). $$

\end{proof}

\section{$C^2$-convergence  \label{C2}}

We now consider second derivatives in $\rho, t$. Again we must
separately consider derivatives in the interior and near the
boundary.  The following Proposition completes the proof of
Theorem \ref{SUM}.

\begin{prop}  Uniformly for $(t, z) \in [0, 1] \times M$, we have,
for any $\delta > 0$,
\begin{enumerate}

\item   $\frac{1}{k}\left|\frac{\partial^2}{\partial \rho_i
\rho_j}\log \sum_{\alpha \in k P \cap \Z^m}  \rcal_k(t, \alpha)
\frac{|S_{\alpha}(z)|^2_{h_t^k}}{\QQ_{h_t^k}(\alpha)} \right| =
O(k^{-\frac{1}{3} + 2 \delta})$;

\item   $ \frac{1}{k}\left|\frac{\partial^2}{\partial t
\partial \rho_j}\log \sum_{\alpha \in k P \cap \Z^m} \rcal_k(t, \alpha)
\frac{|S_{\alpha}(z)|^2_{h_t^k}}{\QQ_{h_t^k}(\alpha)}  \right| =
O(k^{-\frac{1}{3} + 2\delta})$;

\item   $ \frac{1}{k}\left|\frac{\partial^2}{\partial t^2}\log
\sum_{\alpha \in k P \cap \Z^m} \rcal_k(t, \alpha)
\frac{|S_{\alpha}(z)|^2_{h_t^k}}{\QQ_{h_t^k}(\alpha)}  \right| =
O(k^{-\frac{1}{3} + 2\delta})$

\item   The same estimates are valid if we replace
$\frac{\partial}{\partial r_n}$ in directions near $\dcal$ as in
Proposition \ref{ONESZEGODERr}.

\end{enumerate}

\end{prop}

We break up the proof into the four  cases. To simplify the
exposition, we introduce some new notation for localizing sums
over lattice points. By the Localization Lemma \ref{LOCALIZATION},
sums over lattice points can be localized to a ball of radius
$O(k^{-\frac{1}{2} + \delta})$ around $\mu_t(z)$. We emphasize
that although there are three metrics at play, it is the metric
$h_t$ along the Monge-Amp\`ere geodesic   that is used to localize
the sum. We introduce a notation for localized sums over pairs of
lattice points: let
\begin{equation} \widetilde{\Sigma}_{\alpha, \beta} \;\; F(\alpha, \beta): = \sum_{ |\frac{\alpha}{k} - \mu_t(z)|,
|\frac{\beta}{k} - \mu_t(z)| \leq k^{-\frac{1}{2} + \delta} }
F(\alpha, \beta).
\end{equation}

\subsection{\label{SECONDSPACE}Second space derivatives in the interior}

In this section we prove case (1). We have,

\begin{eqnarray} \label{3}
&&\frac{1}{k}\left|\frac{\partial ^2} {\partial
\rho_i\partial\rho_j}\log \sum_{\alpha \in k P \cap \Z^m}
\rcal_k(t, \alpha)
\frac{|S_{\alpha}(z)|^2_{h_t^k}}{\QQ_{h_t^k}(\alpha)} \right|\\
&=&\frac{1}{k}\left| \frac{\frac{1}{2} \sum_{\alpha, \beta}(\alpha
-\beta)^2\;\;  \rcal_k(t, \alpha) \rcal_k(t, \beta)
\frac{e^{\langle\alpha, \rho
\rangle}}{\QQ_{h_t^k}(\alpha)}\frac{e^{\langle \beta, \rho
\rangle}}{\QQ_{h_t^k}(\beta)} }{\left(\sum_{\alpha} \rcal_k(t,
\alpha) \frac{e^{\langle\alpha, \rho
\rangle}}{\QQ_{h_t^k}(\alpha)}\right)^2}
-k \; \frac{\partial^2}{\partial \rho_i\partial \rho_j} \phi_t \right|\nonumber \\
&\equiv& \frac{1}{k} \left| \frac{\frac{1}{2}\sum_{\alpha,
\beta}(\alpha -\beta)^2\;\; \rcal_k(t, \alpha) \rcal_k(t, \beta)
\frac{e^{\langle\alpha, \rho
\rangle}}{\QQ_{h_t^k}(\alpha)}\frac{e^{\langle \beta, \rho
\rangle}}{\QQ_{h_t^k}(\beta)} }{\left(\sum_{\alpha} \rcal_k(t,
\alpha) \frac{e^{\langle\alpha, \rho
\rangle}}{\QQ_{h_t^k}(\alpha)}\right)^2} - \frac{\frac{1}{2}
\sum_{\alpha, \beta}(\alpha -\beta)^2\;\; \frac{e^{\langle\alpha,
\rho \rangle}}{\QQ_{h_t^k}(\alpha)}\frac{e^{\langle \beta, \rho
\rangle}}{\QQ_{h_t^k}(\beta)} }{\left(\sum_{\alpha}
\frac{e^{\langle\alpha, \rho
\rangle}}{\QQ_{h_t^k}(\alpha)}\right)^2} \right|, \nonumber
\end{eqnarray}
modulo $O(\frac{1}{k})$ by Proposition \ref{COMPAREPIT}. We also
completed the square and used that the sum over $\alpha$ is a
probability measure to replace $\alpha^2 - \alpha \beta$ by
$\frac{1}{2} (\alpha - \beta)^2$. We also use Lemma
\ref{TWOSZEGODER} to write $\frac{\partial^2}{\partial
\rho_i\partial \rho_j} \phi_t$ as a sum over lattice points.

By the
Localization Lemma \ref{LOCALIZATION}, each sum over lattice
points can be localized to a ball of radius $O(k^{-\frac{1}{2} +
\delta})$ around $\mu_t(z)$.  Then, by Lemma \ref{MAINLEM} each
occurrence of $\rcal_k(t, \alpha)$ or $\rcal_k(t, \beta)$
 may be replaced by $\rcal_{\infty}(t,
\frac{\alpha}{k} )$ plus an error of order   $k^{-\frac{1}{3}}$.
Since $\frac{1}{k} (\alpha - \beta)^2 = O(k^{ 2 \delta})$ the
total error is of order $k^{2 \delta - \frac{1}{3}}$.  Since
$\delta$ is arbitrarily small,   this term is
decaying.  Further, after replacing $\rcal_k(t, \beta)$
  by $\rcal_{\infty}(t,
\frac{\alpha}{k} )$ we may  then replace $\frac{\alpha}{k},
\frac{\beta}{k}$ by $\mu_t(z)$ at the expense of another error of
order $k^{-\frac{1}{2} + \delta}$.  By modifying (\ref{3})
accordingly, we have
\begin{eqnarray}\label{2SPACE}
&&\frac{1}{k}\left|\frac{\partial^2} {\partial
\rho_i\partial\rho_j}\log \sum_{\alpha \in k P \cap \Z^m}
\rcal_k(t, \alpha)
\frac{|S_{\alpha}(z)|^2_{h_t^k}}{\QQ_{h_t^k}(\alpha)} \right| +
O(k^{-\frac{1}{3} + 2 \delta} )\\
&\equiv& \frac{1}{k} \left|
\frac{\frac{1}{2}\widetilde{\sum}_{\alpha, \beta}(\alpha
-\beta)^2\;\; \rcal_{\infty}(t, \mu_t(e^{\rho/2}))^2
\frac{e^{\langle\alpha, \rho
\rangle}}{\QQ_{h_t^k}(\alpha)}\frac{e^{\langle \beta, \rho
\rangle}}{\QQ_{h_t^k}(\beta)} }{\left(\sum_{\alpha}
\rcal_\infty(t, \mu_t(e^{\rho/2}))  \frac{e^{\langle\alpha, \rho
\rangle}}{\QQ_{h_t^k}(\alpha)}\right)^2} -
\frac{\frac{1}{2}\widetilde{\sum}_{\alpha, \beta}(\alpha
-\beta)^2\;\; \frac{e^{\langle\alpha, \rho
\rangle}}{\QQ_{h_t^k}(\alpha)}\frac{e^{\langle \beta, \rho
\rangle}}{\QQ_{h_t^k}(\beta)} }{\left(\sum_{\alpha}
\frac{e^{\langle\alpha, \rho
\rangle}}{\QQ_{h_t^k}(\alpha)}\right)^2} \right| \equiv 0,
\nonumber
\end{eqnarray}
where $\equiv$ means that the lines agree modulo errors of order
$O(k^{-\frac{1}{3} + 2 \delta} )$. In the last estimate, we use
that $ \rcal_{\infty}(t, \mu_t(e^{\rho/2}))^2$ cancels out in the
first term. This completes the proof in the spatial interior
case.

The modifications when $z$ is close to $\partial P$ are just as in
the case of the first derivatives.

\subsection{Mixed space-time derivatives} The mixed space-time
derivative is given by

\begin{eqnarray*}
&&\frac{1}{k}\left|\frac{\partial^2} {\partial \rho_i\partial
t}\log \sum_{\alpha \in k P \cap \Z^m} \rcal_k(t, \alpha)
\frac{|S_{\alpha}(z)|^2_{h_t^k}}{\QQ_{h_t^k}(\alpha)} \right|\\
&=&\frac{1}{k}\left| \frac{1}{2}\frac{\sum_{\alpha, \beta}(\alpha
-\beta)\;\; \rcal_k(t, \beta)  \rcal_k(t, \alpha) \partial_t \log
\left( \frac{\rcal_k(t, \alpha)\QQ_{h_t^k}(\beta)}{\rcal_k(t, \beta)\QQ_{h_t^k}(\alpha)} \right)
\frac{e^{\langle\alpha, \rho
\rangle}}{\QQ_{h_t^k}(\alpha)}\frac{e^{\langle \beta, \rho
\rangle}}{\QQ_{h_t^k}(\beta)} }{\left(\sum_{\alpha} \rcal_k(t,
\alpha) \frac{e^{\langle\alpha, \rho
\rangle}}{\QQ_{h_t^k}(\alpha)}\right)^2}
-k \; \frac{\partial^2}{\partial \rho_i\partial t} \phi_t \right|\\
\end{eqnarray*}

 It suffices to prove that
\begin{eqnarray*}
\frac{1}{k} \left| \frac{\sum_{\alpha, \beta} (\alpha -\beta)\;\;
\partial_t \log \left( \rcal_k(t, \alpha) \right)
\frac{e^{\langle\alpha, \rho
\rangle}}{\QQ_{h_t^k}(\alpha)}\frac{e^{\langle \beta, \rho
\rangle}}{\QQ_{h_t^k}(\beta)}}{\left(\sum_{\alpha}
\frac{e^{\langle\alpha, \rho \rangle}}{\QQ_{h_t^k}(\alpha)}
\right)^2} \right| = O(k^{-\frac{1}{2} + \delta})
\end{eqnarray*}
and

\begin{eqnarray*}\frac{1}{k}\left| \frac{1}{2}\frac{\sum_{\alpha, \beta}(\alpha
-\beta)\;\; \rcal_k(t, \beta)  \rcal_k(t, \alpha) \partial_t \log
\left( \frac{ \QQ_{h_t^k}(\beta)}{\QQ_{h_t^k}(\alpha)} \right)
\frac{e^{\langle\alpha, \rho
\rangle}}{\QQ_{h_t^k}(\alpha)}\frac{e^{\langle \beta, \rho
\rangle}}{\QQ_{h_t^k}(\beta)} }{\left(\sum_{\alpha} \rcal_k(t,
\alpha) \frac{e^{\langle\alpha, \rho
\rangle}}{\QQ_{h_t^k}(\alpha)}\right)^2}
-k \; \frac{\partial^2}{\partial \rho_i\partial t} \phi_t \right| = O(k^{-\frac{1}{3} + 2\delta}).\\
\end{eqnarray*}

The first estimate follows by the Localization Lemma \ref{LOCALIZATION}  and from  Lemma \ref{BOUNDEDRCALDER}, i.e.,   that
$\partial_t \log \left( \rcal_k(t, \alpha) \right) = O(1)$. The
second estimate is very similar to  that in \S \ref{SECONDSPACE},
specifically in (\ref{2SPACE}),  so we do not write it out in
full. In outline, we first apply the Localization Lemma and
replace each $\rcal_k(t, \alpha)$ by $\rcal_{\infty}(\mu_t(z))$
with $z = e^{\rho/2}$.  The errors in making these replacements
are of order $k^{-1/3 + \delta}$ because $\partial_t \log \left(
\frac{ \QQ_{h_t^k}(\beta)}{\QQ_{h_t^k}(\alpha)} \right) =  O( k|
u_t(\alpha)-u_t(\beta)|)= O(k^{\frac{1}{2}+\delta}) $ and because
$(\alpha -\beta) = O(k^{\frac{1}{2}+\delta}) $ in the localized
sum. We then express $\frac{\partial^2}{\partial \rho_i\partial t}
\phi_t $ in terms of the \szego kernel, i.e., as a sum over lattice
points, using Proposition \ref{TWOSZEGODER}, and cancel the
$\frac{\partial^2}{\partial \rho_i\partial t} \phi_t$ term. The
sum of the remainders is then of order  $k^{-1/3 + \delta}$,
completing the proof in this mixed case.

\subsection{Second time derivatives}

The proof in this case follows the same pattern, although the
estimates are somewhat more involved. The main steps are to
localize the sums over lattice points, to replace each  $\rcal_k$
by $\rcal_{\infty}$, then to cancel out $\rcal_{\infty}$ after all
replacements,  and to see that the resulting lattice point sum
cancels $k \; \frac{\partial^2}{\partial  t^2} \phi_t$. The
complications are only due to the number of estimates that are
required to justify the replacements.

 The second time derivative
equals

\begin{equation}\label{STD}\begin{array}{l}
\frac{1}{k}\frac{\partial^2} {\partial t^2}\log \sum_{\alpha \in k
P \cap \Z^m} \rcal_k(t, \alpha)
\frac{|S_{\alpha}(z)|^2_{h_t^k}}{\QQ_{h_t^k}(\alpha)} \\ \\
=\frac{1}{k} \frac{\sum_{\alpha, \beta}   \rcal_k(t, \beta)
\rcal_k(t, \alpha)\left(\partial_t \log \frac{\rcal_k(t,
\alpha)}{\QQ_{h_t^k}(\alpha)}  (\frac{\rcal_k(t,
\beta)}{\QQ_{h_t^k}(\beta)})^{-1} \right)^2
\frac{e^{\langle\alpha, \rho
\rangle}}{\QQ_{h_t^k}(\alpha)}\frac{e^{\langle \beta, \rho
\rangle}}{\QQ_{h_t^k}(\beta)} }{\left(\sum_{\alpha} \rcal_k(t,
\alpha) \frac{e^{\langle\alpha, \rho
\rangle}}{\QQ_{h_t^k}(\alpha)}\right)^2} \\ \\
+ \frac{1}{k}\left( \frac{\sum_{\alpha, \beta}   \rcal_k(t, \beta)
\rcal_k(t, \alpha)
\partial_t^2 \log  \left( \frac{\rcal_k(t,
\alpha)}{\QQ_{h_t^k}(\alpha)} \right) \frac{e^{\langle\alpha, \rho
\rangle}}{\QQ_{h_t^k}(\alpha)}\frac{e^{\langle \beta, \rho
\rangle}}{\QQ_{h_t^k}(\beta)} }{\left(\sum_{\alpha} \rcal_k(t,
\alpha) \frac{e^{\langle\alpha, \rho
\rangle}}{\QQ_{h_t^k}(\alpha)}\right)^2} -k \;
\frac{\partial^2}{\partial  t^2} \phi_t\right)
\end{array}\end{equation}
Here, we have simplified the numerator of the first term by
replacing  $$\left(\partial_t \log  \left( \frac{\rcal_k(t,
\alpha)}{\QQ_{h_t^k}(\alpha)} (\frac{\rcal_k(t,
\beta)}{\QQ_{h_t^k}(\beta)})^{-1} \right) \right) \left(\partial_t
\log \frac{\rcal_k(t, \alpha)}{\QQ_{h_t^k}(\alpha)} \right) \to
\frac{1}{2}  \left(\partial_t \log \left( \frac{\rcal_k(t,
\alpha)}{\QQ_{h_t^k}(\alpha)} (\frac{\rcal_k(t,
\beta)}{\QQ_{h_t^k}(\beta)})^{-1}\right) \right)^2, $$ which is
valid since the expression is anti-symmetric in $(\alpha, \beta)$
and since we are summing in $(\alpha, \beta)$.

To simplify the notation, we now abbreviate $\rcal(\alpha) =
\rcal_k(t, \alpha), \;\; \tcal(\alpha) =
\frac{1}{\QQ_{h_t^k}(\alpha)}$, $f' = \frac{\partial f}{\partial
t}$,  and write $(\ref{STD}) = \frac{N}{D}$ where the numerator
has the schematic form
\begin{equation} \label{TTD2} \begin{array}{lll} N \\= \sum_{\alpha, \beta} \left(
\left(\frac{\rcal'}{\rcal} (\alpha) + \frac{\tcal'}{\tcal}
(\alpha) \right)' + \frac{1}{2} \left(\frac{\rcal'}{\rcal}(\alpha)
+ \frac{\tcal'}{\tcal} (\alpha)- (\frac{\rcal'}{\rcal} (\beta)+
\frac{\tcal'}{\tcal} (\beta) )\right)^2 \right) \rcal(\alpha)
\tcal(\alpha) \rcal(\beta) \tcal(\beta) e^{\langle \alpha, \rho\rangle} e^{\langle\beta, \rho\rangle}
 \end{array} \end{equation} and where the
denominator is $D = \left(\sum_{\alpha} \rcal(\alpha)
\tcal(\alpha) \right)^2$. We omit the factors $
\frac{e^{\langle\alpha, \rho
\rangle}}{\QQ_{h_t^k}(\alpha)}\frac{e^{\langle \beta, \rho
\rangle}}{\QQ_{h_t^k}(\beta)} $ from the notation since they are
always present.

We now  compare $N$ and $D$ to the corresponding expressions in
the  second time derivative of the \szego kernel in Proposition
\ref{TWOSZEGODER}. In the latter case, $\rcal \equiv 1$ so any
terms with $t$-derivatives of $\rcal$ above do not occur in the
third comparison expression of Proposition \ref{TWOSZEGODER}.
Terms with no $t$ derivatives of $\rcal$ will be precisely as in
the comparison except that $\rcal$ is replaced by $1$. So we
consider the sub-sum of $N$, \begin{equation} \label{N1}
\begin{array}{lll} N_1 &=& \sum_{\alpha, \beta} \left( \left(
\frac{\tcal'}{\tcal} (\alpha) \right)' + \frac{1}{2} \left(
\frac{\tcal'}{\tcal} (\alpha)- \frac{\tcal'}{\tcal} (\beta)
)\right)^2 \right) \rcal(\alpha) \tcal(\alpha) \rcal(\beta)
\tcal(\beta)
 \end{array} \end{equation}
 If we now replace all occurrences of $\rcal_k(t, \alpha)$ by
$\rcal_{\infty}(\mu_t(z))$ in both numerator and denominator we
get the \szego kernel expression (the third comparison expression
of Proposition \ref{TWOSZEGODER}) of order $\frac{1}{k^2}$. (This
is verified in more detail at the end of the proof).  So
 we are left with estimating two remainder terms: First,
the difference $N_1 - \tilde{N_1}$ where $\tilde{N_1} $ is a sum
of terms in which we replace at least  one $\rcal(\alpha)$ by $
\rcal_{\infty}(\mu_t(z))$ (or with $\beta$). Second, we must
estimate $N - N_1$.

We first consider  $N_1 - \tilde{N_1}$. It arises by substituting
at least one     $\rcal(\alpha) -
\rcal_{\infty}(\mu_t(z)) = O(k^{-\frac{1}{3} })$  for one of the
$\rcal(\alpha)$'s in $N_1$. We   apply the localization argument
Lemma \ref{LOCALIZATION}  to replace $N_1$ (and $D$) by sums over
$\frac{\alpha}{k}, \frac{\beta}{k} \in B(\mu_t(z), k^{-\frac{1}{2}
+ \delta}).$ We thus need to estimate the following expression,
when at least one $\rcal(\alpha)$ is replaced by  $\rcal(\alpha) -
\rcal_{\infty}(\mu_t(z))$:
\begin{eqnarray*}
&&\frac{1}{k} \frac{\widetilde{\sum}_{\alpha, \beta} \rcal_k(t,
\beta) \rcal_k(t, \alpha)\left(\partial_t \log
\frac{\QQ_{h_t^k}(\beta) }{\QQ_{h_t^k}(\alpha)} \right) \left(-
\partial_t \log \QQ_{h_t^k}(\alpha) \right)
\frac{e^{\langle\alpha, \rho
\rangle}}{\QQ_{h_t^k}(\alpha)}\frac{e^{\langle \beta, \rho
\rangle}}{\QQ_{h_t^k}(\beta)} }{\left(\sum_{\alpha \in B(\mu_t(z),
k^{-\frac{1}{2} + \delta}) } \rcal_k(t, \alpha)
\frac{e^{\langle\alpha, \rho
\rangle}}{\QQ_{h_t^k}(\alpha)}\right)^2}
\\
\\
&+&\frac{1}{k}\left( \frac{\widetilde{\sum}_{\alpha, \beta}
\rcal_k(t, \beta) \rcal_k(t, \alpha)
\partial_t^2 \log  \left( \frac{1}{\QQ_{h_t^k}(\alpha)} \right) \frac{e^{\langle\alpha,
\rho \rangle}}{\QQ_{h_t^k}(\alpha)}\frac{e^{\langle \beta, \rho
\rangle}}{\QQ_{h_t^k}(\beta)} }{\left(\sum_{\alpha \in B(\mu_t(z),
k^{-\frac{1}{2} + \delta})} \rcal_k(t, \alpha)
\frac{e^{\langle\alpha, \rho
\rangle}}{(\QQ_{h_t^k}(\alpha))}\right)^2}
-k \; \frac{\partial^2}{\partial t^2} \phi_t\right) \\
\end{eqnarray*}

Due to the factor $\frac{1}{k}$ outside the sum, it suffices to
prove that
$$\left( \left( \frac{\tcal'}{\tcal} (\alpha) \right)' +
\frac{1}{2} \left( \frac{\tcal'}{\tcal} (\alpha)-
\frac{\tcal'}{\tcal} (\beta) \right)^2 \right) = O(k^{1 + 2
\delta}).$$ By Proposition \ref{QPINV}, we have
$$\frac{\tcal'}{\tcal} = - \frac{\pcal'}{\pcal} + k
u_t'(\frac{\alpha}{k}). $$ Since $u_t = (1 - t) u_0 + t u_1$, we
have
$$\frac{\tcal'}{\tcal}
(\alpha) = - \frac{\pcal'}{\pcal}+ k(u_1 -
u_0)(\frac{\alpha}{k}) = - \frac{\pcal'}{\pcal} + k(f_1 -
f_0)(\frac{\alpha}{k}) , $$ where we recall from \S \ref{KPSP}
that $u_{\phi} = u_0 + f_{\phi}$ with $f_{\phi}$ smooth up to the
boundary of $P$.

It follows that, \begin{equation} \label{DIFFT}
\frac{\tcal'}{\tcal} (\alpha)- \frac{\tcal'}{\tcal} (\beta) = -
\frac{\pcal'}{\pcal}(\alpha) + \frac{\pcal'}{\pcal}(\beta) + k(f_1
- f_0)(\frac{\alpha}{k}) - k(f_1 - f_0)(\frac{\beta}{k}),
\end{equation}
\begin{equation}
(\frac{\tcal'}{\tcal}
(\alpha))' = - (\frac{\pcal'}{\pcal})'= O(1).
\end{equation}
with
$$ k (f_1
- f_0)(\frac{\alpha}{k}) - k(f_1 - f_0)(\frac{\beta}{k}) = k
O(|\frac{\alpha}{k} - \frac{\beta}{k}|) \; = \; O(k^{\frac{1}{2} +
\delta}).$$ Further, by Lemma \ref{RATIOCANCEL} (using Lemma
\ref{MAINPCALD}), the factors of
$$\frac{(\frac{\partial}{\partial t}) \pcal_{h_t^k}(\alpha)}{\pcal_{h_t^k}(\alpha)}  \; = \;
  \frac{\left(S_1(t, \alpha, k)  +
R_k(\frac{\alpha}{k}, h) \right)}{S_0(t, \alpha, k)} = O(1), $$
and similarly $ (\frac{\pcal'}{\pcal})' = O(1)$. Since
(\ref{DIFFT})  is squared, it has terms as large as  $O(k^{1 + 2
\delta})$. Taking into account the overall factor of $\frac{1}{k}$
and the presence of  at least one factor of size $k^{-\frac{1}{3}}$ coming from the replacement of at least one $\rcal_k(t,
\alpha)$ by $\rcal_{\infty}(\mu_t(z))$, we see that $N_1 -
\tilde{N_1}$ has order  $k^{-\frac{1}{3} + 2 \delta}$ and again this decays
for sufficiently small $\delta$.

Now we estimate $N - N_1$, which consists of  terms with at least
one $t$-derivative of $\rcal$. By Lemma \ref{RATIOCANCEL}, the
terms with no $t$ derivatives on $\tcal$ give the terms
\begin{eqnarray*}
&&\frac{1}{k} \frac{\widetilde{\sum}_{\alpha, \beta} \rcal_k(t,
\beta) \rcal_k(t, \alpha)\left(\partial_t \log \frac{\rcal_k(t,
\alpha)}{\rcal_k(t, \beta)} \right) \left(\partial_t \log
\rcal_k(t, \alpha)\right) \frac{e^{\langle\alpha, \rho
\rangle}}{\QQ_{h_t^k}(\alpha)}\frac{e^{\langle \beta, \rho
\rangle}}{\QQ_{h_t^k}(\beta)} }{\left(\sum_{\alpha} \rcal_k(t,
\alpha) \frac{e^{\langle\alpha, \rho
\rangle}}{(\QQ_{h_t^k}(\alpha))}\right)^2}
\\
\\
&+&\frac{1}{k} \frac{\widetilde{\sum}_{\alpha, \beta}   \rcal_k(t,
\beta) \rcal_k(t, \alpha)
\partial_t^2 \log \left( \rcal_k(t, \alpha) \right)
\frac{e^{\langle\alpha, \rho
\rangle}}{\QQ_{h_t^k}(\alpha)}\frac{e^{\langle \beta, \rho
\rangle}}{\QQ_{h_t^k}(\beta)} }{\left(\sum_{\alpha} \rcal_k(t,
\alpha) \frac{e^{\langle\alpha, \rho
\rangle}}{(\QQ_{h_t^k}(\alpha))}\right)^2} = O(k^{-1}),
\end{eqnarray*}
  by Lemma \ref{MAINLEM}.

  This leaves us with the terms
  $$(\frac{\rcal'}{\rcal}(\alpha)
- \frac{\rcal'}{\rcal} (\beta))(\frac{\tcal'}{\tcal} (\alpha) -
\frac{\tcal'}{\tcal} (\beta)).   $$ Again by Lemma
\ref{BOUNDEDRCALDER},  the first term is $O(1)$ while the second
factor is (\ref{DIFFT}) and has size $k k^{-\frac{1}{2} +
\delta}$.  Here, we again use Propositions \ref{QPINV} and
\ref{MAINPCALD}. Due to the overall factor of $\frac{1}{k}$ this
term has size $k^{-\frac{1}{2} + \delta}$.

Therefore, as stated above,  up to errors of order
$k^{-1/3 + \delta}$, (\ref{STD}) is simplified to $- \;
\frac{\partial^2}{\partial t^2} \phi_t$ plus
\begin{equation} \frac{1}{k}\left( \frac{\widetilde{\sum}_{\alpha, \beta}
\rcal_{\infty}(\mu_t(e^{\rho/2}))
\rcal_{\infty}(\mu_t(e^{\rho/2}))\left(
\partial_t^2 \log  \left( \frac{1}{\QQ_{h_t^k}(\alpha)} \right) + (\partial_t \log
\frac{1}{\QQ_{h_t^k}}) (\partial_t \log  ( \frac{\QQ_{h_t^k}
(\beta)}{\QQ_{h_t^k} (\alpha)}) \right) \frac{e^{\langle\alpha,
\rho \rangle}}{\QQ_{h_t^k}(\alpha)}\frac{e^{\langle \beta, \rho
\rangle}}{\QQ_{h_t^k}(\beta)} }{\left(\sum_{\alpha \in B(\mu_t(z),
k^{-\frac{1}{2} + \delta})}
\rcal_{\infty}(\mu_t(e^{\rho/2}))\frac{e^{\langle\alpha, \rho
\rangle}}{(\QQ_{h_t^k}(\alpha))}\right)^2} \right) \end{equation}
As before, we  cancel the factors of
$\rcal_{\infty}(\mu_t(e^{\rho/2}))$. The resulting difference then
cancels to order $k^{-1/2 + \delta}$ by Lemma \ref{TWOSZEGODER}
(3).

 This completes the
proof of the second time derivative estimate, and hence of the
main theorem.

\end{document}